\documentclass[english]{article}
\usepackage[T1]{fontenc}
\usepackage{geometry}
\usepackage{amsmath}
\usepackage{amsthm}
\usepackage{amssymb}
\usepackage{stackrel}
\usepackage{graphicx}

\makeatletter

\providecommand{\tabularnewline}{\\}

\newcommand{\lyxaddress}[1]{
	\par {\raggedright #1
	\vspace{1.4em}
	\noindent\par}
}
\theoremstyle{plain}
\newtheorem{thm}{\protect\theoremname}
\theoremstyle{plain}
\newtheorem{lem}[thm]{\protect\lemmaname}
\ifx\proof\undefined
\newenvironment{proof}[1][\protect\proofname]{\par
	\normalfont\topsep6\p@\@plus6\p@\relax
	\trivlist
	\itemindent\parindent
	\item[\hskip\labelsep\scshape #1]\ignorespaces
}{%
	\endtrivlist\@endpefalse
}
\providecommand{\proofname}{Proof}
\fi
\theoremstyle{plain}
\newtheorem{prop}[thm]{\protect\propositionname}
\theoremstyle{definition}
\newtheorem{defn}[thm]{\protect\definitionname}

\renewcommand{\vec}[1]{\boldsymbol #1}

\usepackage{algorithmic}
\usepackage{algorithm}

\renewcommand\[{\begin{equation}}
\renewcommand\]{\end{equation}}

\@ifundefined{showcaptionsetup}{}{%
 \PassOptionsToPackage{caption=false}{subfig}}
\usepackage{subfig}
\makeatother

\usepackage{babel}
\providecommand{\definitionname}{Definition}
\providecommand{\lemmaname}{Lemma}
\providecommand{\propositionname}{Proposition}
\providecommand{\theoremname}{Theorem}

\begin{document}
\title{Compact Feature-Aware Hermite-Style \\
High-Order Surface Reconstruction}
\author{Yipeng Li$^{1}$, Xinglin Zhao$^{1}$\thanks{Current address: Google Inc., Mountain View, CA, USA.},
Navamita Ray$^{2}$, and Xiangmin Jiao$^{1}$\thanks{Corresponding author. Email: xiangmin.jiao@stonybrook.edu,}}
\date{}
\maketitle

\lyxaddress{\begin{center}
$^{1}$Department of Applied Mathematics \& Statistics and Institute
for Advanced Computational Science, Stony Brook University, USA\\
$^{2}$Computer, Computational \& Statistical Sciences, Los Alamos
National Laboratory, Los Alamos, NM, USA
\par\end{center}}
\begin{abstract}
High-order surface reconstruction is an important technique for CAD-free,
mesh-based geometric and physical modeling, and for high-order numerical
methods for solving partial differential equations (PDEs) in engineering
applications. In this paper, we introduce a novel method for accurate
and robust reconstructions of piecewise smooth surfaces from a triangulated
surface. Our proposed method extends the Continuous Moving Frames
(CMF) and the Weighted Averaging of Local Fittings (WALF) methods
(\emph{Engrg. Comput.} 28 (2012)) in two main aspects. First, it utilizes
a Hermite-style least squares approximation to achieve fourth and
higher-order accuracy with compact support, even if the input mesh
is relatively coarse. Second, it introduces an iterative feature-aware
parameterization to ensure high-order accurate, $G^{0}$ continuous
reconstructions near sharp features. We present the theoretical framework
of the method and compare it against CMF and WALF in terms of accuracy
and stability. We also demonstrate that the use of the Hermite-style
reconstruction in the solutions of PDEs using finite element methods
(FEM), and show that quartic and sextic FEMs using the high-order
reconstructed surfaces produce nearly identical results as using exact
geometry while providing additional flexibility.
\end{abstract}

\section{Introduction}

Surface meshes and their manipulations are critical for geometric
modeling, meshing, numerical simulations, and other related problems.
Some example problems that involve manipulating surface meshes include
mesh generation and mesh enhancement for high-order finite element
methods \cite{FG00MGA}, mesh smoothing in ALE methods \cite{DHPR04ALE},
mesh adaptation in moving boundary problems \cite{JCNH09AMA}, and
geometric modeling and meshing in computer graphics \cite{Fleishman2005,Walton1996}.
In these settings, it is often critical to have a high-order accurate
representation of the surface to support the mesh manipulations. The
high-order surfaces are also important for modern numerical discretization
methods, such as quadratic, quartic, and even sextic finite element
methods and higher-order spectral element methods, which have become
increasingly in recent years. However, a continuous CAD model is often
not available. Instead, only a surface mesh, typically with piecewise
linear or bilinear faces, is available.

In this paper, we consider the problem of reconstructing a high-order
accurate, continuous surface from a surface triangulation. We refer
to this problem as \emph{high-order surface reconstruction} (or simply
\emph{high-order reconstruction}). By ``high order,'' we mean that
the method should be able to deliver more than second-order accuracy,
and preferably fourth or higher order, compared to just first or second
order accuracy of traditional techniques. The high-order accuracy
is important for accurate treatment of boundary conditions along curved
boundaries, especially with high-order finite element methods.

In \cite{Jiao2011RHO}, four requirements were posed for high-order
reconstruction:
\begin{description}
\item [{Geometric~accuracy:}] The reconstruction should be accurate and
asymptotically convergent to the exact surface to high order under
mesh refinement.
\item [{Continuity:}] The reconstructed surface should be continuous to
some degree (e.g., $G^{0}$, $G^{1}$, or $G^{2}$ continuous, depending
on applications).
\item [{Feature~preservation:}] The reconstruction should preserve sharp
features (such as ridges and corners) in the geometry.
\item [{Numerical~stability:}] The reconstruction should be numerically
stable and must not be oscillatory under noise.
\end{description}
Different applications may have different emphases on the requirements.
For example, in computer graphics and geometric design, the visual
effect may be the ultimate goal, so smoothness and feature preservation
may be more important. Our focus in this paper is on numerical solutions
of partial differential equations (PDEs), for which the numerical
accuracy and stability are critical. Indeed, if the geometry is low-order
accurate, the numerical solutions of a high-order finite element method
will most likely be limited to low-order accuracy. An unstable surface
approximation with excessive oscillations can have even more devastating
effect on the numerical simulations. Some efforts, such as isogeometric
analysis \cite{Hughes_isogeo}, aim to improve accuracy by using continuous
CAD models directly in numerical simulations, but CAD models are often
not available, such as in moving boundary problems, or they may be
impractical to be used directly, such as on large-scale supercomputers.
When only a mesh is given, a high-order reconstruction may sometimes
be the best option, in that they provide additional flexibility and
can deliver the same accuracy as using exact geometry.

Two methods, called \emph{Continuous Moving Frames} (\emph{CMF}) and
\emph{Weighted Averaging of Local Fittings} (\emph{WALF}), were proposed
in \cite{Jiao2011RHO}, for reconstructing a piecewise smooth surfaces.
Both methods were based on weighted least squares (\emph{WSL}) approximations
using local polynomial fittings with the assumption that the vertices
of the mesh accurately sample the surface, and the connectivity of
the mesh correctly represent the topology of the surface. These methods
can achieve fourth- and even higher order accuracy, and WALF guarantees
$G^{0}$ continuity for smooth surfaces. In contrast, most other methods
could achieve only first- or second-order accuracy. Between CMF and
WALF, the former tends to be more accurate whereas the latter tends
to be more efficient. However, both CMF and WALF had two key limitations.
First, if the input mesh is relatively coarse, then there may be a
lack of points in the stencil, so the methods are forced to use low-order
approximations to avoid instability \cite{Jiao2011RHO}. This loss
of accuracy tends to happy more pronounced near boundaries or sharp
features (i.e., ridges and corners), where the stencils tend to be
one-sided. Second, $G^{0}$ continuity may be lost near sharp features,
so the reconstructed surface may not be ``watertight.'' As a result,
when used to generate high-order finite element meshes, the parameterization
of the surface elements incident on sharp features may suffer from
a loss of smoothness and in turn loss of precision in the numerical
solutions.

In this paper, we address these limitations of CMF and WALF to achieve
high-order accuracy and $G^{0}$ continuity near sharp features. To
this end, we extend these methods in two main aspects. First, we introduce
a Hermite-style weighted least squares approximation, to take into
account both point locations as well as normals in surface reconstruction.
We assume that the point coordinates and normals are high-order accurate,
which, for example, may be obtained from the original CAD models or
obtained from solutions of differential equations. A key advantage
of this Hermite-style reconstruction is that it allows much more compact
stencils, so high-order accuracy can be achieved even with relatively
coarse input meshes. Second, we introduce an iterative feature-aware
parameterization for constructing high-order surface elements, so
that these surface elements can define a $G^{0}$ continuous surface
with smooth parameterizations and uniform high-order accuracy.

Using both point and normal for numerical approximation on discrete
surfaces is not a new idea. It is analogous to Hermite interpolation
in numerical analysis. For surface modeling, Walton \cite{Walton1996}
defined an approach to reconstruct $G^{1}$ continuous surfaces. Another
example is the curved PN-triangles \cite{VPBM01CPN}. Other related
work includes \cite{GI04NCO}, in which points and normals are used
together for estimating curvatures. However, to the best of our knowledge,
our proposed technique is among the first that leverage both points
and normals to deliver guaranteed high-order accuracy in surface reconstructions
with sharp features. As a result, it provides a valuable alternative
to the traditional CAD models, such as NURBS \cite{Far93} and T-splines
\cite{Sederberg03TSplines} for engineering applications. In particular,
it provides a flexible approach for high-order methods for solving
PDEs, for which we will demonstrate that high-order reconstructed
surfaces enable nearly identical PDE solutions as using exact geometry
during mesh generation.

Besides the Hermite-style reconstruction, another contribution of
this work is to identify the significance of the smoothness of parameterizations
near features. Traditionally, smoothness in surface reconstruction
had focused on high-order (e.g., $G^{1}$, $G^{2}$, and even $G^{\infty}$)
continuity. For example, the reconstruction in \cite{Walton1996}
aimed for $G^{1}$ continuity, and the moving least squares (MLS)
\cite{LevinD.1998,Fleishman2005} aimed for $G^{\infty}$ continuity.
However, such high-order continuity does not have direct correlation
with the accuracy of the reconstructed surface or numerical PDEs.
In particular, the $G^{1}$ reconstruction of Walton \cite{Walton1996}
is no more accurate than the piecewise linear surface \cite{Jiao2011RHO}.
The MLS reconstruction was conjectured in \cite{LevinD.1998}, but
in practice it is non-convergent even for simple geometries such as
a torus \cite{Jiao2011RHO}. These higher-order continuities introduce
additional difficulties in resolving sharp features. In contrast,
our method aims for only $G^{0}$ continuity, but we emphasize the
smoothness of the parameterization of high-order surface elements
near sharp features, and we show that it has a direct correlation
with the accuracy of the high-order surface reconstruction and the
numerical solutions of FEM.

The remainder of the paper is organized as follows. Section~\ref{sec:Background}
presents some background knowledge, including local polynomial fittings
and the method of weighted least squares. Sections~\ref{sec:HWALF-surface}
and \ref{sec:HWALF-curve} introduce the algorithms of Hermite-style
weighted least squares for high-order surface and curve reconstructions,
respectively. Section~\ref{sec:Feature-Aware-HOSR} describes the
construction of a $G^{0}$ continuous surface using high-order parametric
elements with an iterative feature-aware parameterization. Section~\ref{sec:Numerical-Results}
assesses the proposed method in terms of geometric accuracy and feature
preservation, compares it against CMF and WALF, and demonstrates its
effectiveness as an alternative of exact geometry in PDE discretizations.
Section~\ref{sec:Conclusions} concludes the paper with some discussions
on the future work.

\section{\label{sec:Background}Background}

In this section, we review some basic concepts related to surfaces
and space curves, followed by a brief review of weighted least squares
(WLS) approximations and WLS-based surface reconstructions using CMF
and WALF. These techniques are the foundations of the feature-aware
Hermite-style reconstruction proposed in this paper.

\subsection{\label{subsec:surfaces-and-curves}Surfaces and Space Curves}

\subsubsection{\label{subsec:Smooth-Surfaces}Smooth Surfaces.}

Consider a smooth surface $\Gamma$ defined in the global $xyz$ coordinate
system. Given a point $\boldsymbol{x_{0}}=[x_{0},y_{0},z_{0}]^{T}$
on the surface (note that for convenience we treat points as column
vectors), let the origin of the local frame be at $\vec{x}_{0}$.
Let $\vec{m}_{0}$ be an approximate normal vector with unit length
at $\boldsymbol{x_{0}}$. Let $\boldsymbol{s}_{0}$ and $\boldsymbol{t}_{0}$
denote a pair of orthonormal basis vectors perpendicular to $\vec{m}_{0}$.
The vectors $\boldsymbol{s}_{0}$, $\boldsymbol{t}_{0}$, and $\vec{m}_{0}$
form the axes of a local $uvw$ coordinate system at $\boldsymbol{x}_{0}$.
Let $\vec{Q}_{0}$ be the matrix composed of column vectors $\boldsymbol{s}_{0}$,
$\boldsymbol{t}_{0}$, and $\vec{m}_{0}$, i.e., $\vec{Q}_{0}=\left[\boldsymbol{s}_{0},\,\boldsymbol{t}_{0},\,\vec{m}_{0}\right]$.
Any point $\boldsymbol{x}$ in the global coordinate system can be
then transformed to the point 
\begin{equation}
\vec{p}(\vec{u})=\left[u,v,w(\vec{u})\right]^{T}=\vec{Q}_{0}^{T}(\boldsymbol{x}-\boldsymbol{x}_{0})\label{eq:coordinate-transformation}
\end{equation}
in the local frame, where $\boldsymbol{u}=\left[u,v\right]^{T}$.
Since $\vec{m}_{0}$ is an approximate normal vector, $w(\vec{u})$
is expected to be one-to-one in a neighborhood of $\vec{x}_{0}$.
We refer to $f(\vec{u})=w(\vec{u})$ as the \emph{local height function}
about $\vec{x}_{0}$. This transformation is important for high-order
surface reconstruction, since it allows reducing the problem to high-order
approximations to the local height function. The vectors $\vec{p}_{u}$
and $\vec{p}_{v}$ are tangent to the surface in the local frame.
 Let $\ell=\left\Vert \vec{p}_{u}\times\vec{p}_{v}\right\Vert =\sqrt{1+\left\Vert \vec{\nabla}f\right\Vert ^{2}}$,
which is the \emph{area measure}. The \emph{unit normal} to the surface
in the local frame is then given by 
\begin{equation}
\hat{\vec{n}}=\frac{\vec{p}_{u}\times\vec{p}_{v}}{\ell}=\frac{1}{\ell}\left[\begin{array}{c}
-\vec{\nabla}f\\
1
\end{array}\right].\label{eq:normal}
\end{equation}
This connection between the normal $\hat{\vec{n}}$ and the gradient
of local height functions will be important for Hermite-style surface
reconstruction.

\subsubsection{\label{subsec:Space-Curves}Space Curves.}

For piecewise smooth surfaces, there can be ridge curves (or feature
curves). These curves are \emph{space curves}, in that they are embedded
in $\mathbb{R}^{3}$ and they may not be coplanar. Similarly, for
an open surface, its boundary curve is a space curve, and it can be
treated in the same fashion as feature curves.

Given a point $\vec{x}_{0}$ on a space curve $\gamma$ in $\mathbb{R}^{3}$,
let $\vec{s}_{0}$ be an approximate tangent vector of unit length.
Let $\vec{m}_{0}$ and $\vec{b}_{0}$ denote a pair of orthonormal
basis functions perpendicular to $\vec{s}_{0}$. The vectors $\vec{s}_{0}$,
$\vec{m}_{0}$, and $\vec{b}_{0}$ form the axes of a local $uvw$
coordinate system at $\vec{x}_{0}$. Let $\vec{Q}_{0}=\left[\vec{s}_{0},\,\vec{m}_{0},\,\vec{b}_{0}\right]$.
Then, any point $\vec{x}$ in the global coordinate system can be
transformed to
\begin{equation}
\vec{p}(u)=\left[u,v(u),w(u)\right]^{T}=\vec{Q}_{0}^{T}(\vec{x}-\vec{x}_{0})\label{eq:point-on-curve}
\end{equation}
in the local frame. We refer to $\vec{f}(u)=[v(u),w(u)]^{T}$ as a
\emph{vector-valued local height function}, of which each component
is one-to-one in a neighborhood of $\vec{x}_{0}$. The \emph{length
measure} is $\ell=\Vert\vec{p}(u)\Vert=\sqrt{1+\left\Vert \vec{f}'(u)\right\Vert ^{2}}$.
The \emph{unit tangent} to the curve in the local frame is then 
\begin{equation}
\hat{\vec{t}}=\frac{\vec{p}}{\ell}=\frac{1}{\ell}\left[\begin{array}{c}
1\\
\vec{f}'(u)
\end{array}\right],\label{eq:curve-tangent}
\end{equation}
which will be important for Hermite-style curve reconstruction.

\subsection{\label{subsec:Local-Polynomial-Fittings.}Local Weighted Least Squares
Fittings.}

Weighted least squares is a powerful method for constructing high-order
polynomial fitting of a smooth function. Let us first derive it for
a function $f(\vec{u}):\mathbb{R}^{2}\rightarrow\mathbb{R}$ at a
given point $\boldsymbol{u}_{0}=\left[0,0\right]^{T}$, where $f$
is the local height function in surface reconstruction. Suppose $f$
is smooth and its value is known only at a sample of $m$ points $\vec{u}_{i}$
near $\vec{u}_{0}$, where $1\leq i\leq m$. We refer to these points
as the \emph{stencil} for the fitting. The 2D Taylor series of $f(\vec{u})$\emph{
}about $\boldsymbol{u}_{0}$ is given by 
\begin{equation}
f(\boldsymbol{u})=\sum_{q=0}^{\infty}\sum_{j,k\ge0}^{j+k=q}c_{jk}u^{j}v^{k},\label{eq:FTaylorseries_cont}
\end{equation}
where $c_{jk}=\dfrac{1}{j!k!}\dfrac{\partial^{j+k}}{\partial u^{j}\partial v^{k}}f(\boldsymbol{0})$.
Suppose $f$ is continuously differentiable to $(p+1)$st order for
some $p>1$. $f(\boldsymbol{u})$ can be approximated to (\emph{$p+1$})st\emph{
}order accuracy about $\boldsymbol{u}_{0}$ as 
\begin{equation}
f(\boldsymbol{u})=\sum_{q=0}^{p}\sum_{j,k\ge0}^{j+k=q}c_{jk}u^{j}v^{k}+\mathcal{O}\left(\Vert\boldsymbol{u}\Vert^{p+1}\right).\label{eq:FTaylorseries_disc}
\end{equation}
The first term in \eqref{eq:FTaylorseries_disc} is the \emph{degree-$p$
Taylor polynomial} about the origin, which has $n=(p+1)(p+2)/2$ coefficients,
namely $c_{jk}$ with $0\leq j+k\leq p$. Assume $m\ge n$, and let
$f_{i}$ denote $f(\vec{u}_{i})$. We then obtain a system of $m$
equations 
\begin{equation}
\sum_{q=0}^{p}\sum_{j,k\ge0}^{j+k=q}c_{jk}u_{i}^{j}v_{i}^{k}\approx f_{i}\label{eq:1}
\end{equation}
for $1\leq i\leq m$. The equation can be written in matrix form as
$\vec{A}\vec{x}\approx\vec{b}$, where $\vec{A}$ is a \emph{generalized
Vandermonde matrix}, $\vec{x}$ is composed of $c_{jk}$, and $\vec{b}$
is composed of $f_{i}$. 

The generalized Vandermonde system obtained from \eqref{eq:1} is
rectangular. In general, it can be solved by posing as a minimization
of the weighted norm of the residual vector $\vec{r}=\vec{b}-\vec{A}\vec{x}$,
i.e.,
\begin{equation}
\min_{\vec{x}}\Vert\vec{r}\Vert_{\vec{W}}\equiv\min_{\vec{x}}\Vert\vec{W}(\vec{A}\vec{x}-\vec{b})\Vert_{2},\label{eq:weighted_norm}
\end{equation}
where $\vec{W}=\mbox{diag}\{\omega_{1},\omega_{2},\dots,\omega_{m}\}$
is diagonal and is referred to as the \emph{weighting matrix}. The
weights in $\vec{W}$ assign priorities to different rows in the generalized
Vandermonde system, where each row corresponds to a different point
in the stencil. If both $\vec{A}$ and $\vec{W}$ are nonsingular
matrices, then $\vec{W}$ has no effect on the solution. However,
if $m\neq n$ or $\vec{A}$ is singular, then different $\vec{W}$
would lead to different solutions. Let $\vec{T}\in\mathbb{R}^{n\times n}$
be a scaling matrix, and we arrive at the least squares problem 
\begin{equation}
\vec{W}\vec{A}\vec{T}\vec{y}\approx\vec{W}\vec{b},\label{eq:least-squares}
\end{equation}
and then $\vec{x}=\vec{T}\vec{y}$. In general, given a weighting
matrix $\vec{W}$, let $\vec{v}_{i}$ denote the $i$th column vector
of $\vec{W}\vec{A}$. We choose
\begin{equation}
\vec{T}=\text{diag}\{1/\Vert\vec{v}_{1}\Vert_{2},\dots,1/\Vert\vec{v}_{n}\Vert_{2}\},\label{eq:column-scaling-matrix}
\end{equation}
which approximately minimizes the condition number; see \cite[p. 265]{Golub13MC}
and \cite{VDS69CNE}. We refer to this as \emph{algebraic column scaling}.
Alternatively, let $h$ be a length measure of the stencil, and one
can scale the coordinates of $u_{i}$ and $v_{i}$ by dividing them
by $h$. We refer to this as \emph{geometric scaling}, which is equivalent
to the algebraic scaling with $\vec{T}$ equal to the diagonal matrix
composed of $h^{-j-k}$, where $j$ and $k$ are the corresponding
indices of $c_{jk}$ in \eqref{eq:1}.

The matrix $\vec{W}\vec{A}\vec{T}$ may still be ill-conditioned even
after scaling. For efficiency and robustness, \eqref{eq:least-squares}
can be solved using a truncated QR factorization with column pivoting
(QRCP). Specifically, let $\tilde{\vec{A}}=\vec{W}\vec{A}\vec{T}$.
The QRCP is
\[
\tilde{\vec{A}}\vec{P}=\vec{Q}\vec{R},
\]
where $\vec{Q}$ is $m\times n$ with orthonormal column vectors,
$\vec{R}$ is an $n\times n$ upper-triangular matrix, $\vec{P}$
is a permutation matrix, and the diagonal entries in $\vec{R}$ are
in descending order \cite{Golub13MC}. For ill-conditioned systems,
the condition number of $\vec{R}$ can be estimated incrementally;
if a column results in a large condition number, its corresponding
monomial should be truncated, so do the other monomials that contain
it as a factor. Let $\tilde{\vec{Q}}$ and $\tilde{\vec{R}}$ denote
the truncated matrices. The final solution of $\vec{x}$ is then given
by
\begin{equation}
\vec{x}=\vec{T}\vec{P}\tilde{\vec{R}}^{-1}\tilde{\vec{Q}}^{T}\vec{W}\vec{b},\label{eq:solution}
\end{equation}
where $\tilde{\vec{R}}^{-1}$ denotes a back substitution step.

The solution vector $\vec{x}$ contains the coefficients $c_{ij}$,
from which we obtain a polynomial $\tilde{f}(\vec{u})=\sum_{q=0}^{p}\sum_{j,k\ge0}^{j+k=q}c_{jk}u^{j}v^{k}$.
We refer to this approach for estimating the Taylor polynomial as\emph{
local WLS fitting. }Note that if $\vec{u}_{0}$ is a point in the
stencil, we can force $\tilde{f}$ to be interpolatory at $\vec{u}_{0}$,
i.e., $\tilde{f}(\boldsymbol{u}_{0})=f(\boldsymbol{u}_{0})$, by setting
$c_{00}=0$ and removing the equation corresponding to $\vec{u}_{0}$.
This reduces the problem to an $(m-1)\times(n-1)$ linear system,
and it tends to be more accurate if the function $f$ is known to
be accurate at $\vec{u}_{0}$.

\subsection{\label{subsec:WLS-Tri-Surf}Local Fittings on Triangulated Surfaces.}

The local WLS fittings can be used in local high-order reconstructions
of a triangulated surface, on which a feature curve is composed of
a collection of edges. This requires three key components: 1) construct
a local frame, 2) select a proper stencil about the vertex, and 3)
define the weighting scheme.

To define the local frame at a point $\vec{x}_{0}$ on a surface,
we need an approximate surface normal $\vec{m}_{0}$ at the point,
which can be obtained by averaging the normals to its adjacent faces.
Such an averaging is in general first-order accurate, which suffices
for this purpose. Similarly, for a curve, we need an approximate tangent
vector $\vec{s}_{0}$ at a point, which can be obtained by averaging
the tangents to its adjacent edges.

For the stencil selection, it is simple and efficient to use mesh
connectivity. For curves, as long as the points are distinct, it suffices
to make the number of points to be equal to the number of coefficients,
but $m>n$ may lead to better error cancellation for even-degree polynomials
with nearly symmetric stencils. For a triangular mesh, we define $\mathit{k}$-ring
neighborhoods, with half-ring increments, as follows:
\begin{itemize}
\item The\emph{ 1-ring neighbor faces} of a vertex $v$ are the faces incident
on $v$, and the \emph{1-ring neighbor vertices} are the vertices
of these faces. 
\item The\emph{ 1.5-ring neighbor faces} are the faces that share an edge
with a 1-ring neighbor face, and the \emph{1.5-ring neighbor vertices
}are the vertices of these faces. 
\item For an integer $k\geq1$, the\emph{ $(k+1)$-ring neighborhood} of
a vertex is the union of the 1-ring neighbors of its $k$-ring neighbor
vertices, and the\emph{ $(k+1.5)$-ring neighborhood} is the union
of the $1.5$-ring neighbors of the\emph{ $k$-}ring neighbor vertices.
\end{itemize}
Figure~\ref{fig:Stencil-for-point} illustrates this definition up
to 2.5 rings. In general, for degree-$p$ fitting, we use the $(p+1)/2$-ring
for accurate input. For a curve, the $\mathit{k}$-ring neighborhood
can be defined similarly for an integer $k$, and we use $\left\lceil (p+1)/2\right\rceil $-ring
for degree-$p$ fitting. We adaptively enlarge the ring size if there
are too few points or the input is relatively noisy. This approach
is efficient since it takes constant-time per vertex with a proper
data structure, such as the half-facet (or the half-edge) data structure
\cite{DREJTAHF2014}. However, if the mesh is poor shared, the stencil
may be highly skewed, which can be mitigated with a proper weighting
scheme.

\begin{figure}
\begin{centering}
\includegraphics[scale=0.6]{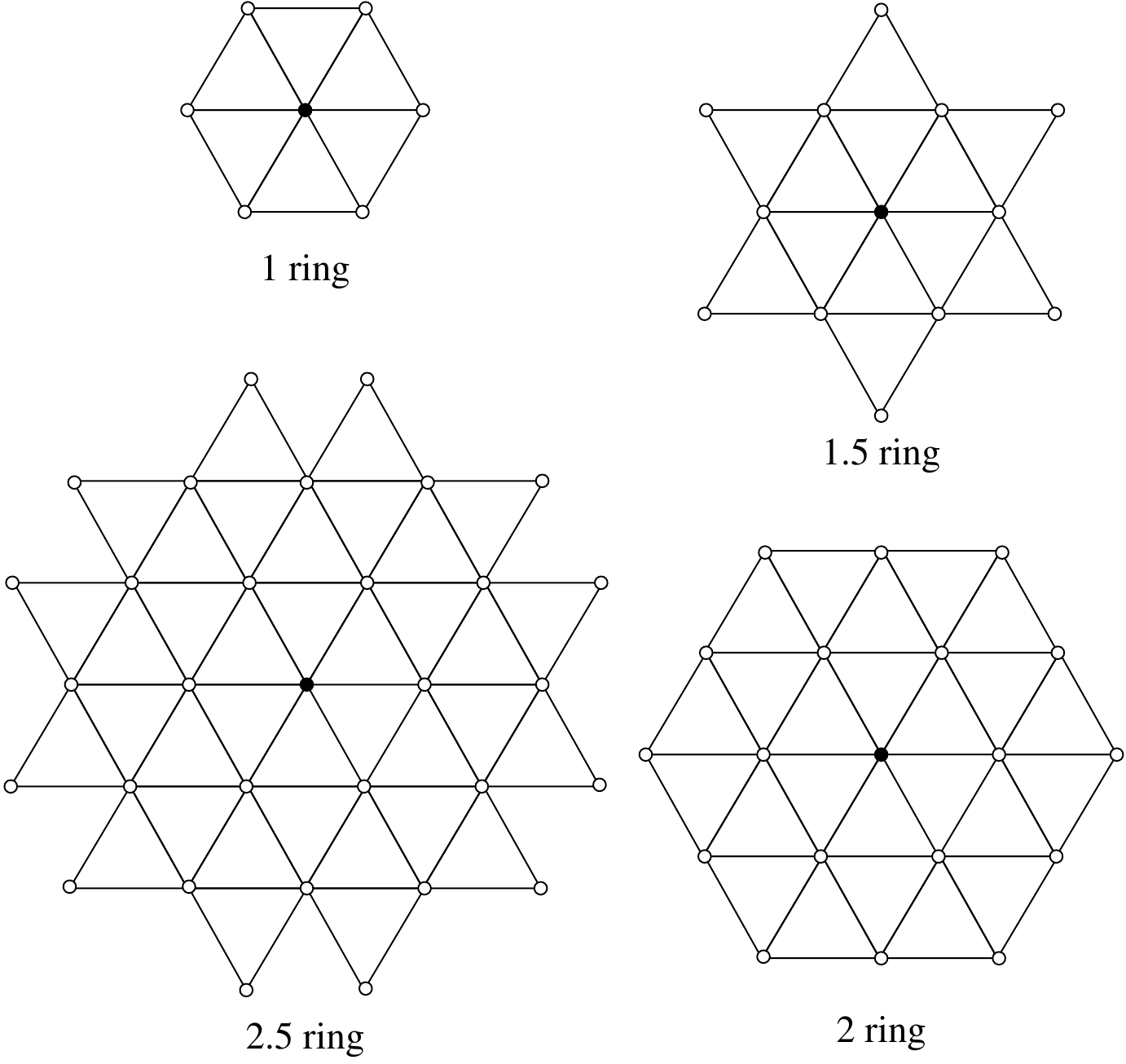}
\par\end{centering}
\caption{\label{fig:Stencil-for-point}Illustration of $k$-ring and $k.5$-ring
neighborhoods for stencil selections.}
\end{figure}

There are many options to define the weighting matrix $\vec{W}$ in
\eqref{eq:weighted_norm}. A commonly used weighting scheme is the
so-called \emph{inverse distance weighting} and its variants. The
standard inverse-distance weighting assigns $\omega_{i}=1/\Vert\vec{x}_{i}-\vec{x}_{0}\Vert^{q}$
to some $q$th power. This weighting scheme assigns smaller weights
for points that are farther away from the origin. However, the inverse
distance has a singularity if $\vec{x}_{i}$ is too close to $\vec{x}_{0}$.
This singularity can be resolved by safeguarding the denominator with
some small $\epsilon$. For coarse meshes or surfaces with sharp features,
it is desirable to use a small and even zero weight for $\vec{x}_{i}$
if its (approximate) normal $\vec{m}_{i}$ deviates too much from
$\vec{m}_{0}$. Let $\theta_{i}^{+}\equiv\textrm{max}(0,\vec{m}_{i}^{T}\vec{m}_{0})$.
We then arrive at the weight
\begin{equation}
\omega_{i}=\theta_{i}^{+}\left/\left(\sqrt{\Vert\vec{u}_{i}-\vec{u}_{0}\Vert^{2}+\epsilon}\right)^{q}\right.,\label{eq:normal-safeguarded-inverse-distance}
\end{equation}
where $q=p/2$ and $\epsilon=0.1$ in \cite{Jiao2008}. The factor
$\theta_{i}^{+}$ serves as a safeguard for discontinuous surface
or very coarse meshes. Similarly, given a piecewise linear curve,
let $\theta_{i}^{+}\equiv\textrm{max}(0,\vec{s}_{i}^{T}\vec{s}_{0})$,
where $\vec{s_{i}}$ denote the approximate unit tangent at $\vec{x}_{i}$,
and the same weighting scheme applies.

The inverse-distance-weighting scheme tends to give much higher weights
to points closest to $\vec{x}_{0}$, especially if $\epsilon$ is
close to zero. If $\vec{x}_{0}$ is not at a vertex, the vertices
closest to $\vec{x}_{0}$ tend to be highly asymmetric about $\vec{x}_{0}$.
In this case, it is desirable to use a weighting scheme that is flatter
about the origin while being smooth and compact. A class of such functions
is due to Wendland \cite{wendland1995piecewise}. We shall consider
three of these functions: 
\begin{align}
\psi_{3,1}(r) & =(1-r)_{+}^{4}(4r+1),\label{eq:Wendland_phi_3_1}\\
\psi_{4,2}(r) & =(1-r)_{+}^{6}(35r^{2}+18r+3),\label{eq:Wendland_phi_4_2}\\
\psi_{5,3}(r) & =(1-r)_{+}^{8}(32r^{3}+25r^{2}+8r+1),\label{eq:Wendland_phi_5_3}
\end{align}
where $(1-r)_{+}\equiv\max\{0,1-r\}$. In Figure~\ref{fig:Wendland-functions},
the left panel shows these functions, while the right panel shows
the scaled functions so that their maximum values are all ones. In
this paper, we will combine these Wendland functions with $\theta_{i}^{+}$
as weighting functions for both surfaces and curves; see Sections~\ref{subsec:weighting-scheme}
and \ref{subsec:H-CMF-curve} for more detail.

\begin{figure}
\begin{centering}
\includegraphics[width=0.45\textwidth]{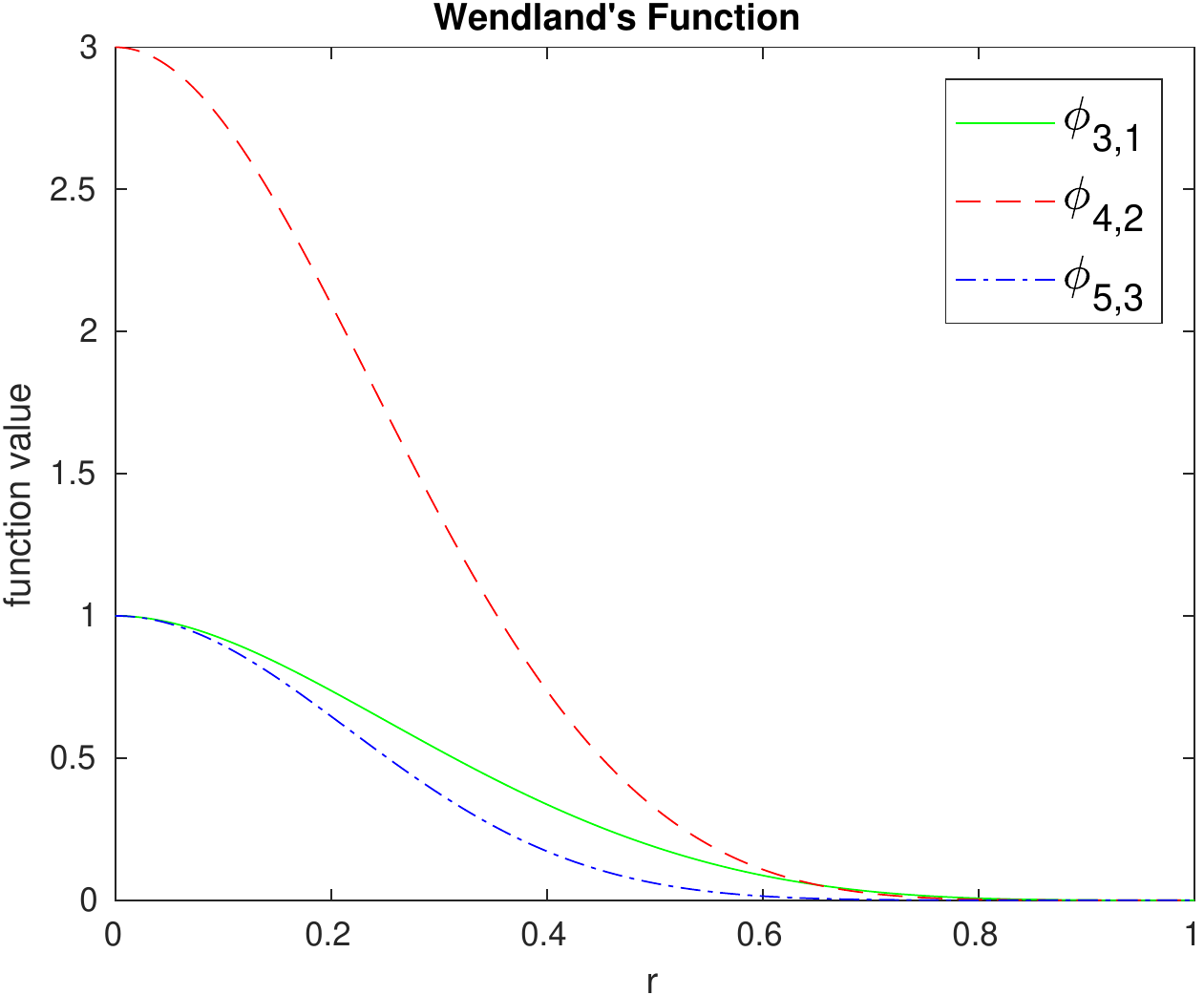}\hfill\includegraphics[width=0.45\textwidth]{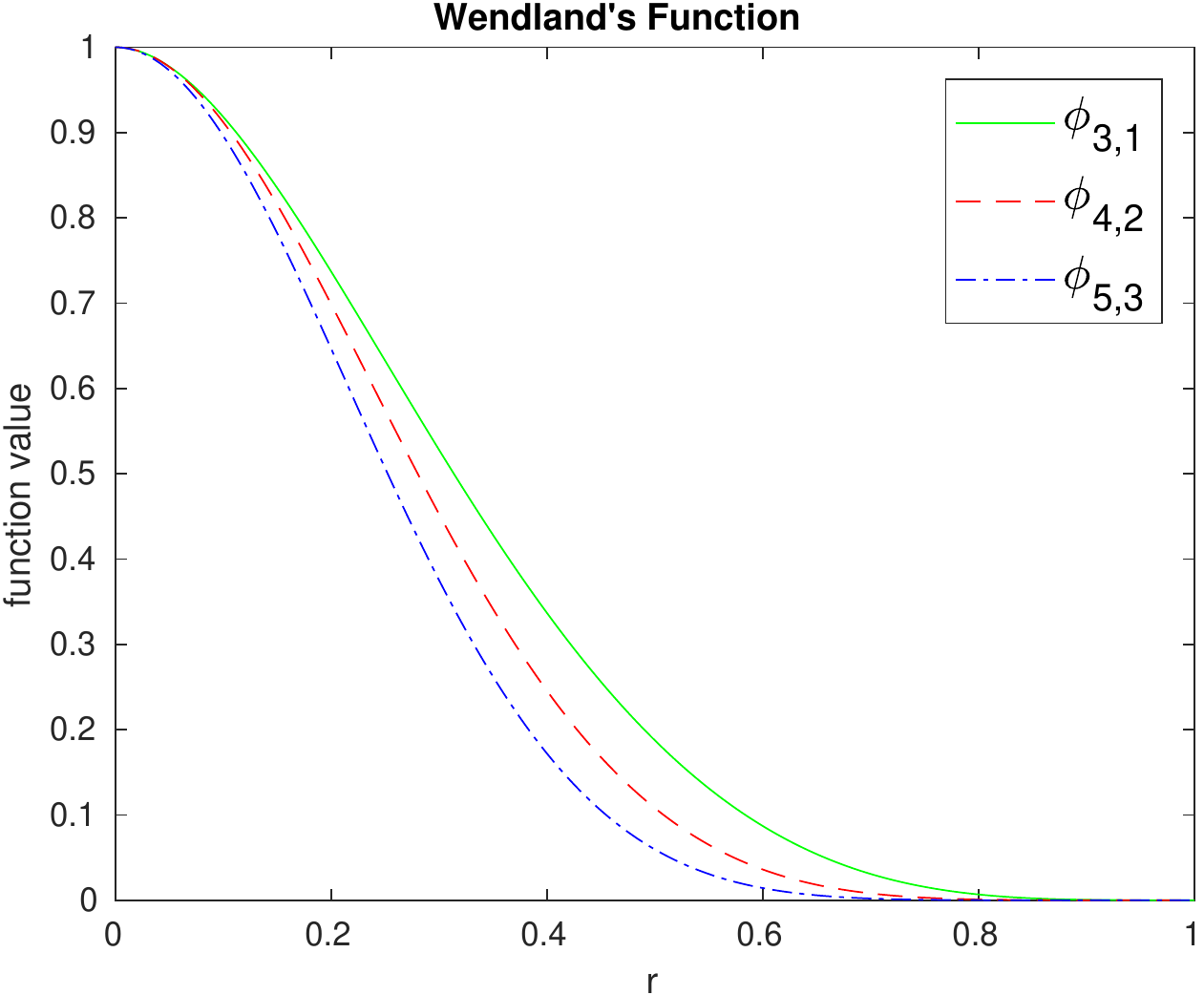}
\par\end{centering}
\caption{\label{fig:Wendland-functions}Wendland's functions before and after
scaling.}
\end{figure}

With these three components, we can apply local WLS fittings to construct
a local surface patch at an arbitrary point of a triangulated surface
or a piecewise linear curve. More specifically, consider a point $\vec{p}$
on a triangle $\vec{x}_{1}\vec{x}_{2}\vec{x}_{3}$. For each vertex
$\vec{x}_{j}$, we find its $k$-ring neighborhood $S_{j}$. If $\vec{p}$
is on the edge $\vec{x}_{j_{1}}\vec{x}_{j_{2}}$, we use $S_{j_{1}}\bigcup S_{j_{2}}$
as the stencil; if $\vec{x}_{0}$ is in the interior of the triangle,
we use $\bigcup_{j=1}^{3}S_{j}$ as the stencil. To build the local
frame, we take $\sum_{j=1}^{3}\xi_{j}\vec{m}_{j}$ as an approximate
normal, where $\vec{m}_{j}$ is the approximate normal at $\vec{x}_{j}$
and $\xi_{j}$ is the barycentric coordinates of $\vec{p}$ in the
triangle. This construction ensures the local frames change continuously
from point to point, and hence it is referred to as the \emph{Continuous
Moving Frames} (\emph{CMF}) method \cite{Jiao2011RHO}. If the input
vertices approximate a smooth surface $\Gamma$ with an error of $\mathcal{O}(h^{p+1})$,
it can be shown that the CMF reconstruction with degree-$p$ polynomials
can achieve $\mathcal{O}(h^{p+1})$ accuracy, where $h$ is proportional
to the radius of the stencil. For even-degree polynomials, the error
may be $\mathcal{O}(h^{p+2})$ for symmetric stencils due to error
cancellation. For this reason, it is in general advantageous to use
even-degree polynomials.

\subsection{\label{subsec:WALF}Weighted Averaging of Local Fittings}

The local WLS fittings and CMF do not necessarily produce a $G^{0}$
continuous surface. One approach to recover $G^{0}$ continuity is
\emph{Weighted Averaging of Local Fittings} (\emph{WALF}) \cite{Jiao2011RHO},
which computes a weighted average of the local fittings at the vertices,
where the weights are the barycentric coordinates. For example, consider
a triangle with vertices $\vec{x}_{j}$, $j=1,2,3$, and an arbitrary
point $\vec{p}$ in the triangle. For each vertex $\vec{x}_{j}$,
a point $\vec{q}_{j}$ is obtained for $\vec{p}$ from the corresponding
local fitting within its own local coordinate system. Let $\xi_{j}$,
$j=1,2,3$ denote the barycentric coordinates of $\vec{p}$ within
the triangle. Then, $\vec{q}=\sum_{j=1}^{3}\xi_{j}\vec{q}_{j}$ is
the WALF reconstruction for $\vec{p}$. A similar construction also
applies to curves. Figure~\ref{fig:weighted_averaging} shows a 2-D
illustration of this construction. For a smooth surface, WALF constructs
a $G^{0}$ continuous surface, due to the $C^{0}$ continuity of finite-element
interpolation. It can be shown that if the input vertices approximate
a smooth surface $\Gamma$ with an error of $\mathcal{O}(h^{p+1})$,
then WALF reconstruction with degree-$p$ polynomials can achieve
$\mathcal{O}(h^{p+1}+h^{6})$ accuracy for $p\leq6$ in terms of the
shortest distance to the true surface \cite{Jiao2011RHO}. 
\begin{figure}
\begin{centering}
\includegraphics[height=1.5in]{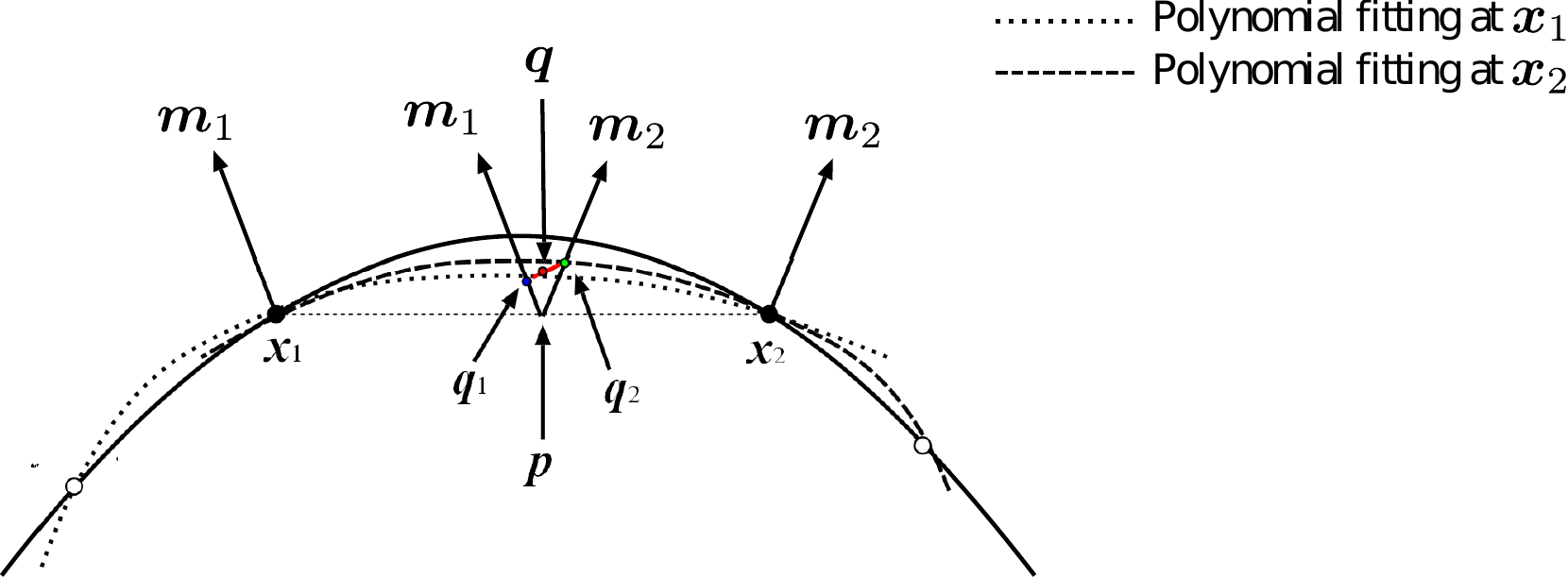}
\par\end{centering}
\caption{\label{fig:weighted_averaging}2-D illustration of WALF. The solid
curve indicates the exact curve. The dashed and dotted curves indicate
the fittings at $\vec{x}_{1}$ and $\vec{x}_{2}$, respectively. $\vec{q}$
is the WALF reconstruction for point $\vec{p}$, computed as a weighted
average of $\vec{q}_{1}$ and $\vec{q}_{2}$ from the fittings at
$\vec{x}_{1}$ and $\vec{x}_{2}$, respectively.}
\end{figure}

\subsection{\label{subsec:High-Order-Parametric-Elements}High-Order Parametric
Surface Elements}

Besides WALF, another approach to obtain a $G^{0}$ continuous surface
is to use high-degree piecewise polynomial interpolation, as in high-order
finite-element methods. Specifically, for each triangle in the input
mesh, one can construct a degree-$p$ surface patch from $n=(p+1)(p+2)/2$
points, including the three corner nodes in the original triangle,
along with additional \emph{mid-edge nodes }and\emph{ mid-face nodes}.
Let $\{\vec{\xi}_{i}\}$ denote the natural coordinates of $\vec{x}_{i}$
in the reference space, which is typically chose as the right triangle
with vertices $\vec{\xi}_{1}=[0,0]$, $\vec{\xi}_{2}=[1,0]$, and
$\vec{\xi}_{3}=[0,1]$. Let $\vec{x}_{i}$ denote the coordinate of
the $i$th node of the element. The degree-$p$ surface patch is then
defined by
\begin{equation}
\vec{x}(\vec{\xi})=\sum_{i=1}^{n}N_{i}(\vec{\xi})\vec{x}_{i},\label{eq:param-surf-element}
\end{equation}
where the $N_{i}(\vec{\xi})$ are the Lagrange polynomial basis of
degree-$p$ interpolation within the reference space, also know as
the \emph{shape function} of the degree-$p$ element. We refer to
such a degree-$p$ triangular patch as a \emph{parametric surface
element}. It is commonly used in defining the geometry in high-order
finite element methods, where $\vec{x}_{i}$ are sampling points on
the exact geometry if a CAD is given. In the context of high-order
reconstructions, the $\vec{x}_{i}$ corresponding to the mid-edge
and mid-face nodes can be obtained from degree-$p$ CMF or WALF \cite{Jiao2011RHO}.

A key question in constructing a parametric surface element is the
placement of the mid-edge and mid-face nodes. This includes two aspects:
the selection of $\vec{\xi}_{i}$ for the mid-edge and mid-face nodes,
and the placement of $\vec{x}_{i}$ based on $\vec{\xi}_{i}$. Traditionally,
the $\vec{\xi}_{i}$ are equally spaced in the reference space, as
illustrated in Figure~\ref{fig:iso-point} for degree-2, 4, and 6
triangles. For high-degree interpolation, such equally space points
lead to ill-conditioned Vandermonde matrices and hence unstable Lagrange
basis functions. For very high-degree interpolation, it is desirable
to use nonuniform nodes that resemble the distributions of the Chebyshev
points in 1-D for better stability. There are various choices for
such points; see e.g. \cite{rapetti2012generation}. Among them, the
Lebesgue-Gauss-Lobatto symmetric (LEBGLS) points, because they approximately
minimize the condition number of the interpolation (a.k.a., the Lebesgue
constant in the polynomial interpolation theory), and the points have
a three-way symmetry. Hence, they are well suited for high-order surface
reconstruction over triangular meshes. Figure~\ref{fig:leb-points}
shows the LEBGLS points for degree-2, 4, and 6 triangles. Note that
the degree-2 LEBGLS points are equally spaced, but those of higher-degrees
tend to be more clustered toward the edges and corners. Given the
points $\vec{\xi}_{i}$, the positioning of $\vec{x}_{i}$ requires
special attention, especially near features, so that the derivatives
of $\vec{x}(\vec{\xi})$ defined by \eqref{eq:param-surf-element}
are uniformally bounded up to order $p+1$. We will address it in
Section~\ref{sec:Feature-Aware-HOSR}.

\begin{figure}
\centering{}\subfloat[degree 2]{\raggedright{}\includegraphics[width=0.33\columnwidth]{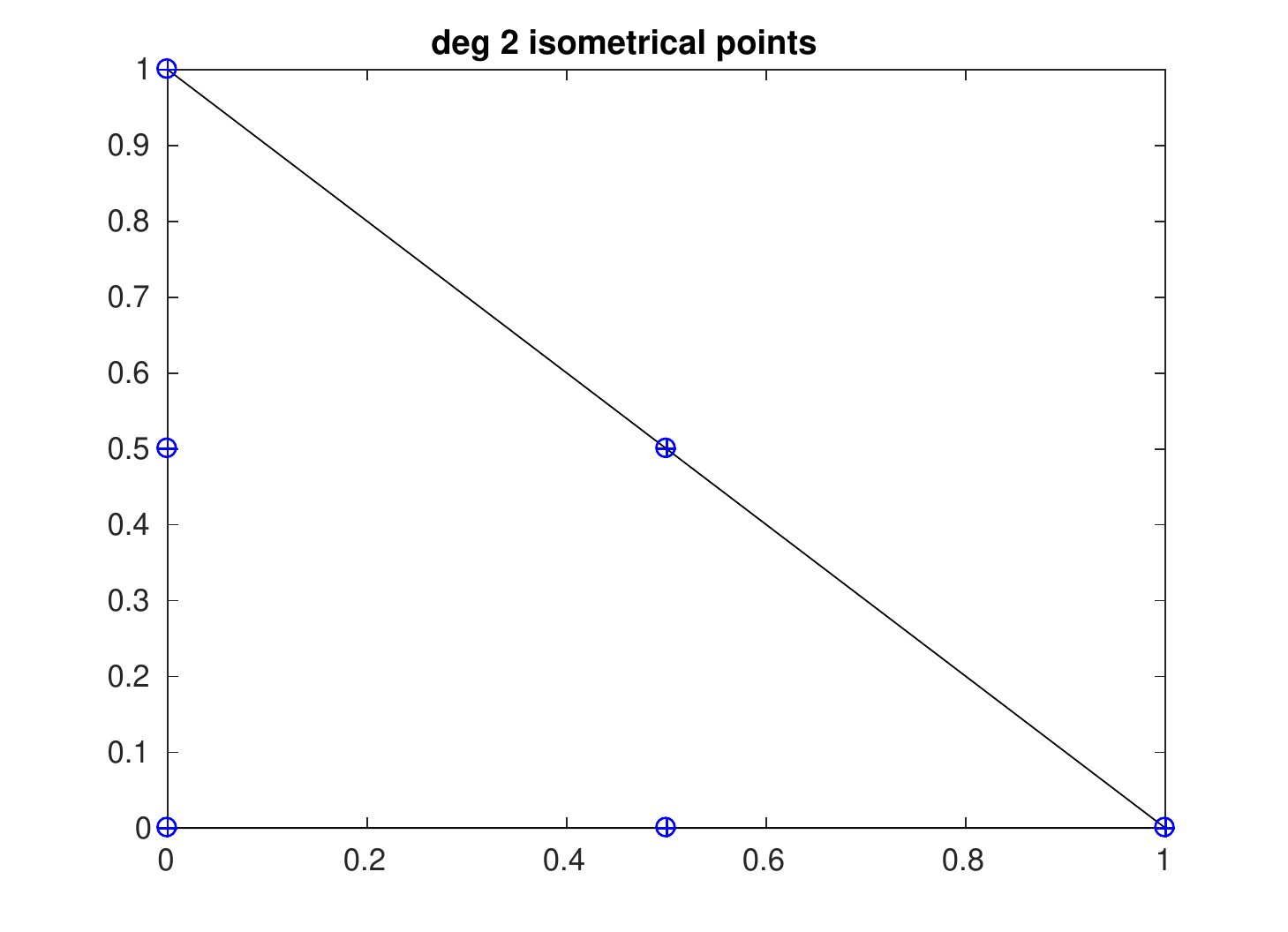}}\subfloat[degree 4]{\raggedright{}\includegraphics[width=0.33\columnwidth]{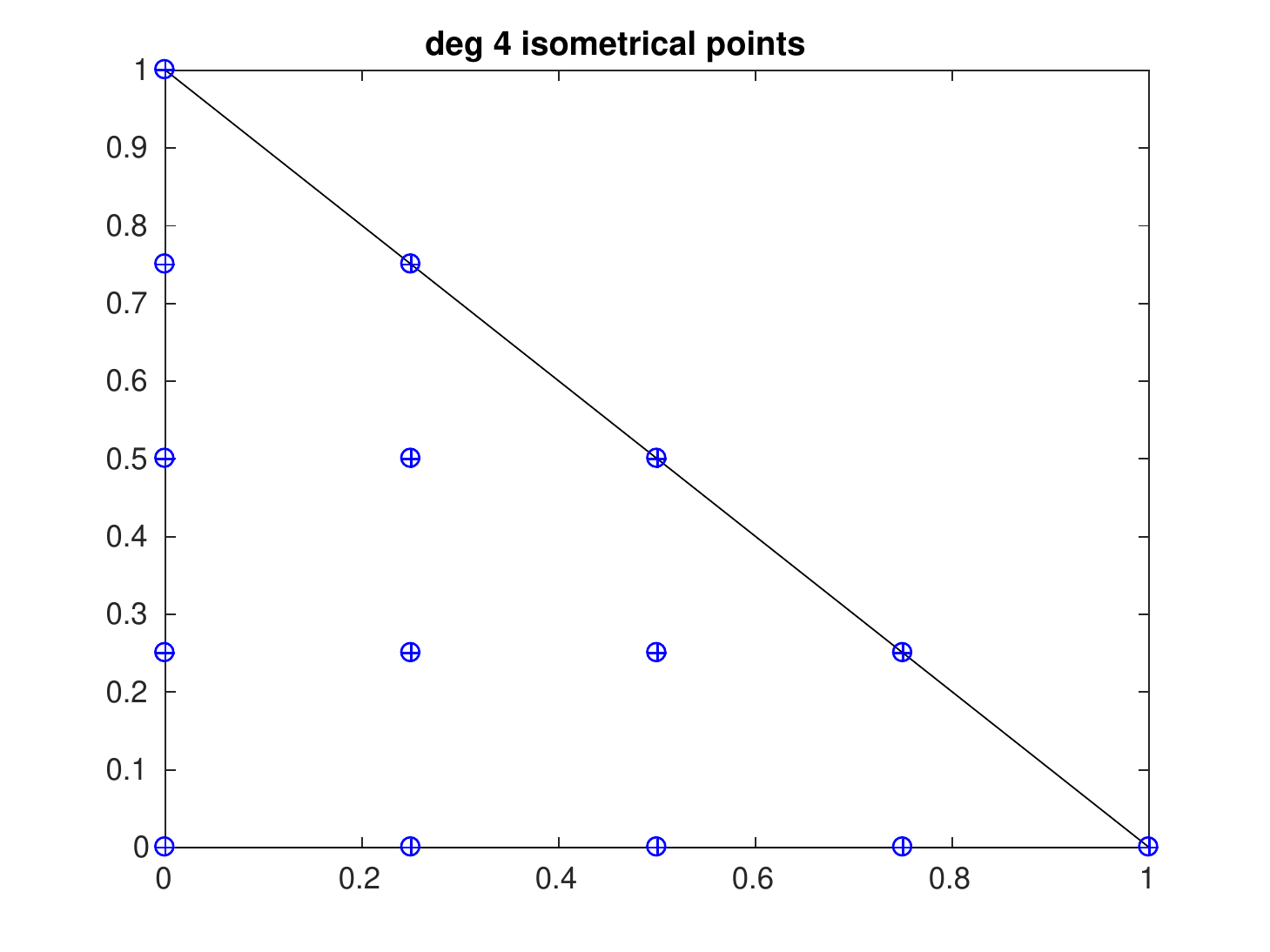}}\subfloat[degree 6]{\includegraphics[width=0.33\columnwidth]{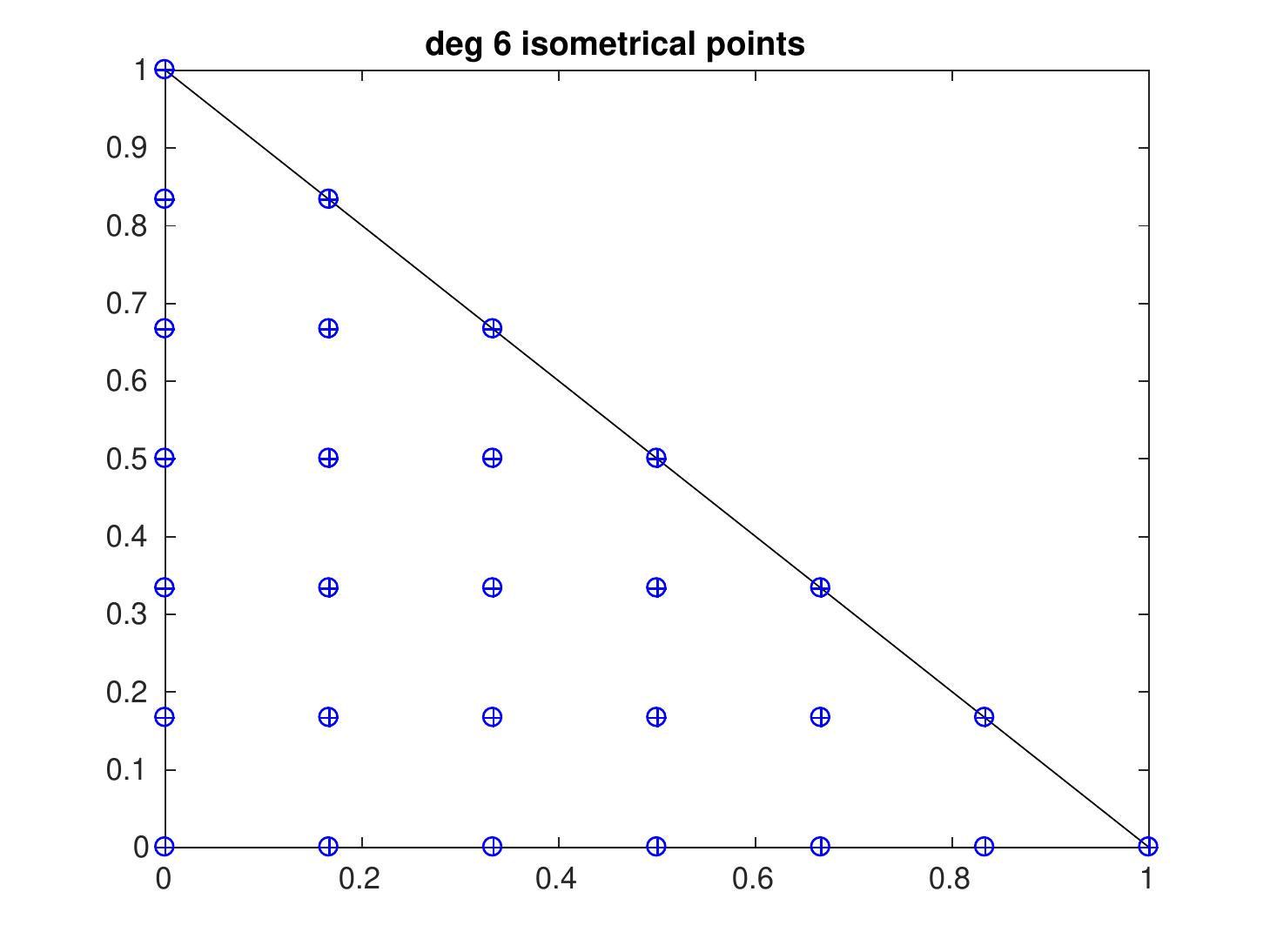}}\caption{\label{fig:iso-point}Parametric triangular elements with equally-spaced
points in the reference space.}
\end{figure}

\begin{figure}
\centering{}\subfloat[degree 2]{\includegraphics[width=0.33\textwidth]{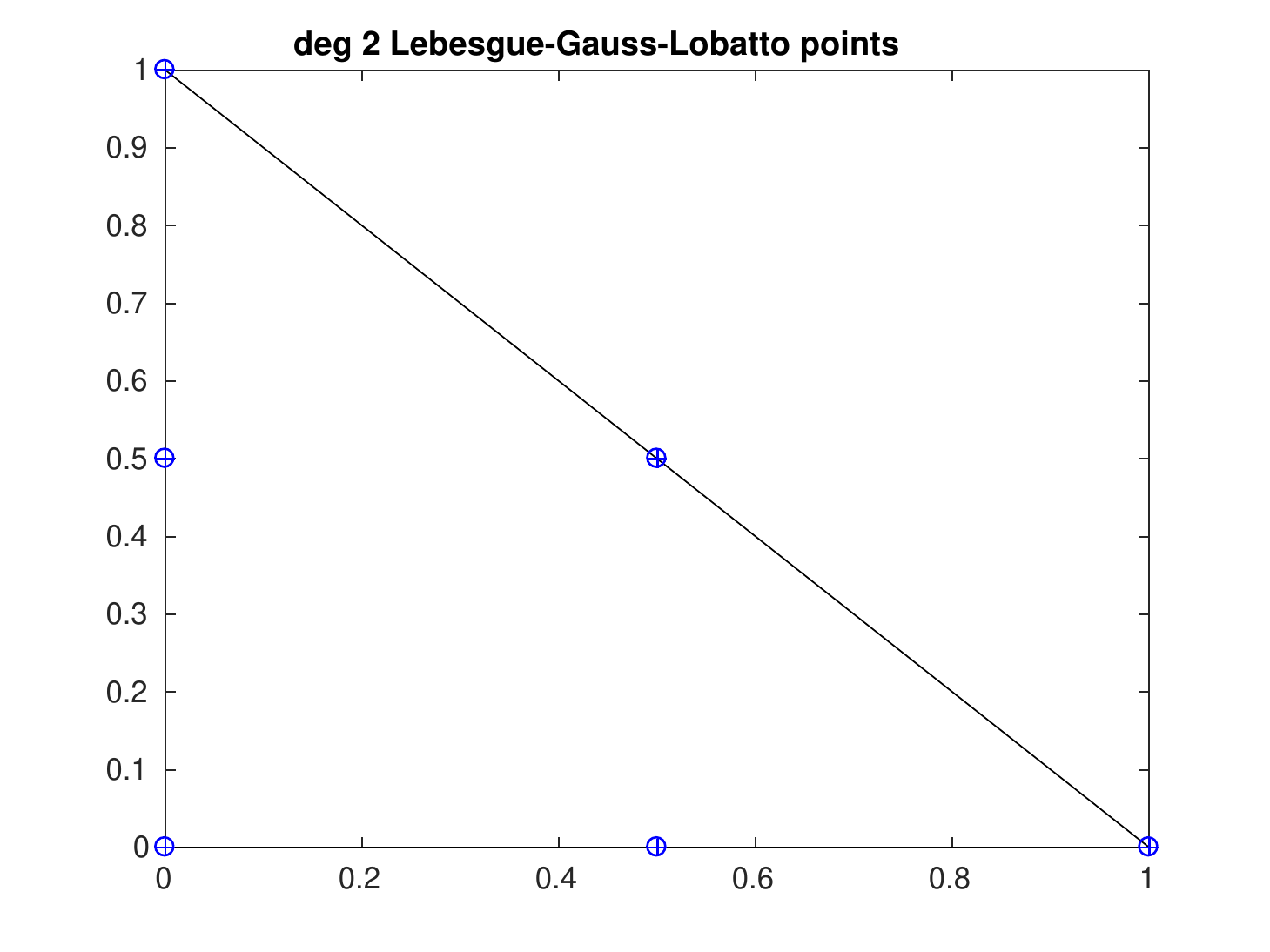}}\subfloat[degree 4]{\includegraphics[width=0.33\columnwidth]{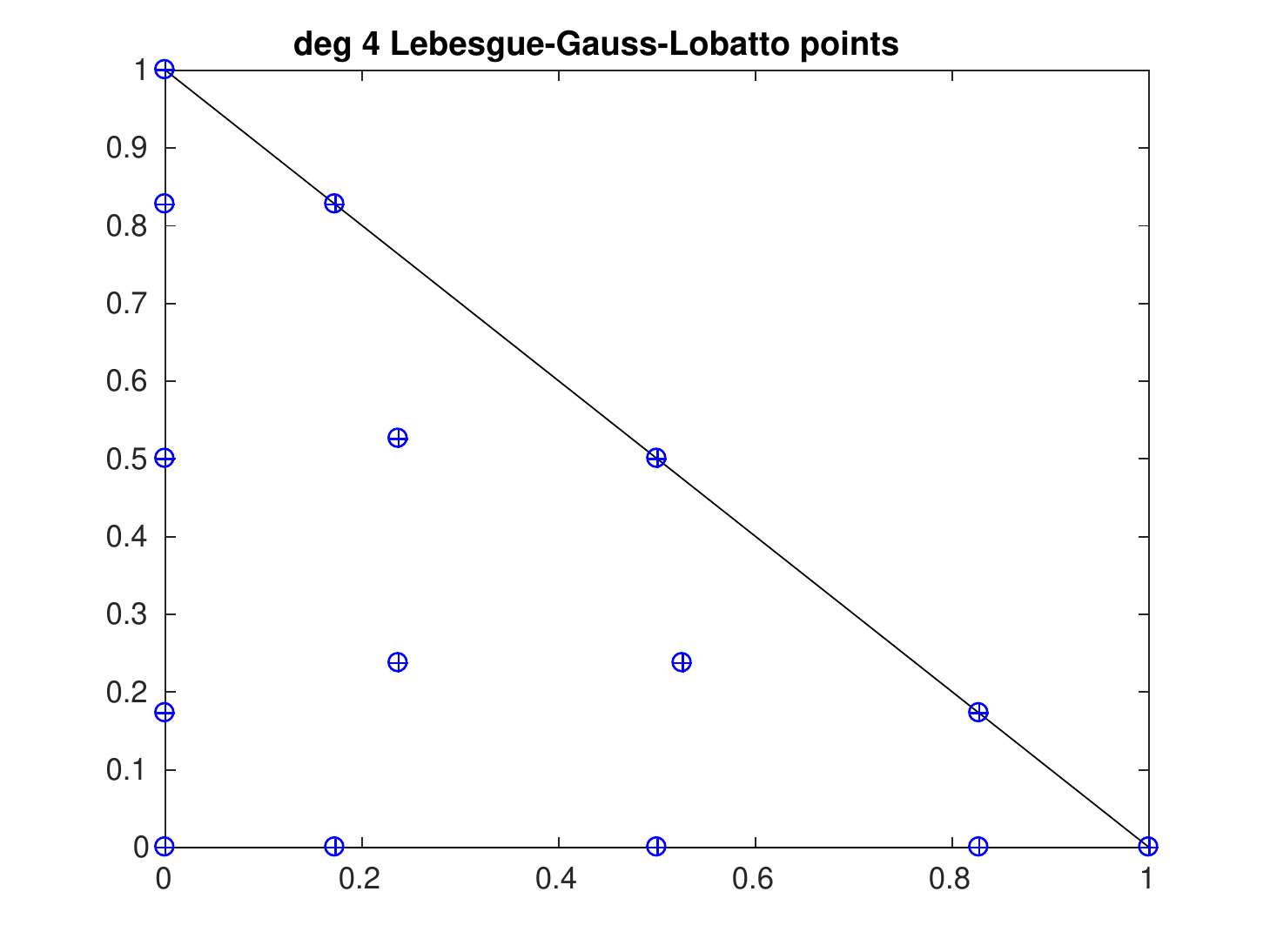}}\subfloat[degree 6]{\includegraphics[width=0.33\columnwidth]{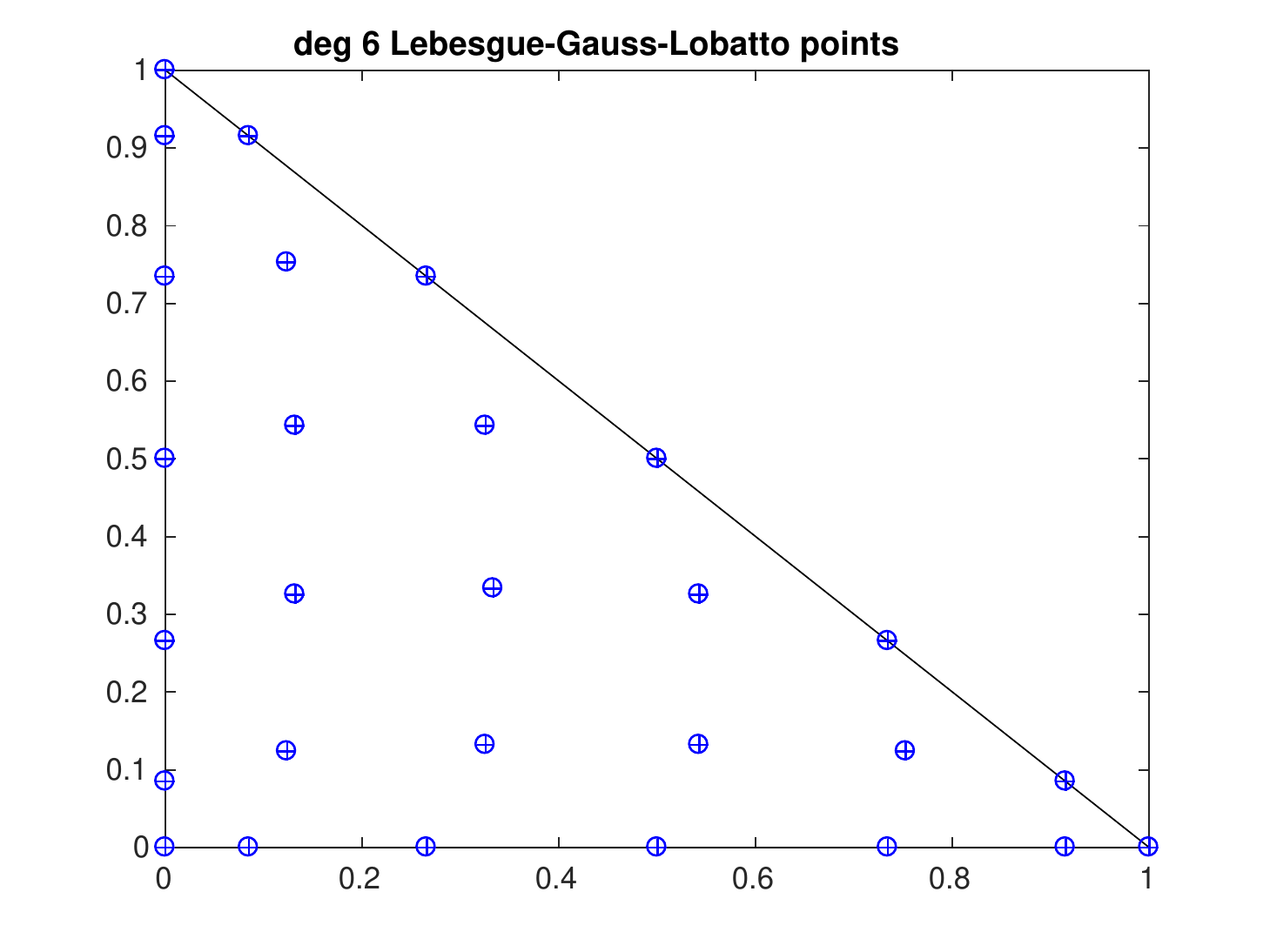}}\caption{\label{fig:leb-points}Parametric triangular elements with Lebesgue-Gauss-Lobatto
points.}
\end{figure}

\section{\label{sec:HWALF-surface}Hermite-Style High-Order Surface Reconstruction }

The CMF and WALF methods in \cite{Jiao2011RHO} had two main limitations.
First, the methods may be inaccurate for relatively coarse meshes,
for which the stencils may not have sufficient points and the safeguards
would likely reduce the degree of the basis functions. Second, they
do not guarantee $G^{0}$ continuity near sharp features (such as
ridges and corners), which might be present in piecewise smooth surfaces.
In this section, we address the first issue by extending CMF and WALF
to include normal-based information, similar to Hermite interpolation.
We refer to this as the \emph{Hermite-style reconstruction}, and refer
to its integrations with CMF and WALF as \emph{Hermite-style CMF}
(or \emph{H-CMF}) and \emph{Hermite-style WALF} (or \emph{H-WALF}),
respectively. We will describe these methods, including the selection
of stencils and weighting schemes as well as the analysis of their
accuracy.

\subsection{\label{subsec:Hermite-Style-Surface}Hermite-Style Polynomial Fittings}

\subsubsection{Local Polynomial Fittings.}

Given a point $\vec{x}_{i}$ on a smooth surface $\Gamma$, let $\vec{n}_{i}$
denote an accurate normal to $\Gamma$ at $\vec{x}_{i}$. Note that
unlike $\vec{m}_{i}$ in Section~\ref{sec:Background}, which only
needed to be first order, $\boldsymbol{n_{i}}$ should be at least
$p$th order accurate for degree-$p$ fittings, as we will show in
Section~\ref{subsec:Accuracy-Hermite}. Let $\vec{Q}_{0}^{T}\boldsymbol{n}=\left[\alpha_{i},\,\beta_{i},\,\gamma_{i}\right]^{T}$,
where $\vec{Q}_{0}$ was defined in Section~\ref{subsec:Smooth-Surfaces}.
From \eqref{eq:normal}, we have 
\begin{align}
\gamma_{i}f_{u}(\boldsymbol{u}_{i}) & =-\alpha_{i},\label{eq:deriv_u}\\
\gamma_{i}f_{v}(\boldsymbol{u}_{i}) & =-\beta_{i}.\label{eq:deriv_v}
\end{align}
From the Taylor series expansion of $f(\boldsymbol{u})$ in \eqref{eq:FTaylorseries_disc},
we have 
\begin{align}
f_{u}(\boldsymbol{u}_{i}) & \approx\sum_{q=0}^{p}\sum_{j,k\ge0}^{j+k=q}c_{jk}ju_{i}^{j-1}v_{i}^{k}\approx-\frac{\alpha_{i}}{\gamma_{i}},\label{eq:f_u}\\
f_{v}(\boldsymbol{u}_{i}) & \approx\sum_{q=0}^{p}\sum_{j,k\ge0}^{j+k=q}c_{jk}ku_{i}^{j}v_{i}^{k-1}\approx-\frac{\beta_{i}}{\gamma_{i}}.\label{eq:f_v}
\end{align}
These two equations along with the point-based approximation in \eqref{eq:1}
lead to the following linear system 
\begin{equation}
\boldsymbol{U}\vec{c}\approx\boldsymbol{f},\label{eq:unscaled-Hermite}
\end{equation}
where $\vec{c}$ is an $n$-vector composed of $c_{jk}$, $\boldsymbol{U}$
is a $3m\times n$ matrix, and $\boldsymbol{f}$ is a $3m$-vector.
For example, a degree-2 fitting with $m$ points in the stencil results
in the following $\boldsymbol{U}$ and $\boldsymbol{f}$: 
\[
\boldsymbol{U}=\left[\begin{array}{cccccc}
1 & u_{1} & v_{1} & u_{1}^{2} & u_{1}v_{1} & v_{1}^{2}\\
1 & u_{2} & v_{2} & u_{2}^{2} & u_{2}v_{2} & v_{2}^{2}\\
 & \cdots &  &  & \cdots\\
1 & u_{m} & v_{m} & u_{m}^{2} & u_{m}v_{m} & v_{m}^{2}\\
0 & 1 & 0 & 2u_{1} & v_{1} & 0\\
0 & 1 & 0 & 2u_{2} & v_{2} & 0\\
 & \cdots &  &  & \cdots\\
0 & 1 & 0 & 2u_{m} & v_{m} & 0\\
0 & 0 & 1 & 0 & u_{1} & 2v_{1}\\
0 & 0 & 1 & 0 & u_{2} & 2v_{2}\\
 & \cdots &  &  & \cdots\\
0 & 0 & 1 & 0 & u_{m} & 2v_{m}
\end{array}\right]\quad\text{and}\quad\boldsymbol{f}=\left[\begin{array}{c}
f_{1}\\
f_{2}\\
\vdots\\
f_{m}\\
-\alpha_{1}/\gamma_{1}\\
-\alpha_{2}/\gamma_{2}\\
\vdots\\
-\alpha_{m}/\gamma_{m}\\
-\beta_{1}/\gamma_{1}\\
-\beta_{2}/\gamma_{2}\\
\vdots\\
-\beta_{m}/\gamma_{m}
\end{array}\right].
\]
Like the regular fitting, the Hermite-style fitting is interpolatory
at $\vec{x}_{0}$ if the local polynomial passes through the origin
of fittings, i.e.,$f(\boldsymbol{u}_{0})=0$.

Note that if $\gamma_{i}\leq0$, the local height function $f(\boldsymbol{u})$
would have foldings about $\boldsymbol{x}_{i}$, which can lead to
non-convergence. This issue could be avoided by settings the weights
for the folded vertices to zero so that they will be eliminated from
the linear system; see Section~\ref{subsec:weighting-scheme}. Also
note that for surfaces with sharp features, the normals along ridges
and at corners are not well-defined. In these cases, we can either
use one-sided normals at each vertex on sharp features or do not include
the normal information for those vertices. In the following, we shall
assume the surface is smooth unless otherwise noted, so that the normal
is well-defined at all the vertices in the stencil; we will address
the reconstruction of feature curves in Section~\ref{sec:HWALF-curve}.

\subsubsection{\label{subsec:geom_scaling_surf}Coordinate Transformation Based
on Geometric Scaling.}

Due to the inclusion of normals into the Hermite-style fittings, the
rows in \eqref{eq:unscaled-Hermite} now have mixed-degree terms.
To normalize the entries in matrix $\boldsymbol{U}$, we apply geometric
scaling similar to that described in Section~\ref{subsec:Local-Polynomial-Fittings.}.
Specifically, let $\mu=\dfrac{u}{h}$ and $\nu=\dfrac{v}{h}$, where
$h$ is a measure of local edge length. From the chain rule, we have

\begin{equation}
\dfrac{\partial f^{j+k}}{\partial\mu^{j}\partial\nu^{k}}=\dfrac{\partial f^{j+k}}{\partial u^{j}\partial v^{k}}h^{j+k}.\label{eq:transformation}
\end{equation}
Let $x_{jk}=h^{j+k}c_{ij}=\dfrac{1}{j!k!}\dfrac{\partial f^{j+k}(\vec{u}_{0})}{\partial\mu^{j}\partial\nu^{k}}$,
$\mu_{i}=\dfrac{u_{i}}{h}$, and $\nu_{i}=\dfrac{v_{i}}{h}$ for $0\leq i\leq m$.
From the Taylor series expansion of $f_{u}$ and $f_{v}$ in \eqref{eq:f_u}
and \eqref{eq:f_v} , we obtain
\begin{align}
\sum_{q=0}^{p}\sum_{j,k\ge0}^{j+k=q}x_{jk}\mu_{i}^{j}\nu_{i}^{k} & \approx f_{i},\\
\sum_{q=0}^{p}\sum_{j,k\ge0}^{j+k=q}x_{jk}j\mu_{i}^{j-1}\nu_{i}^{k} & \approx-\frac{\alpha_{i}}{\gamma_{i}}h,\\
\sum_{q=0}^{p}\sum_{j,k\ge0}^{j+k=q}x_{jk}k\mu_{i}^{j}\nu_{i}^{k-1} & \approx-\frac{\beta_{i}}{\gamma_{i}}h.
\end{align}
This results in a rescaled linear system 
\begin{equation}
\boldsymbol{V}\vec{x}\approx\boldsymbol{g},\label{eq:rescaled-linear-system}
\end{equation}
where $\vec{x}$ is an $n$-vector composed of $x_{jk}$, $\boldsymbol{V}$
is a $3m\times n$ matrix, and $\boldsymbol{g}$ is a $3m$-vector.
For example, a degree-2 fitting results in the following $\boldsymbol{V}$
and $\boldsymbol{g}$: 
\begin{equation}
\boldsymbol{V}=\left[\begin{array}{cccccc}
1 & \mu_{1} & \nu_{1} & \mu_{1}^{2} & \mu_{1}\nu_{1} & \nu_{1}^{2}\\
1 & \mu_{2} & \nu_{2} & \mu_{2}^{2} & \mu_{2}\nu_{2} & \nu_{2}^{2}\\
 & \cdots &  &  & \cdots\\
1 & \mu_{m} & \nu_{m} & \mu_{m}^{2} & \mu_{m}\nu_{m} & \nu_{m}^{2}\\
0 & 1 & 0 & 2\mu_{1} & \nu_{1} & 0\\
0 & 1 & 0 & 2\mu_{2} & \nu_{2} & 0\\
 & \cdots &  &  & \cdots\\
0 & 1 & 0 & 2\mu_{m} & \nu_{m} & 0\\
0 & 0 & 1 & 0 & \mu_{1} & 2\nu_{1}\\
0 & 0 & 1 & 0 & \mu_{2} & 2\nu_{2}\\
 & \cdots &  &  & \cdots\\
0 & 0 & 1 & 0 & \mu_{m} & 2\nu_{m}
\end{array}\right]\quad\text{and}\quad\boldsymbol{g}=\left[\begin{array}{c}
f_{1}\\
f_{2}\\
\vdots\\
f_{m}\\
-h\alpha_{1}/\gamma_{1}\\
-h\alpha_{2}/\gamma_{2}\\
\vdots\\
-h\alpha_{m}/\gamma_{m}\\
-h\beta_{1}/\gamma_{1}\\
-h\beta_{2}/\gamma_{2}\\
\vdots\\
-h\beta_{m}/\gamma_{m}
\end{array}\right].\label{eq:rescaled-linear-system-degree-2}
\end{equation}
Mathematically, this geometric scaling is equivalent to multiplying
the last $2m$ rows of \eqref{eq:unscaled-Hermite} by $h$. In other
words, $\vec{V}=\vec{D}\vec{U}\vec{T}$, where 
\begin{equation}
\vec{D}=\left[\begin{array}{ccc}
\vec{I}\\
 & h\vec{I}\\
 &  & h\vec{I}
\end{array}\right],\label{eq:scaling-Hermite}
\end{equation}
and $\vec{T}$ is the diagonal matrix composed of $h^{-j-k}$, where
$j$ and $k$ are the powers in $\mu^{j}\nu^{k}$.

\subsubsection{\label{subsec:Stencil-Selection-Hermite}Stencil Selection.}

The stencil selection is important for the accuracy and efficiency
of the local polynomial fittings, since too large a stencil tends
to cause overfitting, while too small a stencil leads to low order
accuracy. A degree\emph{-$p$ }polynomial fitting has $n=(p+1)(p+2)/2$
coefficients to determine, so it requires at least $n$ points in
the stencil for point-based fittings. However, for Hermite-style fittings,
there are three equations for each point: one from the vertex position,
and two from the normal vector. Hence, Hermite-style fittings require
much fewer points in the stencil. As described in Section~\ref{subsec:WLS-Tri-Surf},
we choose the stencil of a vertex based on its $k$-ring neighborhood
in $1/2$ increments. Table~\ref{tab:Comparison-of-stencils} shows
a typical choice of ring size of point-based and Hermite-style fittings.
In addition, the table also shows the average numbers of vertices
in the rings for an example triangulation of a torus with 336 triangles.
It can be seen that for the Hermite-style fittings, it typically suffices
to use only a 1-ring neighborhood for degree-4 fittings, and only
a 2-ring neighborhood for degree-6 fittings. These are much more compact
than the typical 2.5-ring and 3.5-ring neighborhoods for the corresponding
point-based fittings. Note that similar to the point-based fittings,
if there are insufficient vertices in a stencil, our stencil selection
procedure adaptively enlarges the ring sizes.

\begin{table}
\caption{\label{tab:Comparison-of-stencils}Comparison of average stencil sizes
for point-based and Hermite-style fittings.}

\centering{}{\scriptsize{}}%
\begin{tabular}{c|c|c|c|c|c|c|c|c|c|c}
\hline 
 & \multicolumn{2}{c|}{{\scriptsize{}degree 2}} & \multicolumn{2}{c|}{{\scriptsize{}degree 3}} & \multicolumn{2}{c|}{{\scriptsize{}degree 4}} & \multicolumn{2}{c|}{{\scriptsize{}degree 5}} & \multicolumn{2}{c}{{\scriptsize{}degree 6}}\tabularnewline
\hline 
\hline 
{\scriptsize{}\#unknowns} & \multicolumn{2}{c|}{{\scriptsize{}6}} & \multicolumn{2}{c|}{{\scriptsize{}10}} & \multicolumn{2}{c|}{{\scriptsize{}15}} & \multicolumn{2}{c|}{{\scriptsize{}21}} & \multicolumn{2}{c}{{\scriptsize{}28}}\tabularnewline
\hline 
 & {\scriptsize{}\#ring} & {\scriptsize{}\#vert} & {\scriptsize{}\#ring} & {\scriptsize{}\#vert} & {\scriptsize{}\#ring} & {\scriptsize{}\#vert} & {\scriptsize{}\#ring} & {\scriptsize{}\#vert} & {\scriptsize{}\#ring} & {\scriptsize{}\#vert}\tabularnewline
\hline 
{\scriptsize{}point-based} & {\scriptsize{}1.5} & {\scriptsize{}12.6} & {\scriptsize{}2} & {\scriptsize{}18.4} & {\scriptsize{}2.5} & {\scriptsize{}29.6} & {\scriptsize{}3} & {\scriptsize{}37.5} & {\scriptsize{}3.5} & {\scriptsize{}52.7}\tabularnewline
\hline 
{\scriptsize{}Hermite-style} & {\scriptsize{}1} & {\scriptsize{}6.8} & {\scriptsize{}1} & {\scriptsize{}6.8} & {\scriptsize{}1} & {\scriptsize{}6.8} & {\scriptsize{}1.5} & {\scriptsize{}12.6} & {\scriptsize{}2} & {\scriptsize{}18.4}\tabularnewline
\hline 
\end{tabular}{\scriptsize\par}
\end{table}

Note that for points near sharp features, we must make sure not to
cross a ridge curve when building the stencil. To this end, we virtually
split the surface into smooth patches along the feature curves, and
reconstruct each smooth patch independently. This is more robust,
if the feature curves can be identified \emph{a priori} from the CAD
model or using some robust algorithms, such as that in \cite{JB08ICC}.
For ridge curves with very large dihedral angles, one could also use
the safeguard $\theta^{+}$ in the weighting scheme to filter out
points for simplicity.

\subsubsection{\label{subsec:weighting-scheme}Weighting Scheme.}

Similar to the point-based fittings, \eqref{eq:rescaled-linear-system}
can be solved under the framework of weighted linear least squares
to minimize the weighted norm, 
\begin{equation}
\min_{\boldsymbol{x}}\left\Vert \boldsymbol{\Omega}(\boldsymbol{Vx}-\boldsymbol{g})\right\Vert _{2}=\min_{\boldsymbol{c}}\left\Vert \boldsymbol{\Omega D}(\boldsymbol{Uc}-\boldsymbol{f})\right\Vert _{2},\label{eq:weighted-least-squares}
\end{equation}
where $\boldsymbol{\Omega}=\text{diag}(\omega_{1},\dots,\omega_{3m})$
is a $3m\times3m$ weighting matrix, and $\vec{c}=\vec{T}\vec{x}$
for $\vec{c}$ and $\vec{T}$ defined in \eqref{eq:unscaled-Hermite}
and \eqref{eq:scaling-Hermite}, respectively.

A key question is the selection of the weights. Let us first consider
the first $m$ weights $\omega_{1},\dots,\omega_{m}$ corresponding
to the equations \eqref{eq:1}. As described in Section~\ref{subsec:WLS-Tri-Surf},
we construct the weights based on a combination of Wendland's functions
\cite{wendland1995piecewise} and normal-based safeguards, Specifically,
let

\begin{equation}
\omega_{i}=\theta_{i}^{+}\psi(\Vert\vec{u}_{i}\Vert/\rho),\label{eq:weights}
\end{equation}
where $\theta_{i}^{+}=\textrm{max}(0,\vec{m}_{i}^{T}\vec{m}_{0})$,
$\rho$ is a measure of the radius of the stencil, and $\psi$ is
a Wendland's function. For even degrees $2,$ 4, and $6$, we use
$\psi=\psi_{3,1}$, $\psi_{4,2}$, and $\psi_{5,3}$ as defined in
\eqref{eq:Wendland_phi_3_1}--\eqref{eq:Wendland_phi_5_3}, respectively;
for odd degrees $p=2q-1$, we use the same weighting schemes as for
degree $2q$. To compute $\rho$, we find the $k$th nearest neighbors
of $\vec{x}_{0}$ in $uv$-plane within the stencil, where $k$ is
chosen to be the ceiling of $1.5$ times the number of unknowns, i.e.,
$k=\left\lceil 0.75(p+1)(p+2)\right\rceil $, and then
\begin{equation}
\rho=c\Vert\vec{u}_{k}\Vert,\label{eq:l-nearest-1}
\end{equation}
where $c>1$ depends on the degree of fitting. In particular, we choose
$c=1.15$, $1.2$, and $1.25$ for degrees $2,$ 4, and $6$, respectively,
which we obtained via numerical experimentation. For the equations
corresponding to $f_{u}$ and $f_{v}$, since we have applied geometric
scalings in \eqref{eq:rescaled-linear-system}, we apply the same
weights $\omega_{i}$ to all the equations associated with points
$\vec{u}_{i}$. In other words, 
\[
\vec{\Omega}=\begin{bmatrix}\vec{W}\\
 & \vec{W}\\
 &  & \vec{W}
\end{bmatrix}\quad\text{and}\quad\vec{\Omega}\vec{D}=\begin{bmatrix}\vec{W}\\
 & h\vec{W}\\
 &  & h\vec{W}
\end{bmatrix},
\]
where $\vec{W}$ is composed of $\omega_{i}$ in \eqref{eq:weights}.
The resulting weighted Vandermonde system can then be solved using
truncated QR with column pivoting as described in Section~\ref{subsec:Local-Polynomial-Fittings.}.

\subsection{\label{subsec:H-CMF}H-CMF and H-WALF}

The Hermite-style polynomial fittings described above can be integrated
into CMF and WALF, which refer to as \emph{H-CMF} and \emph{H-WALF},
correspondingly. H-CMF works similarly to CMF, as described in Section~\ref{subsec:WLS-Tri-Surf}.
In particular, it first constructs the local coordinate frame at point
$\vec{p}$ in triangle $\vec{x}_{1}\vec{x}_{2}\vec{x}_{3}$ by averaging
the approximate nodal normals $\vec{m}_{j}$ using the barycentric
coordinates $\xi_{j}$. Then, the stencil for $\vec{p}$ is taken
to be the union of the stencils of the three nodes. H-CMF does not
guarantee $G^{0}$ continuity, because the weighting scheme may be
discontinuous and there may be truncation when solving the least squares
problems.

H-WALF works similarly to WALF, as described in Section~\ref{subsec:WALF}.
More specifically, consider a triangle composed of vertices $\vec{x}_{j}$,
$j=1,2,3$. For any point $\vec{p}$ in the triangle, we first obtain
a point $\vec{q}_{j}$ from the Hermite-style fitting in the local
coordinate frame at $\vec{x}_{j}$. Let $\xi_{j}$, $j=1,2,3$ be
the barycentric coordinates of $\vec{p}$ within the triangle. Then,
the H-WALF reconstruction of $\vec{p}$ is given by $\vec{q}=\sum_{j=1}^{3}\xi_{j}\vec{q}_{j}$.
For smooth surfaces, H-WALF constructs a $G^{0}$ continuous surface.

\subsection{\label{subsec:Accuracy-Hermite}Accuracy of Hermite-Style Fittings}

To analyze the accuracy of H-CMF and H-WALF, we must first understand
the convergence of the Hermite-style polynomial fittings. This analysis
is similar to the point-based fittings in \cite{JZ08CCF}, but the
inclusion of the normals requires some special care. Hence, we include
the analysis here for completeness. Assuming the mesh is relatively
uniform, let \emph{h} denote a measure of average edge length of the
mesh. We obtain the following lemma.
\begin{lem}
\label{lem:accuracy-Hermite}Given a set of points $[u_{i},v_{i},\tilde{f}_{i}]$
that interpolate a smooth height function $f$ or approximate f with
an error of $\mathcal{O}(h^{p+1})$, along with the gradients $[f_{u}(\vec{u}_{i}),f_{v}(\vec{u}_{i})]$,
which are approximated to $\mathcal{O}(h^{p})$, assume that the point
distribution and the weights are independent of $h$, and the condition
number in any $p$-norm of the scaled matrix $\boldsymbol{\varOmega V}$
is bounded by some constant. Then, the degree-d weighted least squares
fitting approximates $c_{jk}$ in \eqref{eq:unscaled-Hermite} to
$\mathcal{O}(h^{p-j-k+1})$.
\end{lem}
\begin{proof}
Consider the least squares problem \eqref{eq:unscaled-Hermite}. Let
$\vec{A}=\vec{\Omega}\vec{V}=\boldsymbol{\varOmega DUT}$ and $\vec{b}=\vec{\Omega}\vec{g}$.
Let $\hat{\vec{c}}$ denote the exact coefficients in the Taylor polynomial
\eqref{eq:FTaylorseries_cont}. Let $\hat{\vec{x}}=\vec{T}\hat{\vec{c}}$,
$\hat{\vec{b}}=\vec{\Omega}\vec{V}\vec{x}$, $\delta\vec{x}=\vec{x}-\hat{\vec{x}}$,
and $\vec{r}=\vec{b}-\hat{\vec{b}}$. Then, $\vec{A}\hat{\vec{x}}=\hat{\vec{b}}$,
and $\delta\vec{x}$ is the least squares solution to 
\[
\vec{A}\delta\vec{x}\approx\vec{r}.
\]
Hence, 
\[
\left\Vert \delta\vec{x}\right\Vert _{\infty}\leq\left\Vert \vec{A}^{+}\right\Vert _{\infty}\left\Vert \vec{r}\right\Vert _{\infty}.
\]
Note that $\vec{r}=\vec{b}-\hat{\vec{b}}=\vec{\Omega}\vec{D}(\vec{f}-\vec{U}\hat{\vec{c}})$.
Under the assumption that $f(\vec{u}_{i})$ is approximated to at
least $\mathcal{O}(h^{p+1})$ and the derivatives $f_{u}(\vec{u}_{i})$
and $f_{v}(\vec{u}_{i})$ are approximated to $\mathcal{O}(h^{p})$,
each entry in $\vec{D}(\vec{f}-\vec{U}\hat{\vec{c}})$ is $\mathcal{O}(h^{p+1})$,
so is each entry in $\vec{r}$ since $\omega_{i}=\Theta(1)$. Under
the assumptions of the theorem, $\kappa_{\infty}(\vec{A})=\left\Vert \vec{A}\right\Vert _{\infty}\left\Vert \vec{A}^{+}\right\Vert _{\infty}=\Theta(1)$
and $\left\Vert \vec{A}\right\Vert _{\infty}=\Theta(1)$. Hence, $\left\Vert \vec{A}^{+}\right\Vert _{\infty}=\Theta(1)$
and $\left\Vert \delta\vec{x}\right\Vert _{\infty}=\mathcal{O}(h^{p+1})$.
Therefore, each entry in $\delta\vec{c}=\vec{T}\delta\vec{x}$ corresponding
to $c_{ij}$ is $\mathcal{O}(h^{p-j-k+1})$.
\end{proof}
The accuracy of H-CMF directly follows from Lemma~\ref{lem:accuracy-Hermite}.
\begin{prop}
\label{prop:h-cmf}Given a mesh whose vertices approximate a smooth
surface $\Gamma$ with an error of $\mathcal{O}(h^{p+1})$ and the
normal vectors are also approximated to $\mathcal{O}(h^{p})$, assuming
the rescaled Vandermonde systems are well conditioned, the distance
between each point on the H-CMF reconstructed surface and its closest
point on $\Gamma$ is $\mathcal{O}(h^{p+1})$.
\end{prop}
\begin{proof}
In H-CMF, the gradients to the local height function $f$ are $[-\alpha_{i}/\gamma_{i},-\beta_{i}/\gamma_{i}]$,
where $\vec{Q}_{0}^{T}\boldsymbol{n}_{i}=[\alpha_{i},\beta_{i},\gamma_{i}]^{T}$.
Since the normals are $p$th order accurate, so is the approximation
to $f_{u}$ and $f_{v}$ as long as $\gamma_{i}$ is bounded away
from 0. It then follows from Lemma~\ref{lem:accuracy-Hermite} that
the local height function is approximated to $\mathcal{O}(h^{p+1})$,
so is the distance from the reconstructed point to the closest point
on $\Gamma$.
\end{proof}
In Proposition~\ref{prop:h-cmf}, the well-conditioning of the Vandermonde
system is typically achieved by the adaptive stencil selection. If
the Vandermonde system is ill-conditioned, then some higher-order
terms may be truncated by QRCP, and and the convergence rate may be
lower. In the proof, it is important that $\gamma_{i}$ is bounded
away from 0, which can be ensured by the normal-based safeguards in
the weighting schemes.

The accuracy of H-WALF is more complicated, in that like WALF, its
error has a lower bound $\mathcal{O}(h^{6})$. We summarize its convergence
rate as follows.
\begin{prop}
\label{prop:h-walf}Under the same assumption as Proposition~\ref{prop:h-cmf},
the distance between each point on the H-WALF reconstructed surface
and its closest point on $\Gamma$ is $\mathcal{O}(h^{p+1}+h^{6})$.
\end{prop}
The lower bound of $\mathcal{O}(h^{6})$ error is due to the discrepancies
of the local coordinate frames at the vertices of a triangle. Figure~\ref{fig:surface_}
illustrates the origin of this error bound. Let $\vec{q}$ denote
the H-WALF reconstruction of a point $\vec{p}$ in the triangle $\vec{x}_{1}\vec{x}_{2}\vec{x}_{2}$.
Let $\bar{\vec{q}}^{*}$ be the closest point of $\vec{q}$ on $\Gamma$.
Let $\vec{q}_{j}^{*}$ be projection of $\vec{p}$ onto the exact
surface $\Gamma$ along $\vec{m}_{j}$, and $\vec{q}^{*}=\sum_{j=1}^{3}\xi_{j}\vec{q}_{j}^{*}$.
Then 
\[
\text{dist}(\vec{q},\Gamma)\leq\parallel\vec{q}-\bar{\vec{q}}^{*}\parallel\leq\parallel\vec{q}-\vec{q}^{*}\parallel+\parallel\vec{q}^{*}-\bar{\vec{q}}^{*}\parallel.
\]
The error $\parallel\vec{q}-\vec{q}^{*}\parallel=\mathcal{O}(h^{p+1})$
is due to Lemma~\ref{lem:accuracy-Hermite}. It can be shown that
$\Vert\vec{q}_{i}^{*}-\vec{q}_{j}^{*}\Vert=\mathcal{O}(h^{3})$ for
$1\le i,j\leq3$ and $\parallel\vec{q}^{*}-\bar{\vec{q}}^{*}\parallel=\mathcal{O}(\max\Vert\vec{q}_{i}^{*}-\vec{q}_{j}^{*}\Vert)^{2}=\mathcal{O}(h^{6})$;
see \cite{Jiao2011RHO} for a complete proof. 
\begin{figure}
\centering{}\subfloat[An element and exact surface.]{\includegraphics[width=0.45\textwidth]{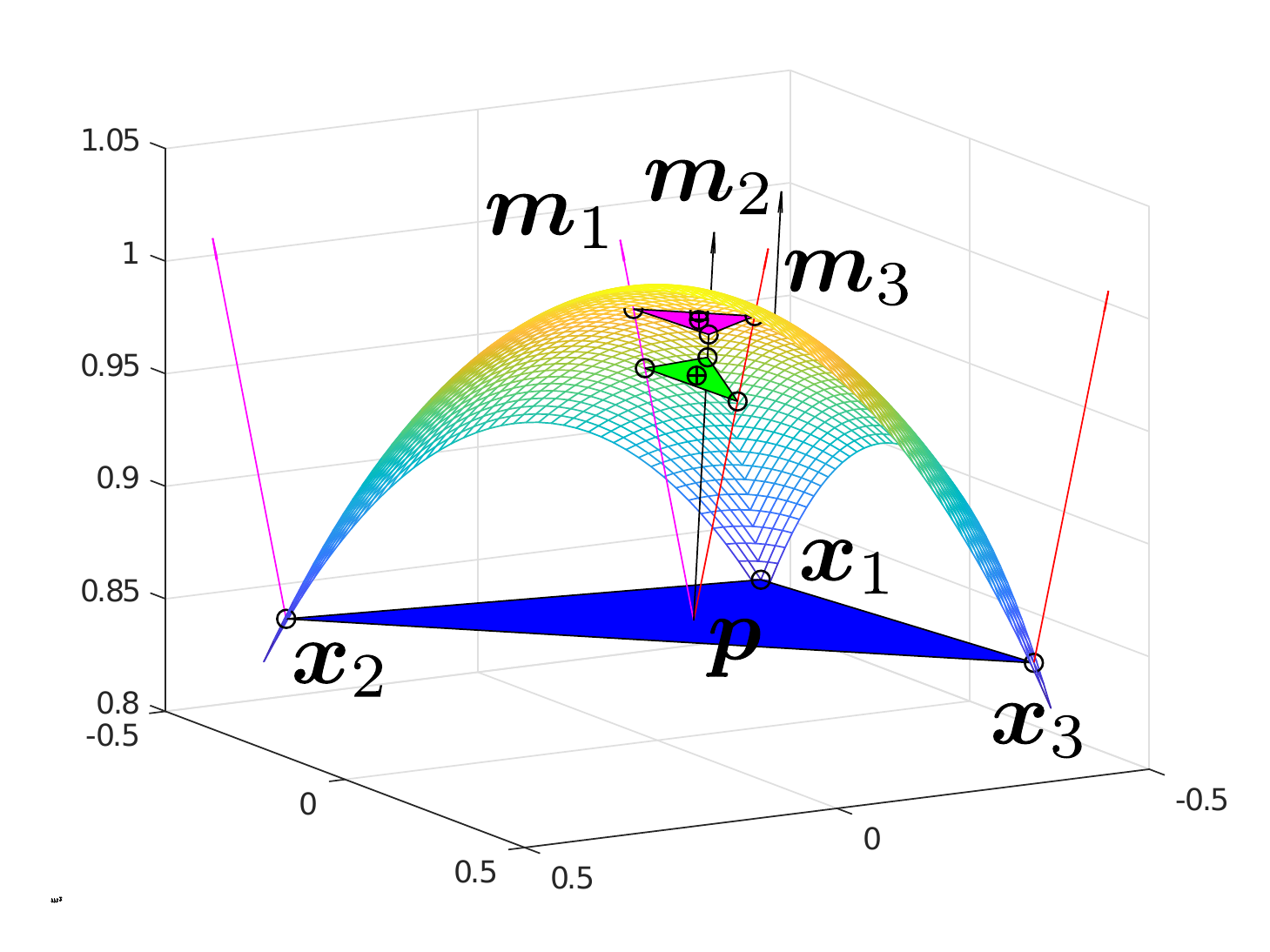}}\subfloat[Enlarged local view.]{\includegraphics[width=0.45\textwidth]{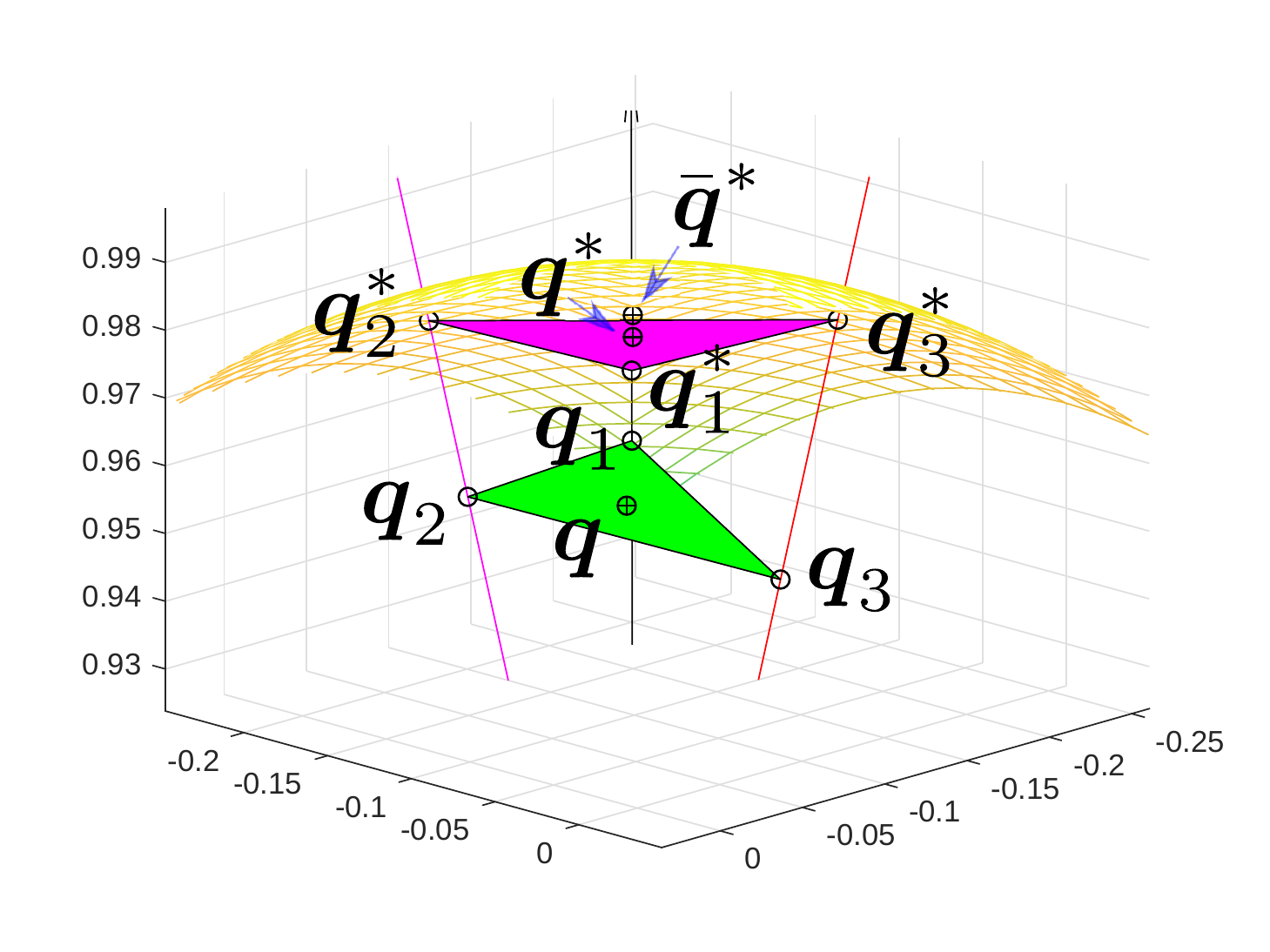}}\caption{\label{fig:surface_}3-D illustration of error analysis in H-WALF.}
\end{figure}

Finally, we note that in Lemma~\ref{lem:accuracy-Hermite}, for even-degree
polynomials, the leading-error terms are odd-degree polynomials, which
may result in error cancellation if the stencils are perfectly symmetric,
analogous to the error cancellation in center-difference schemes.
In practice, error cancellation also occurs for nearly symmetric stencils.
Therefore, we may observe similar convergence rates for polynomial
fittings of degrees $2q$ and $2q+1$. Furthermore, since degree-2$q$
fittings require smaller stencils, they may have even smaller errors
than degree-($2q+1$) fittings. Hence, it is desirable to use even-degree
polynomials for high-order reconstructions for nearly symmetric stencils,
as we will demonstrate in Section~\ref{sec:Numerical-Results}. Furthermore,
Propositions~\ref{prop:h-cmf} and \ref{prop:h-walf} imply that
H-WALF is less accurate than H-CMF for high-degree polynomials. In
practice, H-WALF is well suited for degree-2 or degree-4 reconstructions
for its better efficiency, and H-CMF is better suited for degree-6
or higher-order reconstructions. In addition, we note that H-CMF tends
to deliver better stability than H-WALF if the stencil is one-sided,
and hence for open surfaces or piecewise smooth surfaces, H-CMF is
preferred near boundaries or sharp features.

\section{\label{sec:HWALF-curve}Hermite-Style High-Order Curve Reconstruction}

In this section, we present a procedure for high-order reconstruction
of space curves, such as feature curves on a piecewise smooth surface
or the boundary curve of an open surface. We focus on Hermite-style
reconstruction using points and tangents, which can be simplified
to point-based reconstruction if one omits the equations associated
with tangents. We assume the curve is piecewise smooth, and its end-points
or corners are accurate and do not need reconstruction.

\subsection{\label{subsec:H-CMF-curve}H-CMF and H-WALF Curve Reconstructions}

Consider a point $\vec{x}_{0}$ and a collection of points $\{\vec{x}_{i}\mid1\leq i\leq m\}$
in its neighborhood on a curve. Let $\vec{t}_{i}$ denote an accurate
tangent vector to the curve at $\vec{x}_{i}$. Let $u_{i}$ denote
the local coordinate of $\vec{x}_{i}$ in the local $uvw$ frame centered
$\vec{x}_{0}$ as defined in Section~\ref{subsec:Space-Curves},
and let $\vec{Q}_{0}^{T}(\vec{x}_{i}-\vec{x}_{0})=\left[u_{i},v_{i},w_{i}\right]^{T}$.
Let $\vec{f}$ denote the vector-valued local height function. The
Taylor series of $\vec{f}$ about $u_{0}=0$ is given by 
\begin{equation}
\vec{f}(u)\approx\stackrel[q=0]{p}{\sum}\vec{c}_{q}u^{q}\label{eq:Taylor-polynomial-curve}
\end{equation}
where $\vec{c}_{q}=\dfrac{1}{q!}\dfrac{d^{q}}{du^{k}}\vec{f}(u_{0})$.
Let $c_{q}$ and $d_{q}$ denote the two entries in $\vec{c}_{q}$,
respectively. For each point $\vec{x}_{i}$, we then have two equations
\begin{gather}
\stackrel[q=0]{p}{\sum}c_{q}u_{i}^{q}\approx u_{i},\label{eq:curve_taylordisc1}\\
\stackrel[q=0]{p}{\sum}d_{q}u_{i}^{q}\approx w_{i}.\label{eq:curve_taylordisc2}
\end{gather}
This amounts to a $2m\times n$ linear system for point-based curve
reconstruction, where $n=p+1$.

For Hermite-style reconstruction, let $\vec{Q}_{0}^{T}\vec{t}_{i}=\left[\alpha_{i},\beta_{i},\gamma_{i}\right]^{T}$,
and then $\vec{f}'(u_{i})=\left[\dfrac{\beta_{i}}{\alpha_{i}},\dfrac{\gamma_{i}}{\alpha_{i}}\right]^{T}$.
The Taylor series of $\vec{f}'$ about $u_{0}=0$ is given by 
\begin{align}
\vec{f}'(u) & \approx\stackrel[q=1]{p}{\sum}q\vec{c}_{q}u^{q-1}.\label{eq:deriv_Taylor}
\end{align}
For the given tangent vector at vertex $\vec{x}_{i}$, we then have
two additional equations 
\begin{gather}
\stackrel[q=1]{p}{\sum}qc_{q}u_{i}^{q-1}\approx\dfrac{\beta_{i}}{\alpha_{i}},\label{eq:curve_taylor-deriv1}\\
\stackrel[q=1]{p}{\sum}qc_{q}u_{i}^{q-1}\approx\dfrac{\gamma_{i}}{\alpha_{i}}.\label{eq:curve_taylor-deriv2}
\end{gather}
Assuming a tangent vector is given for each point $\vec{x}_{i}$,
we obtain a $4m\times n$ linear system for Hermite-style curve reconstruction.
For curves with corners, the tangent direction at a corner may not
be well-defined. In this case, we can either use one-sided tangent
at a corner, or do not include the equations associated with the tangents
at corners.

On a triangulated surface, a feature curve is composed of edges. In
general, we choose the stencil to be the $\left\lceil (p+1)/2\right\rceil $
and $\left\lceil (p+1)/4\right\rceil $ rings for the point-based
and Hermite-style reconstructions, respectively, and we adaptively
enlarge the ring sizes if there are insufficient number of points.
We use the Wendland weights with safeguards, as described in Section~\ref{subsec:Hermite-Style-Surface},
except that safeguard $\theta_{i}^{+}$ are now defined based on tangents
instead of normals. The resulting weighted least squares problem is
then rescaled geometrically and solved robustly using QRCP.

By defining a $C^{0}$ continuous tangent vector field on a feature
curve, we obtain the CMF and H-CMF local reconstructions. To recover
a $G^{0}$ continuous curve, we can utilize the weighted averaging
of the local reconstructions at the vertices to obtain WALF and H-WALF
reconstructions. Alternatively, we can use high-order parametric elements
to define a $G^{0}$ continuous curve. These constructions are similar
to their counterparts for surfaces as described in Section~\ref{sec:HWALF-surface},
and hence we omit their details.

\subsection{Accuracy of Curve Reconstructions}

Let \emph{h} denote the average edge length of the mesh. Then, the
following lemma can be established:
\begin{lem}
\label{lem:accuracy-coefficients}Given a set of points $[u_{i},v_{i},w_{i}]$
that interpolate a smooth curve or approximate the curve with an error
of $\mathcal{O}(h^{p+1})$, along withe the derivatives $v'(u_{i})$
and $w'(u_{i})$, which are approximated to $\mathcal{O}(h^{p})$,
assume the point distribution and the weights are independent of $h$,
and the condition number of the scaled matrix $\boldsymbol{\varOmega V}$
is bounded by some constant. The degree-d weighted least squares fitting
approximates $c_{q}$ and $d_{q}$ to $\mathcal{O}(h^{p-q+1})$. 
\end{lem}
The proof of this lemma is similar to that of Lemma~\ref{lem:accuracy-Hermite}.
The key in the proof is that after geometric scaling, the component
in the residual vector is $\Theta(1)$, so is the perturbation to
the solution vector. Undoing the geometric scaling, we bound the perturbations
to $c_{q}$ and $d_{q}$ by $\mathcal{O}(h^{p-q+1})$.

The accuracy of H-CMF of curve reconstruction directly follows from
Lemma~\ref{lem:accuracy-coefficients}. 
\begin{prop}
\label{prop:h-cmf-curve}Given a piecewise linear curve, whose vertices
approximate a smooth curve $\gamma$ with an error of $\mathcal{O}(h^{p+1})$
and the tangents are approximated to $\mathcal{O}(h^{p})$, assuming
the rescaled Vandermonde systems are well conditioned, the distance
between each point on the H-CMF reconstruction with degree-p fittings
and its closest point on $\gamma$ is $\mathcal{O}(h^{p+1})$.
\end{prop}
The result in Proposition~\ref{prop:h-cmf-curve} also holds for
CMF reconstruction. However, the accuracy of H-WALF curve reconstruction
is also bounded by $h^{6}$, similar to surface reconstructions.
\begin{prop}
\label{prop:curve }Under the same assumption as Proposition~\ref{prop:h-cmf-curve},
the distance between each point on the H-WALF reconstruction of a
smooth curve $\gamma$ with degree-p fittings and its closest point
on $\gamma$ is $\mathcal{O}(h^{p+1}+h^{6})$.
\end{prop}
The same result holds for WALF reconstruction. The bound of $h^{6}$
is due to the discrepancy of local coordinate systems at the two vertices
of an edge; we omit the proof here.

\subsection{$G^{0}$ Continuity of H-WALF Reconstruction}

It is clear that for smooth curves, $G^{0}$ continuity is guaranteed
by H-WALF (and WALF) curve reconstructions. For piecewise smooth curves,
some care must be taken. First, when selecting stencils for points
near a corner, we must make sure not to select points across a corner.
This can be done by virtually splitting the curves at the corners
into smooth segments, and then reconstructing each smooth segment
independently. Alternatively, we can use the safeguard $\theta^{+}$
in the weighting scheme to filter out points whose tangent directions
have a large angle against that at the origin. Second, at a corner
$\vec{x}_{0}$, we need to construct a local fitting within each edge
incident on $\vec{x}_{0}$ using the one-sided tangent as $\vec{s}_{0}$
for constructing the local frame. Third, if a corner has more than
two incident edges, we must enforce the local fit within each of its
incident edges to be interpolatory at $\vec{x}_{0}$, so that all
the reconstructed curves would meet at $\vec{x}_{0}$. Under this
construction, H-WALF can deliver accurate $G^{0}$ continuous reconstructions
for piecewise smooth curves.

\section{\label{sec:Feature-Aware-HOSR}Iterative Feature-Aware Parametric
Surfaces Reconstruction}

The preceding two sections focused on the reconstructions of smooth
surfaces and of feature curves on a piecewise smooth surface, respectively.
In this section, we combine the two techniques to reconstruct a piecewise
smooth surface to high-order accuracy with guaranteed $G^{0}$ continuity.
It is challenging to achieve both accuracy and continuity simultaneously.
For example, WALF (or H-WALF) does not guarantee continuity along
sharp features. This is because the reconstructed feature curves in
general do not match with the edges of the reconstructed surface patches
in their incident triangles. This is illustrated in Figure~\ref{fig:walf-disc},
where the curve reconstruction of feature edge $\boldsymbol{v}_{1}\boldsymbol{v}_{2}$
may have a different shape than the corresponding edge in the surface
reconstructions of triangles $\boldsymbol{v}_{1}\boldsymbol{v}_{2}\boldsymbol{v}_{3}$
and $\boldsymbol{v}_{1}\boldsymbol{v}_{2}\boldsymbol{v}_{4}$. A linear
combination of the reconstructed curve and the reconstructed surfaces
can recover continuity but may compromise the convergence rate. Similarly,
the parametric surface elements described in Section~\ref{subsec:High-Order-Parametric-Elements}
can recover continuity, but they may not deliver optimal convergence
rate near sharp features. In this section, we propose an iterative
procedure to construct parametric elements, which achieves both accuracy
and continuity. 
\begin{figure}
\begin{centering}
\includegraphics[width=0.75\textwidth]{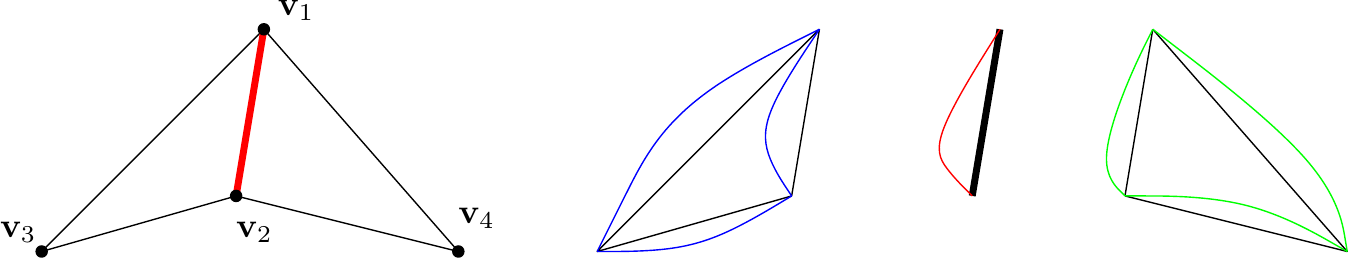}
\par\end{centering}
\caption{\label{fig:walf-disc}The lack of continuity between the WALF reconstructed
feature edge and the WALF-reconstructed surface patches in incident
triangles.}
\end{figure}

\subsection{\label{subsec:Smooth-parameterization}Accuracy and Stability of
Parametric Surfaces}

The parametric elements in Section~\ref{subsec:High-Order-Parametric-Elements}
provide a viable approach for reconstructing $G^{0}$ continuous surfaces.
The key issue is the placement of the mid-edge and mid-face nodes.
In general, this is a two-step procedure: first, define some intermediate
position $\vec{p}_{i}$ for each node; second, project $\vec{p}_{i}$
onto $\vec{x}_{i}$ using high-order reconstruction or onto the exact
surface if available. Both steps can affect the accuracy and the convergence
rate of the reconstructed surface. However, the importance of the
first step is more complicated and often overlooked. In the following,
we analyze the impact of both steps, with an emphasis on the first
step.

Consider a degree-$p$ element with $n$ nodes. Let $\vec{\xi}_{i}$
denote the natural coordinates of the $i$th node of the element.
Without loss of generality, we shall assume the nodes are composed
of LESGLS points. Let $\vec{\xi}$ be the natural coordinates of point
$\vec{x}(\vec{\xi})$, as defined in \eqref{eq:param-surf-element}.
Let $\Pi$ denote the projection from the intermediate points $\vec{p}_{i}$
onto $\vec{x}_{i}$, and then
\begin{equation}
\vec{x}(\vec{\xi})=\sum_{i=1}^{n}N_{i}(\vec{\xi})\vec{x}_{i}=\sum_{i=1}^{n}N_{i}(\vec{\xi})\Pi(\vec{p}_{i}).\label{eq:parametric-surface}
\end{equation}
For $\vec{x}(\vec{\xi})$ to be accurate, it is important that it
is smooth to degree $p+1$ in the following sense.
\begin{defn}
\label{def:smooth-param}The parameterization $\vec{x}(\vec{\xi})$
is \emph{smooth to degree $q$} if the partial derivatives of $\vec{x}$
with respect to $\vec{\xi}$ are uniformally bounded up to $q$th
order.
\end{defn}
The correlation of the smoothness and the accuracy of the parametric
surface is established by the following theorem.
\begin{thm}
\label{thm:error-smooth-param-surf} Given a smooth surface $\Gamma$
that is continuously differentiable to order $p+1$, assume the parameterization
in \eqref{eq:parametric-surface} is smooth to degree $p+1$, the
nodes $\vec{x}_{j}$ are reconstructed using degree-p H-CMF (or CMF),
and the interpolation over the parametric element is stable. The reconstructed
parametric surface approximates $\Gamma$ to $\mathcal{O}(h^{p+1})$,
where $h$ is a characteristic length measure of the local stencil.
\end{thm}
\begin{proof}
Let $\hat{\vec{x}}(\vec{\xi})$ denote the closest point to $\vec{x}(\vec{\xi})$
on $\Gamma$, and $\hat{\vec{x}}_{i}$ the closest point to $\vec{x}_{i}$
on $\Gamma$. Then, 
\begin{align}
\left\Vert \vec{x}(\vec{\xi})-\hat{\vec{x}}(\vec{\xi})\right\Vert  & =\left\Vert \sum_{i=1}^{n}N_{i}(\vec{\xi})\Pi(\vec{p}_{i})-\hat{\vec{x}}(\vec{\xi})\right\Vert \label{eq:err-param-surf1}\\
 & \leq\left\Vert \sum_{i=1}^{n}N_{i}(\vec{\xi})\hat{\vec{x}}_{i}-\hat{\vec{x}}(\vec{\xi})\right\Vert +\left\Vert \sum_{i=1}^{n}N_{i}(\vec{\xi})\left(\Pi(\vec{p}_{i})-\hat{\vec{x}}_{i}\right)\right\Vert .\label{eq:err-param-surf2}
\end{align}
With degree-$p$ CMF or H-CMF, the second term is $\mathcal{O}(h^{p+1})$,
because
\begin{equation}
\left\Vert \sum_{i=1}^{n}N_{i}(\vec{\xi})\left(\Pi(\vec{p}_{i})-\hat{\vec{x}}_{i}\right)\right\Vert \leq\left(\sum_{i=1}^{n}\left|N_{i}(\vec{\xi})\right|\right)\max_{i}\left\Vert \left(\Pi(\vec{p}_{i})-\hat{\vec{x}}_{i}\right)\right\Vert =\mathcal{O}(h^{p+1}),\label{eq:err-term2}
\end{equation}
and $\sum_{i=1}^{n}\left|N_{i}(\vec{\xi})\right|=\mathcal{O}(1)$
for stable elements. From the Taylor series, and in particular the
mean-value forms of its remainder, we can bound the first term in
\eqref{eq:err-param-surf2} by
\begin{equation}
\left\Vert \sum_{i=1}^{n}N_{i}(\vec{\xi})\hat{\vec{x}}_{i}-\hat{\vec{x}}(\vec{\xi})\right\Vert \leq\sum_{j,k\geq0}^{j+k=p+1}\left\Vert \frac{\partial^{p+1}\vec{x}}{\partial\xi^{j}\partial\eta^{k}}\right\Vert _{\infty}\left|\xi^{j}\eta^{k}\right|.\label{eq:bound-param-interp}
\end{equation}
Under the smoothness assumption, $\left\Vert \dfrac{\partial^{p+1}\vec{x}}{\partial\xi^{j}\partial\eta^{k}}\right\Vert _{\infty}=\mathcal{O}(1)$,
so $\left\Vert \vec{x}(\vec{\xi})-\hat{\vec{x}}(\vec{\xi})\right\Vert =\mathcal{O}(h^{p+1})$.
\end{proof}
Theorem~\ref{thm:error-smooth-param-surf} still holds if we replace
the H-CMF reconstruction with the exact surface, because the error
in the first term in \eqref{eq:err-param-surf2} is still bounded
by $\mathcal{O}(h^{p+1})$. We could also use H-WALF (or WALF) for
reconstruction, but it is undesirable because the error would be bounded
by $\mathcal{O}(h^{6})$, and it is less stable than H-CMF due to
the one-sided stencils near sharp features.

There are two key assumptions in Theorem~\ref{thm:error-smooth-param-surf}.
First, the interpolation must be stable, in that $\sum_{i=1}^{n}\left|N_{i}(\vec{\xi})\right|$
is bounded, ideally by a small constant. This may not hold for high-degree
elements with equally space nodes, but this assumption is valid for
elements with LESGLS nodes. Second, it assumes that $\vec{x}(\vec{\xi})$
is smooth to degree $p-1$ in order to bound the interpolation error
in \eqref{eq:bound-param-interp}. This assumption requires some special
attention in selecting the intermediate points $\vec{p}_{i}$ near
features, as we discuss next.

\subsection{\label{subsec:Smoothness-of-Parameterization}Smoothness of Parameterization
Near Features}

When reconstructing high-order surfaces from a surface triangulation,
a somewhat standard approach is to project the mid-edge and mid-face
nodes in the piecewise linear triangle onto the exact surface or a
high-order surface reconstruction. In other words, 
\begin{equation}
\vec{p}_{i}=\sum_{j=1}^{3}N_{j}^{(1)}(\vec{\xi}_{i})\vec{x}_{j}^{(1)},\qquad1\leq i\leq n\label{eq:intermediate-points-linear}
\end{equation}
serve as the intermediate points, where the $N_{j}^{(1)}$ denote
the shape functions of the linear elements and the $\vec{x}_{j}^{(1)}$
are the coordinates of the vertices of the triangle. Here, we shall
focus on the analysis of these intermediate points, and we will assume
that the projection $\vec{x}_{i}=\Pi\vec{p}_{i}$ is onto the exact
surface, so that the second term in \eqref{eq:err-param-surf2} is
0.

Near sharp features, the intermediate points in \eqref{eq:intermediate-points-linear}
can lead to nonsmooth parameterizations, if the projection $\Pi$
cause some abrupt contraction of the mid-edge and the mid-face nodes.
This can happen if the geometry is the union of two spheres that intersect
along a feature curve, as illustrated in Figure~\ref{fig:distorted-param}(a).
We refer to this feature curve as a \emph{bubble-junction curve},
because the union of the two spheres resemble the envelop of a double
bubble. Consider a triangle incident on the feature curve. In Figure~\ref{fig:distorted-param}(b),
we illustrate the projections of the mid-edge nodes of a degree-6
triangle along the left edge onto the exact circle, and the projections
of the other nodes onto the exact sphere. Due to the discontinuity
of the normal directions, the mid-edge nodes on the feature curve
and the adjacent mid-face nodes contract toward each other abruptly.
Figure~\ref{fig:distorted-param}(c) shows the contour plot of the
inverse area measure, i.e. $1\left/\sqrt{\vec{J}^{T}\vec{J}}\right.$,
where $\vec{J}$ denote the Jacobian matrix of $\vec{x}(\vec{\xi})$.
The inverse area measure is clearly much larger near the feature curve.
This non-uniformity can lead to large higher-order derivatives, and
it worsen as the degree of the polynomial increases. Hence, the convergence
rate using high-degree fittings may be compromised, as we will demonstrate
in Section~\ref{sec:Numerical-Results}. 
\begin{figure}
\subfloat[Double sphere.]{\centering{}\includegraphics[width=0.33\columnwidth]{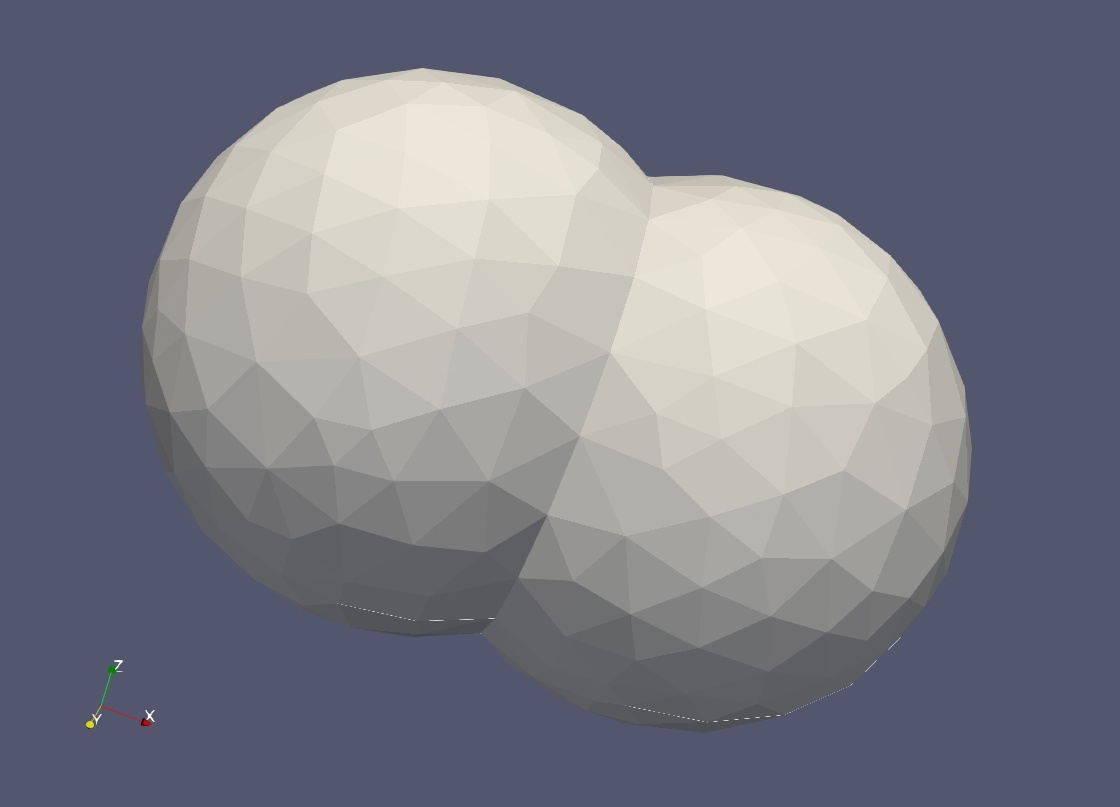}}\subfloat[Mid-edge and mid-face points.]{\centering{}\includegraphics[width=0.33\textwidth]{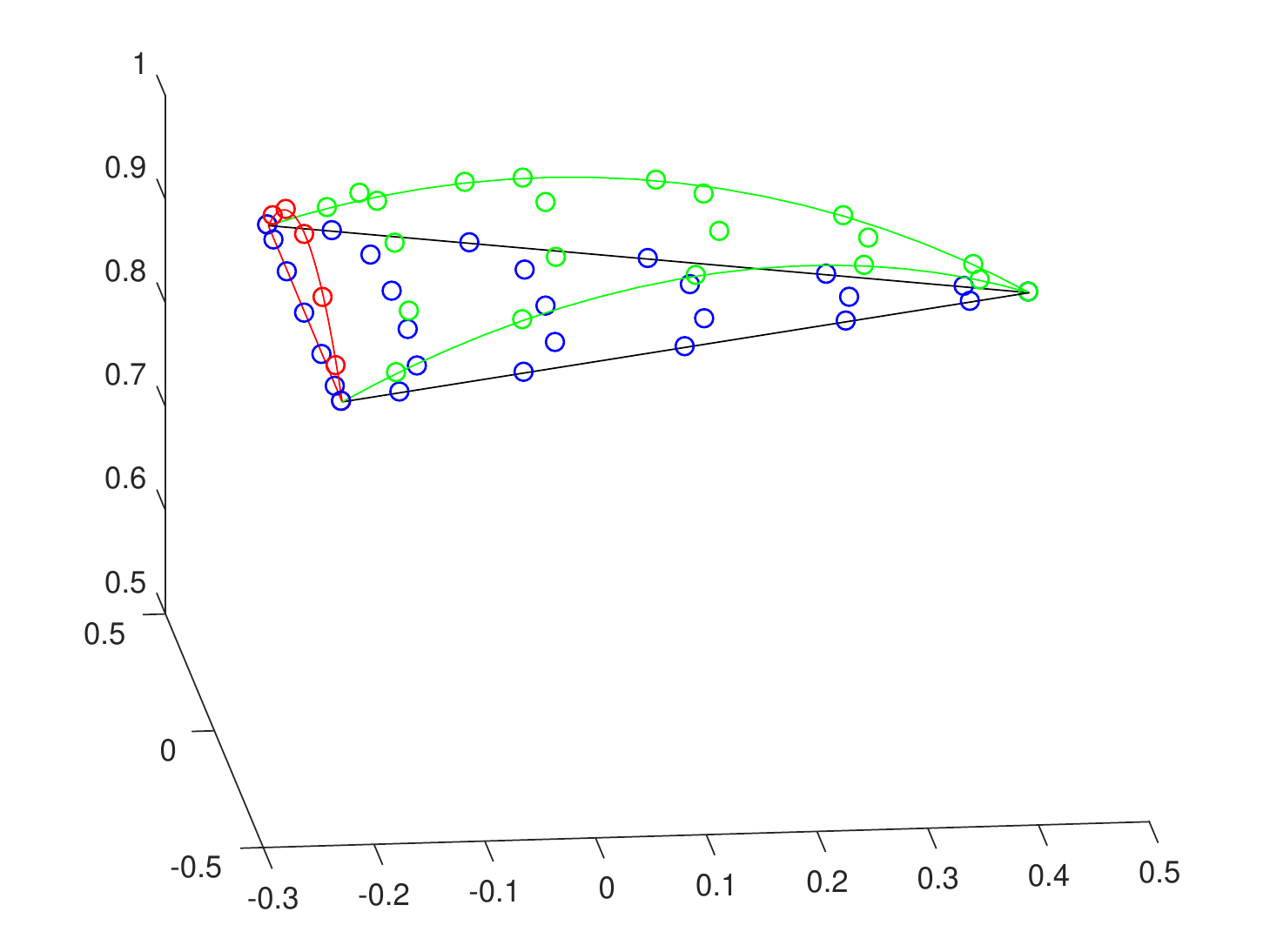}}\subfloat[Inverse area measure.]{\begin{centering}
\includegraphics[width=0.33\textwidth]{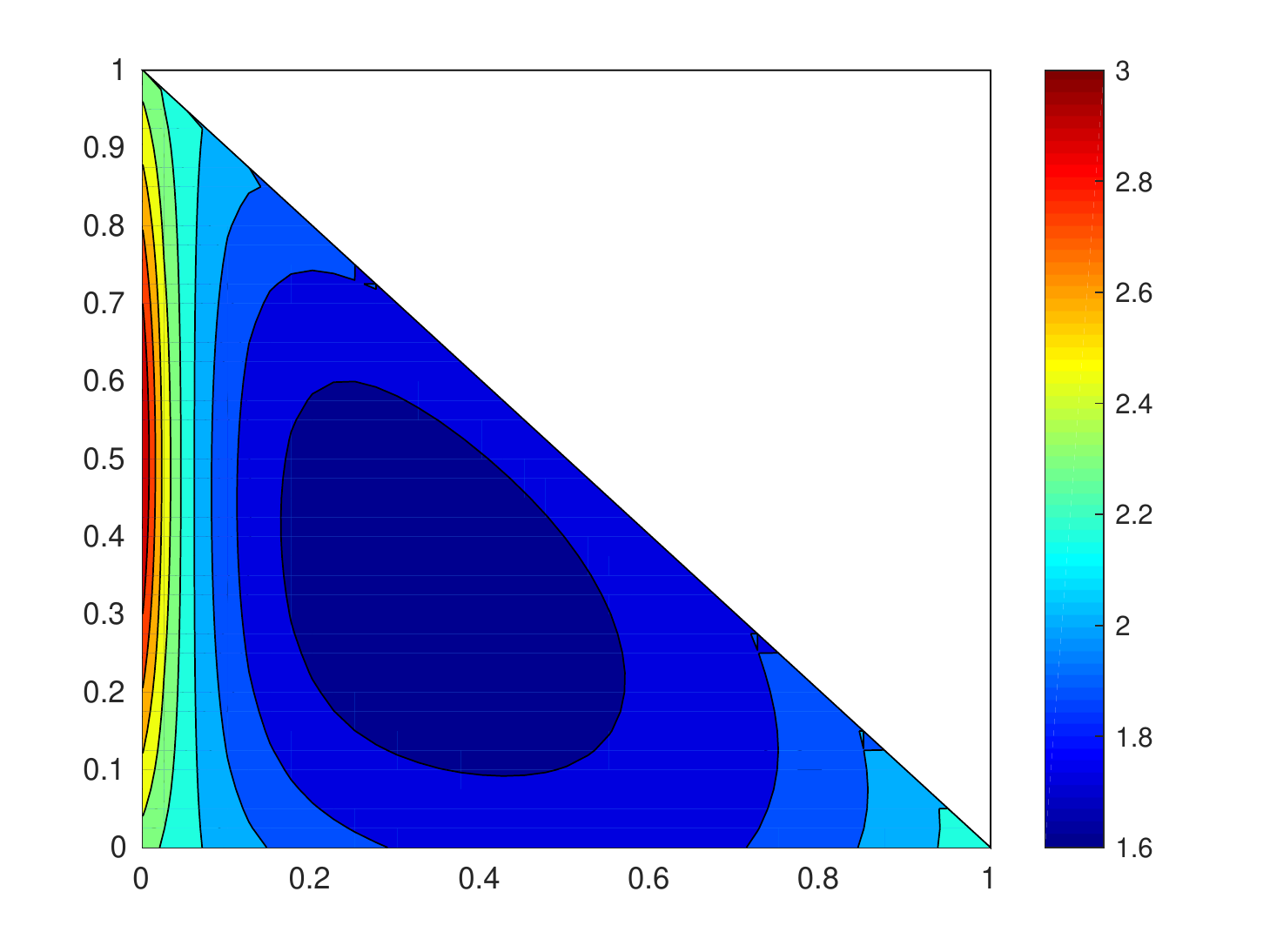}
\par\end{centering}
}\caption{\label{fig:distorted-param}Nonsmooth parameterization near a bubble-junction
feature.}
\end{figure}

One might attempt to improve the smoothness of $\vec{x}(\vec{\xi})$
by using some optimization procedure, but it would be difficult because
$\Pi$ is discontinuous near sharp features and ensuring high-degree
continuity may require the consideration of high-order derivatives
in the objective function. Instead, we can improve the smoothness
by an explicit construction of the intermediate points $\vec{p}_{i}$
using intermediate-degree polynomial interpolation. Specifically,
let 
\begin{equation}
\vec{p}_{i}=\sum_{j=1}^{n^{(q)}}N_{j}^{(q)}(\vec{\xi}_{i})\vec{x}_{j}^{(q)},\qquad1\leq i\leq n,\label{eq:intermediate-nodes-q}
\end{equation}
where $N_{j}^{(q)}$ and $n^{(q)}$ are the shape functions and the
number of nodes of degree-$q$ elements, respectively, where $1<q<p$,
and $\vec{x}_{j}^{(q)}$ are the nodes of degree-$q$ elements that
have taken into account the curved feature edges. For example, in
the double-sphere example above, we use a quadratic element (i.e.,
$q=2$) to construct the intermediate points. Figure~\ref{fig:smoother-param}(a)
shows the intermediate positions of the quadratic element, and Figure~\ref{fig:smoother-param}(b)
shows the projection of these intermediate points onto the exact curve
and surface. From the contour plot of the inverse area measure in
Figure~\ref{fig:smoother-param}(c), it is clear that this new parameterization
is much smoother than that in Figure~\ref{fig:distorted-param}.
We can apply this idea iteratively to achieve higher-degree smoothness,
as we describe next.

\begin{figure}
\subfloat[Intermediate points.]{\includegraphics[width=0.33\textwidth]{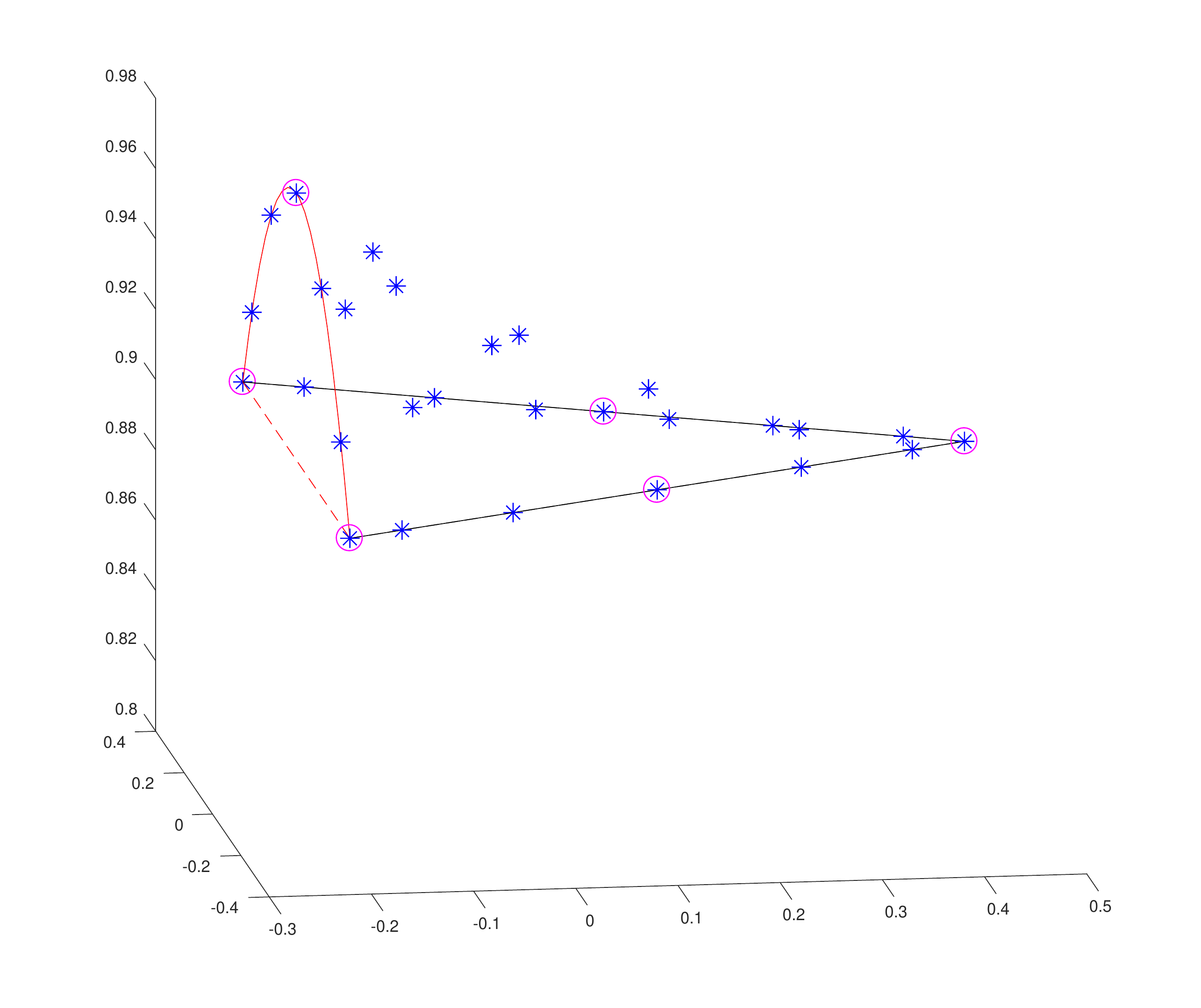}}\subfloat[Projected points.]{\includegraphics[width=0.33\textwidth]{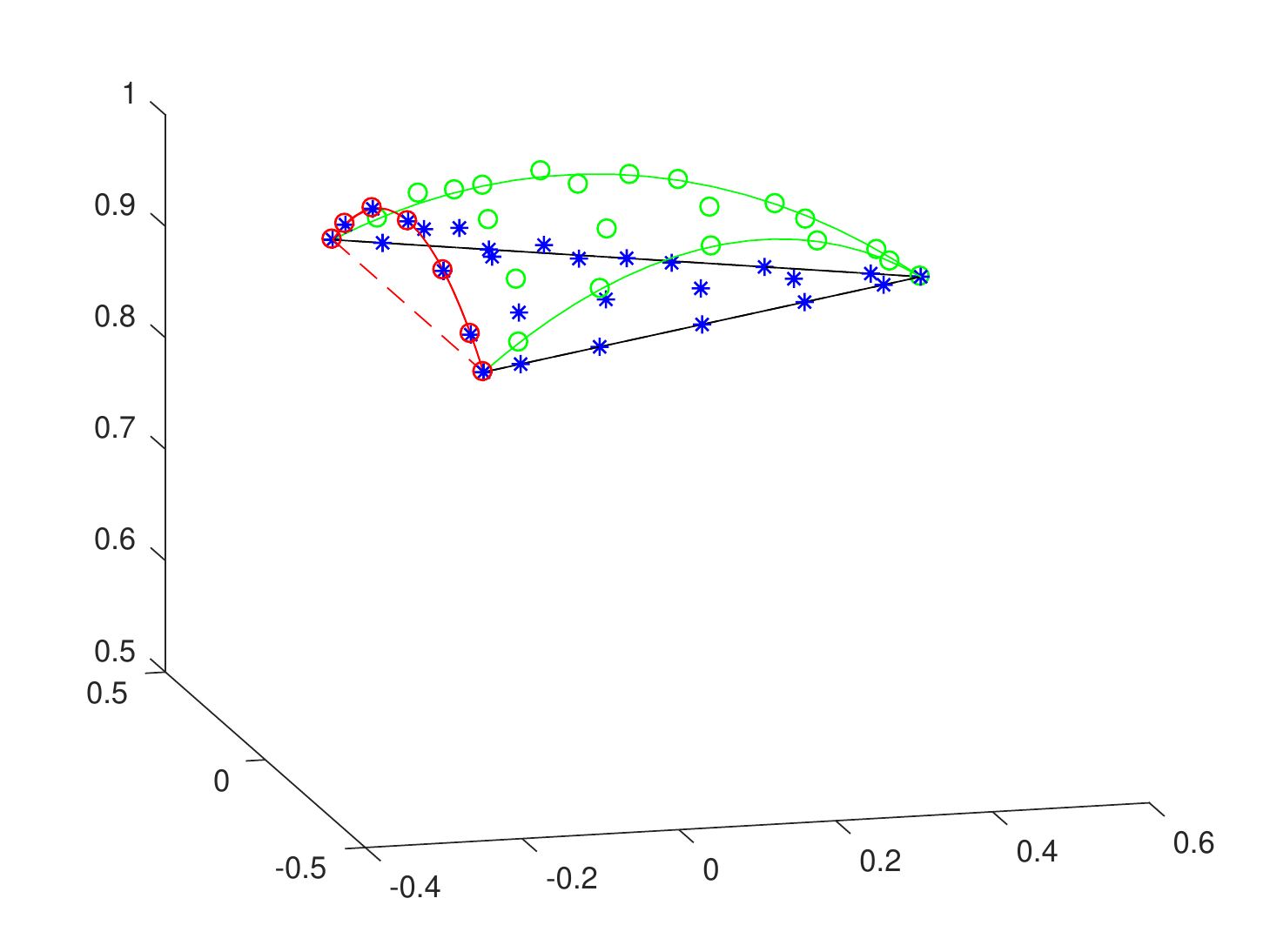}}\subfloat[Inverse area measure.]{\includegraphics[width=0.33\textwidth]{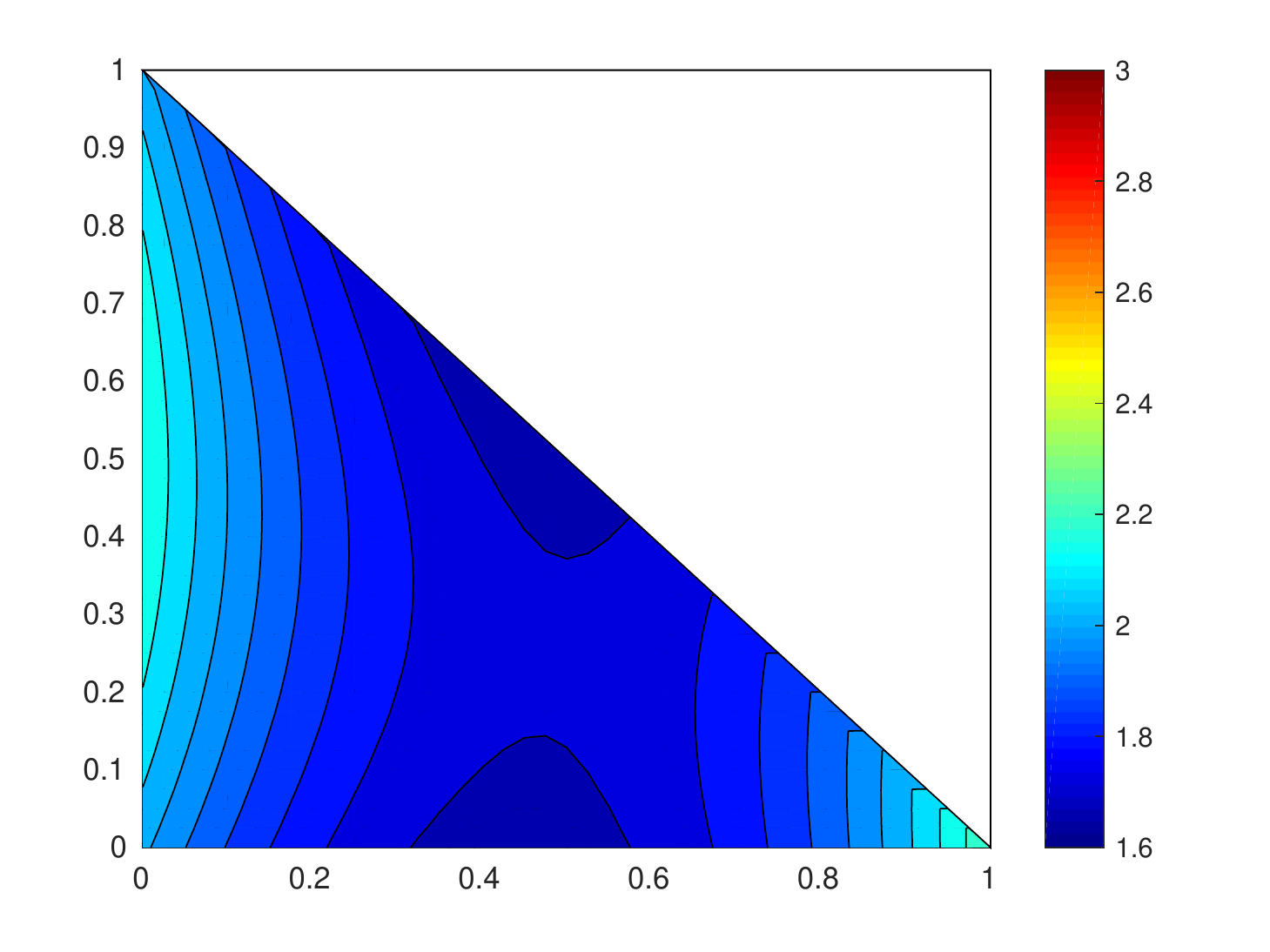}

}\caption{\label{fig:smoother-param}Smoother parameterization by constructing
intermediate points using quadratic elements.}
\end{figure}

\subsection{\label{subsec:Iterative-Feature-Aware-Paramete}Iterative Feature-Aware
Parameterization}

The two-level construction in Figure~\ref{fig:smoother-param} improves
the uniformity of the area measure, which contain information of only
first-order derivatives. To achieve higher-degree smoothness, we use
multiple levels of intermediate nodes. Specifically, for each triangle
incident on a feature (boundary) curve, we use a degree-$q$ interpolation
to define its intermediate nodes, where 
\begin{equation}
q=2^{\left\lceil \log p\right\rceil -1}.\label{eq:iterative-degree-reduction}
\end{equation}
The nodes $\vec{x}_{j}^{(q)}$ for the degree-$q$ element in \eqref{eq:intermediate-nodes-q},
especially the mid-edge nodes on feature edges and the mid-face nodes,
are computed recursively using degree $2^{\left\lceil \log q\right\rceil -1}$
interpolation. We refer to this procedure as \emph{Iterative Feature
Aware (IFA) parameterization}. The exponential decrease of the degree
is due to the observation that the projection from the intermediate
nodes from degree-$q$ interpolation introduces an $\mathcal{O}(h^{2q+2})$
perturbation to the parameterization.

In summary, the overall algorithm for constructing a degree-$p$ parametric
surface proceeds as follows. First, use degree-$p$ H-CMF (or H-WALF
if $p\leq4$) to compute the mid-edge and mid-face nodes of degree-$p$
elements for all the triangles without a feature or border edge. Then
for each element with a feature or border edge, we apply IFA parameterization
to obtain its mid-edge and mid-face nodes. Finally, we use the nodes
in the original input mesh along the mid-edge and mid-face nodes to
define a $G^{0}$ continuous degree-$p$ parametric surface of guaranteed
$(p+1)$st order accuracy.

An important application of this parametric surface is high-order
finite element methods. In that setting, it is also important for
the elements in the volume mesh (i.e., the tetrahedra) to have smooth
parameterizations near boundaries. The IFA parameterization described
here can be adapted to placing the mid-face and mid-cell nodes in
these tetrahedra, which can improve the accuracy of FEM, as we demonstrate
in Section~\ref{subsec:FEM}. These IFA parameterizations may seem
expensive for individual elements. However, this extra cost is negligible,
because the number of elements incident on sharp features or the boundary
is lower order compared to the total number of elements, due to the
surface-to-volume ratio.

\section{\label{sec:Numerical-Results}Numerical Results}

In this section, we assess the proposed high-order reconstruction
numerically, with a focus on the improved accuracy due to the Hermite-style
fittings as well as iterative feature-aware parameterization. In addition,
we also evaluate the effectiveness of the high-order reconstruction
as an alternative of the exact geometry for high-order FEM. We will
evaluate the methods using two surfaces, including the double sphere
in Figure~\ref{fig:distorted-param}(a) and a torus in Figure~\ref{fig:test-geometries}(a).
To evaluate curve reconstruction, we consider a conical helix in Figure~\ref{fig:test-geometries}(b),
with the parametric equations 
\[
\vec{r}(t)=(t\cos(6t),t\sin(6t),t),\:(0\leqslant t\leqslant2\pi).
\]
Although simple, these geometries are representative of piecewise
surfaces with different curvatures and sharp features.

To evaluate the convergence rates, we generate a series of triangular
meshes for each geometry using mesh refinement and compute the pointwise
error as the distance between a reconstructed point and its closest
point on the exact surface or curve. Given an error vector $\vec{e}$
with $n$ points, we use a normalized $l_{2}$-norm of $\vec{e}$,
computed as the standard vector $2$-norm by $\sqrt{n}$, i.e.,
\begin{equation}
\left\Vert \vec{e}\right\Vert _{\ell_{2}}=\frac{\left\Vert \vec{e}\right\Vert _{2}}{\sqrt{n}}=\sqrt{\frac{\sum_{i=1}^{n}e_{i}^{2}}{n}}.\label{eq:l2-pointwise}
\end{equation}
Given a series of $k$ meshes, let $\vec{e}_{i}$ denote the error
vector on the $i$th meshes and $n_{i}$ denote the number of points
in the $i$th mesh, where level-1 denotes the coarsest mesh. The average
convergence rate in $\ell_{2}$-norm is then computed as 
\[
\text{convergence rate}=d\frac{\log(\left\Vert \vec{e}_{1}\right\Vert _{\ell_{2}}/\left\Vert \vec{e}_{k}\right\Vert _{\ell_{2}})}{\log(n_{k}/n_{1})},
\]
where $d=1$, 2, and $3$ for curve, surface, and volume meshes, respectively.

\begin{figure}
\begin{centering}
\subfloat[\label{fig:Torus}Torus.]{\begin{centering}
\includegraphics[width=0.5\textwidth]{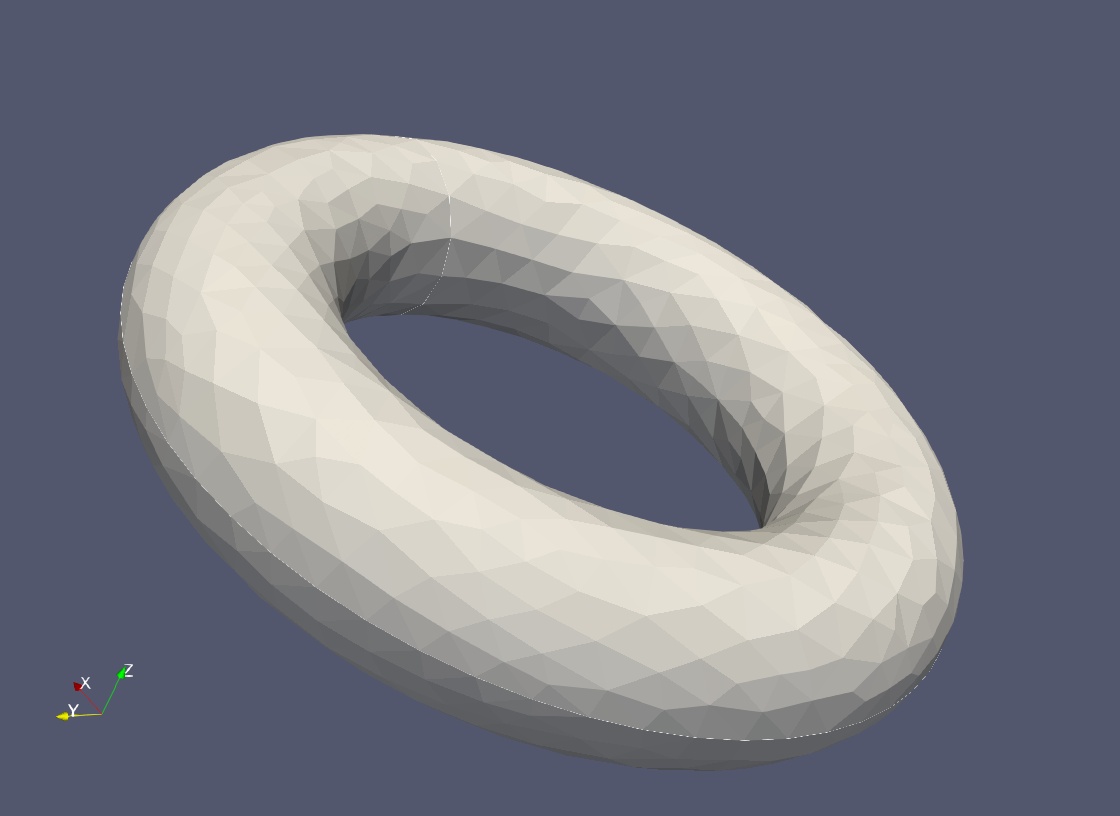}
\par\end{centering}
}\subfloat[\label{fig:coarsest-conical-helix}Conical helix.]{\begin{centering}
\includegraphics[width=0.5\textwidth]{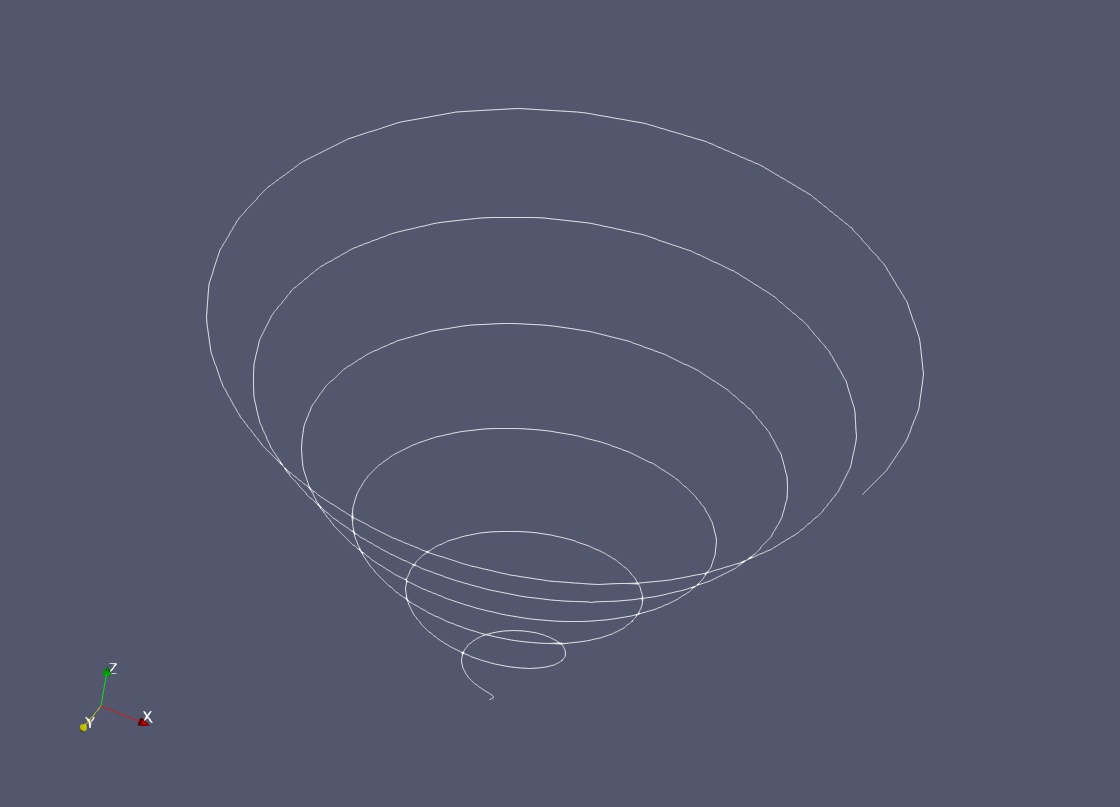}
\par\end{centering}
}
\par\end{centering}
\caption{\label{fig:test-geometries}Test geometries for convergence studies.}
\end{figure}

\subsection{Point-Based vs. Hermite-Style Reconstructions}

We first assess the accuracy of point-based and Hermite-style surface
reconstructions. As a base case, we use the double-sphere geometry,
which was obtained by intersecting two unit spheres centered at the
origin and at $[0.5,0.0]^{T}$, respectively. We generated three meshes
using Gmsh \cite{Geuzaine2009gmsh} with 351, 1,402 and 5,598 vertices,
respectively. We evaluate the convergence of CMF, WALF, H-CMF and
H-WALF of degrees 2--7 with IFA parameterization. Figure~\ref{fig:Comparison-between-odd}
plots the $l_{2}$-norm of pointwise errors sampled at degree-6 Gaussian
quadrature points, where the numbers to the right of the plots are
the average convergence rates.

\begin{figure}
\subfloat[]{\includegraphics[width=0.49\textwidth]{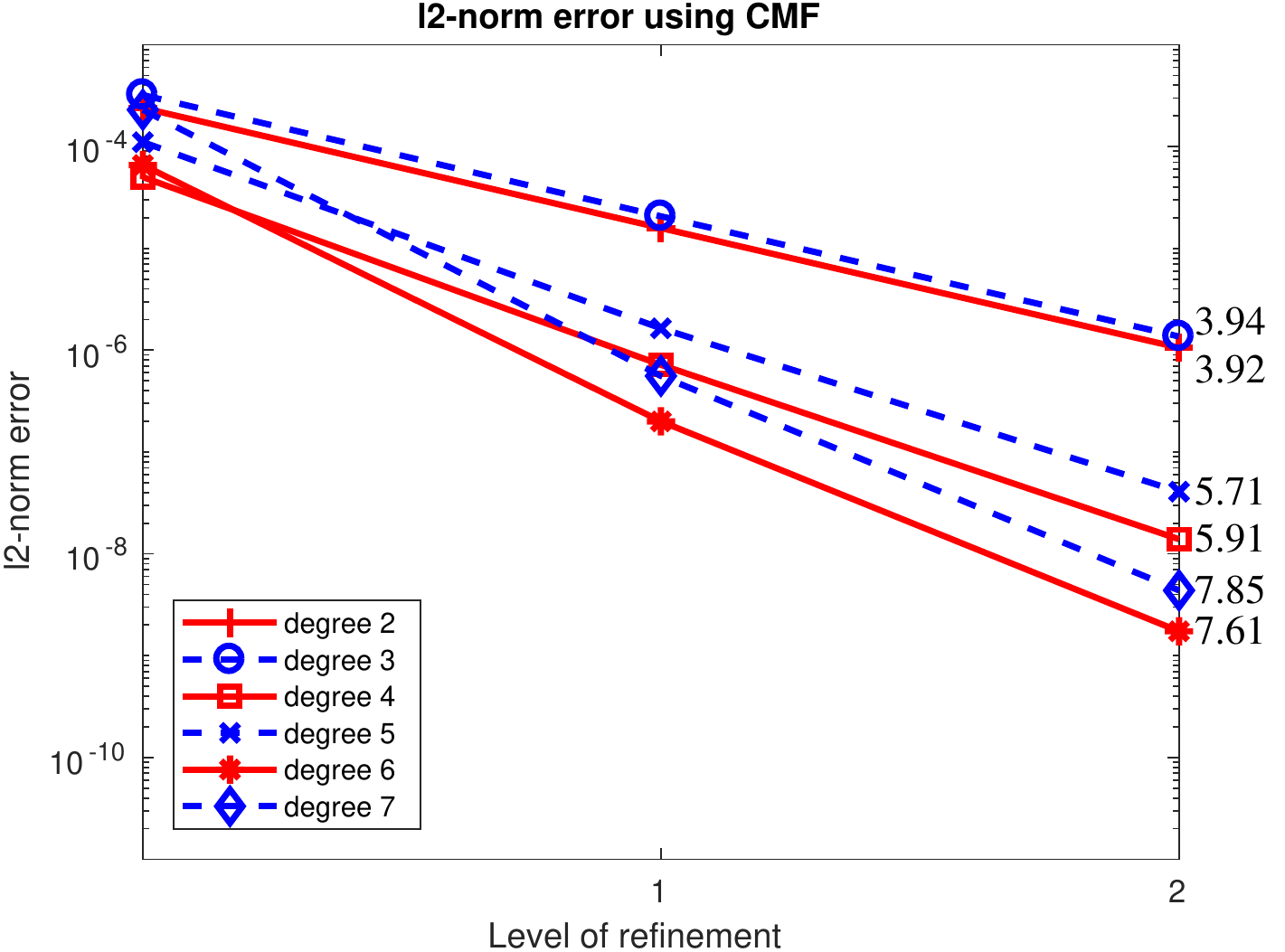}}\subfloat[]{\includegraphics[width=0.49\textwidth]{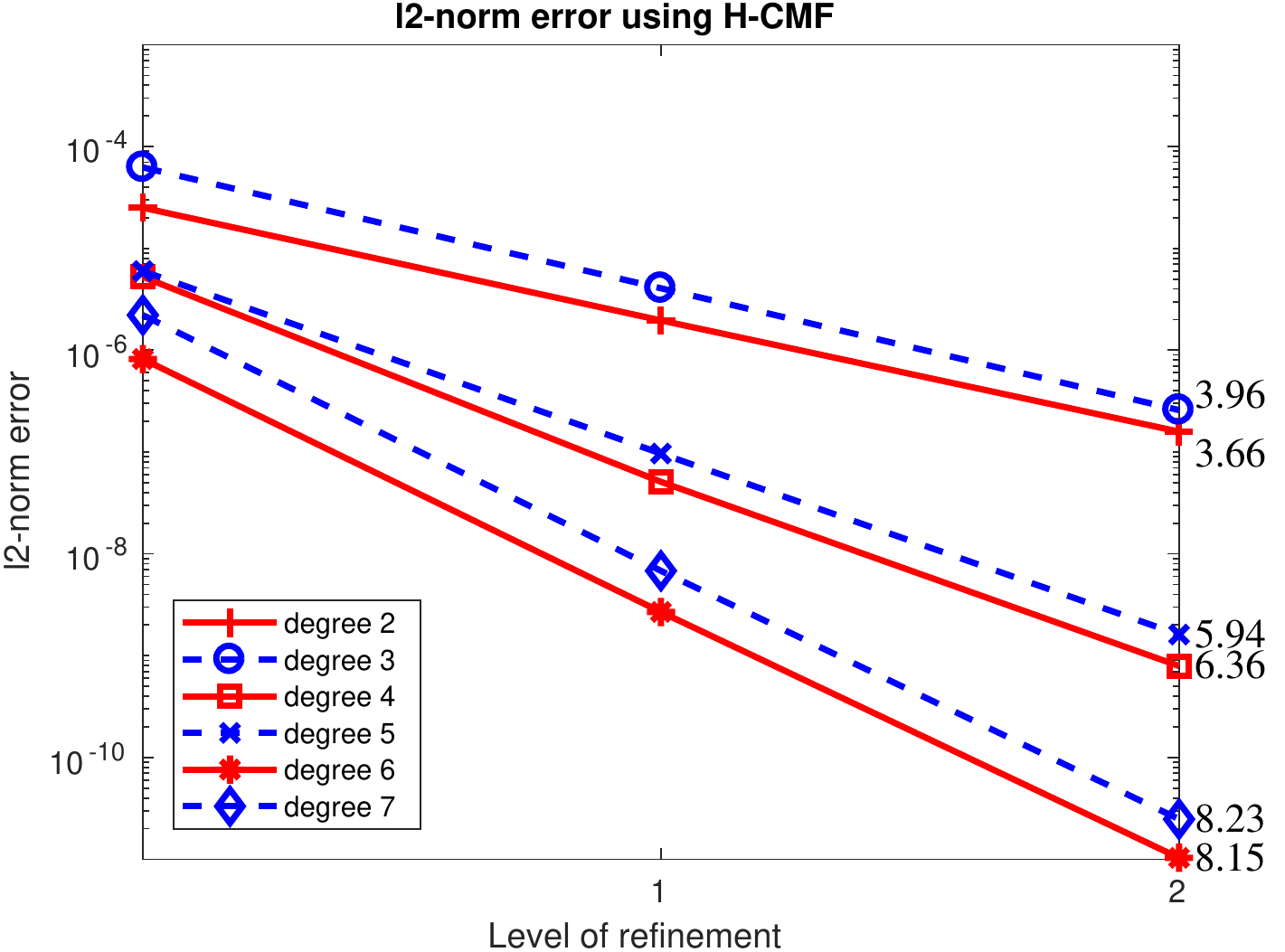}}

\subfloat[]{\includegraphics[width=0.49\textwidth]{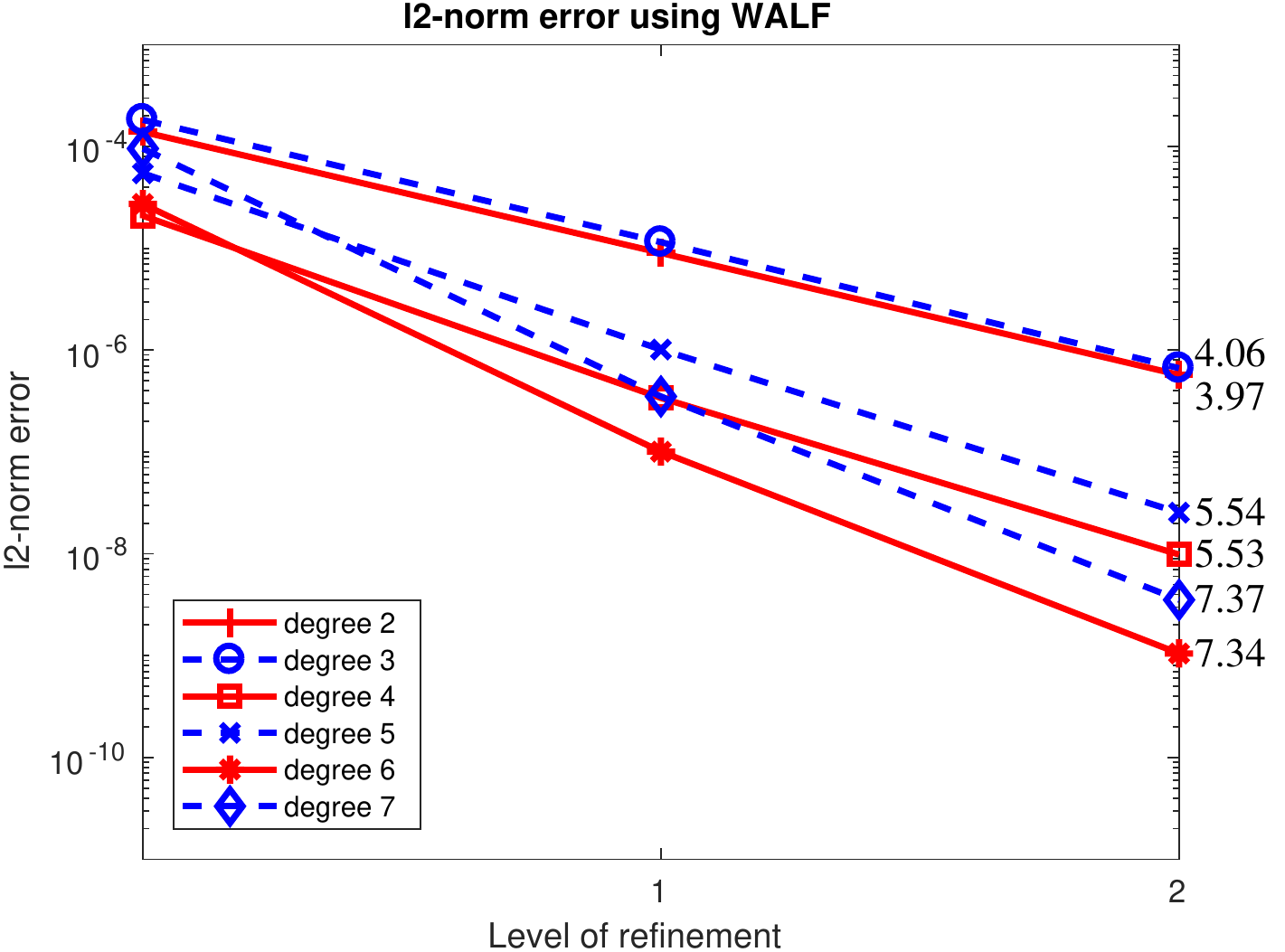}

}\subfloat[]{\includegraphics[width=0.49\textwidth]{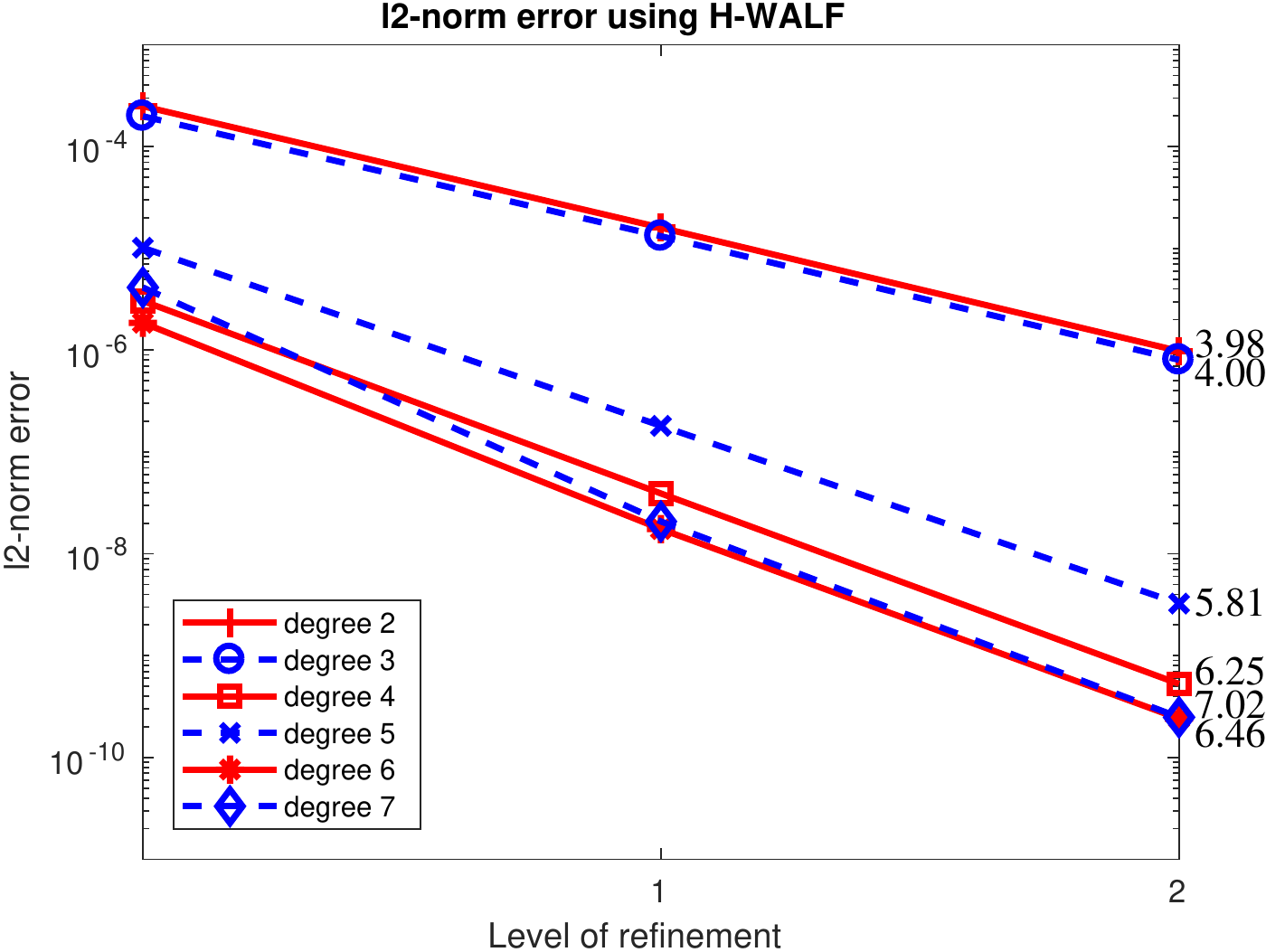}

}

\caption{\label{fig:Comparison-between-odd}Comparison of point-based and Hermite-style
surface reconstructions on double sphere.}
\end{figure}

We make a few observations about this result. First and foremost,
for all the cases, the degree-$2q$ fittings are more accurate than
the degree-$(2q+1)$ fittings. This is because the leading error terms
of the even-degree polynomial fittings are odd degrees, which can
cancel out for nearly symmetric meshes and nearly symmetric geometries.
This leads to superconvergence for even-degree reconstructions. Furthermore,
the degree-$2q$ fittings require smaller stencils than degree-$(2q+1)$
fittings, so they are both more accurate and more efficient. Hence,
in the following tests we will consider only even-degree fittings.

Second, the Hermite-style reconstructions produced much smaller errors
than their point-based counterparts with quartic and sextic polynomials,
due to the more compact stencils and hence smaller constant factors
in the errors. This demonstrates the benefits of Hermite-style reconstruction.
These benefits are even more pronounced on coarsest meshes. Third,
H-CMF and H-WALF had comparable accuracy for degree-4 fittings, but
H-CMF significantly outperformed H-WALF for degree-6 fittings. Hence,
H-CMF is preferred for sextic or higher-degree fittings for its superior
accuracy, but H-WALF is preferred for quartic or lower-degree fittings
for its comparable accuracy and better efficiency.

The above observations are consistent with our theoretical analysis
in the preceding sections. We can also draw similar conclusions for
surface reconstruction of nonuniform geometrics (such as a torus)
as well as curve reconstructions (such as for a helix). Figure~\ref{fig:convergence-torus}
shows the convergence results of surface reconstruction for the torus,
where the three meshes have 898, 3,592, and 14,368 vertices, respectively.
Figure~\ref{fig:convergence-helix} shows the convergence results
of curve reconstruction for the helix, where the three meshes have
256, 512, and 1,024 vertices, respectively. It is clear that (1) even-degree
fittings enjoyed superconvergence, (2) Hermite-style fittings outperformed
their point-based counterparts for quartic and sextic fittings by
one to two orders of magnitude, and (3) H-CMF significantly outperforms
H-WALF for degree-6 reconstructions. However, note that point-based
and Hermite-style reconstructions with quadratic polynomials had comparable
results for smooth surfaces, because they have similar ring sizes.
This behavior is different from that for the double sphere in Figure~\ref{fig:Comparison-between-odd},
for which the Hermite-style reconstruction enabled smaller one-sided
stencils near sharp features and hence better accuracy even for quadratic
polynomials.

\begin{figure}
\subfloat[]{\includegraphics[width=0.5\textwidth]{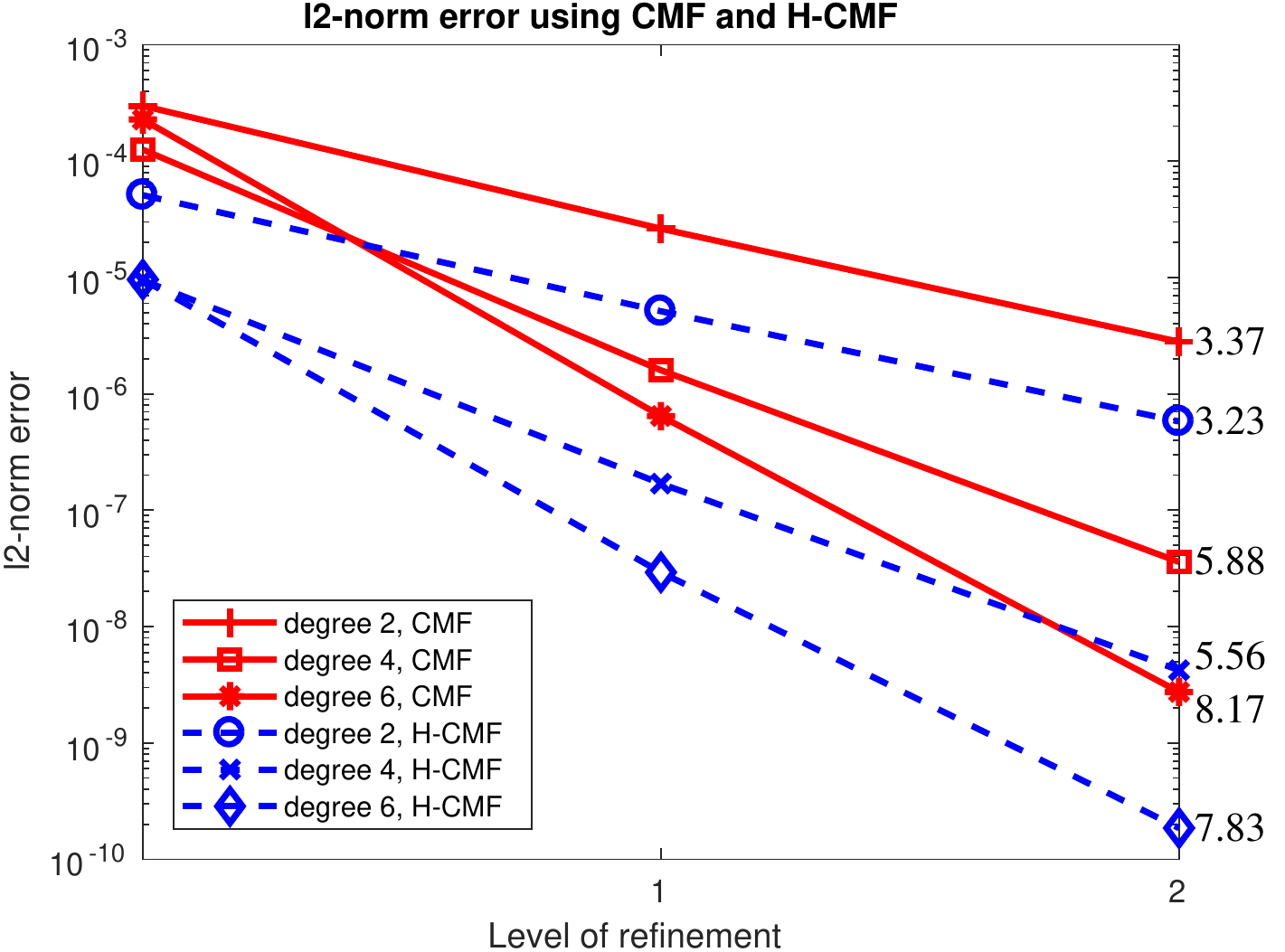}

}\subfloat[]{\includegraphics[width=0.5\textwidth]{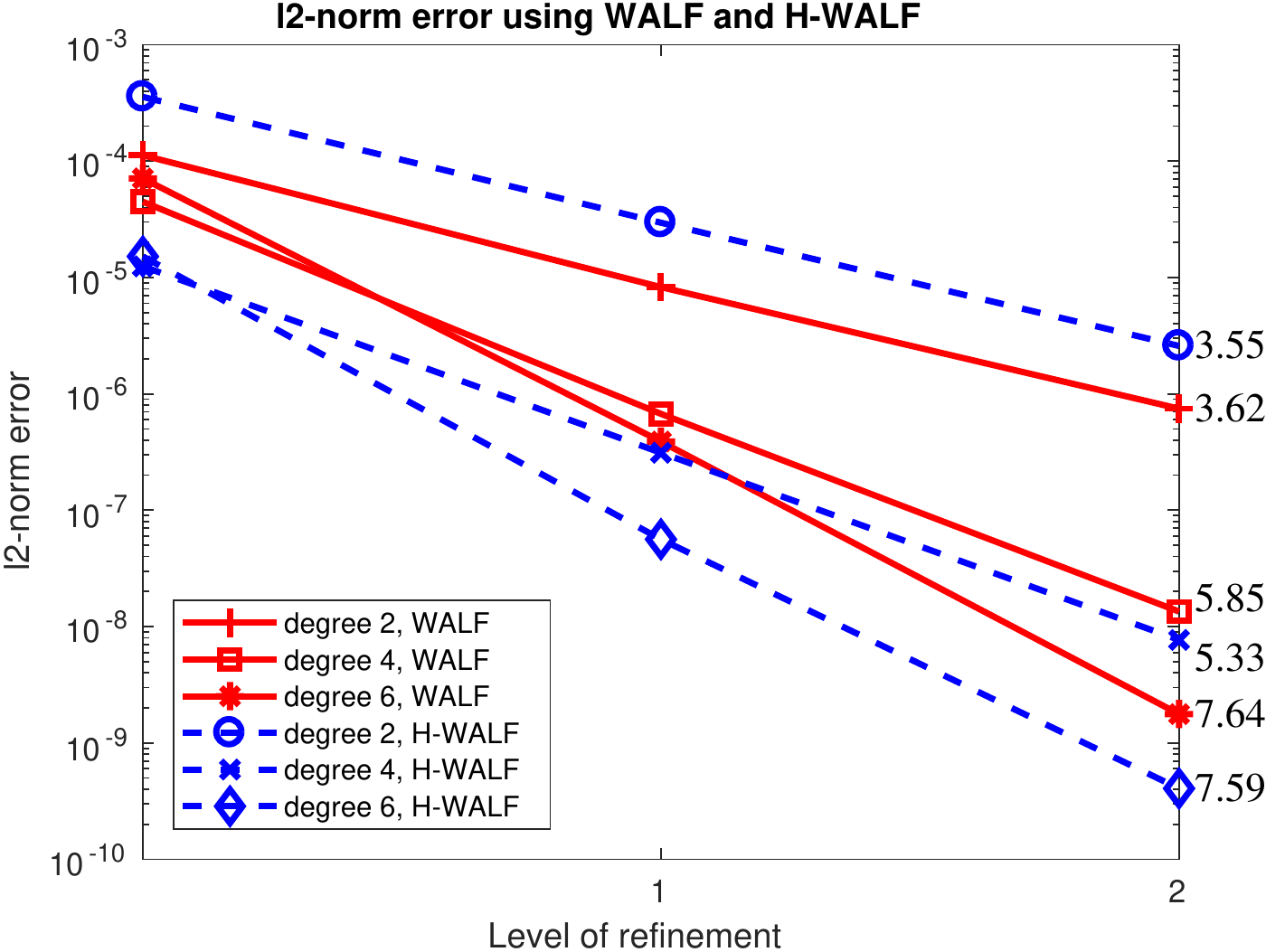}

}

\caption{\label{fig:convergence-torus}Comparison of point-based and Hermite-style
surface reconstructions on the torus.}
\end{figure}

\begin{figure}
\begin{centering}
\subfloat[]{\begin{centering}
\includegraphics[width=0.5\textwidth]{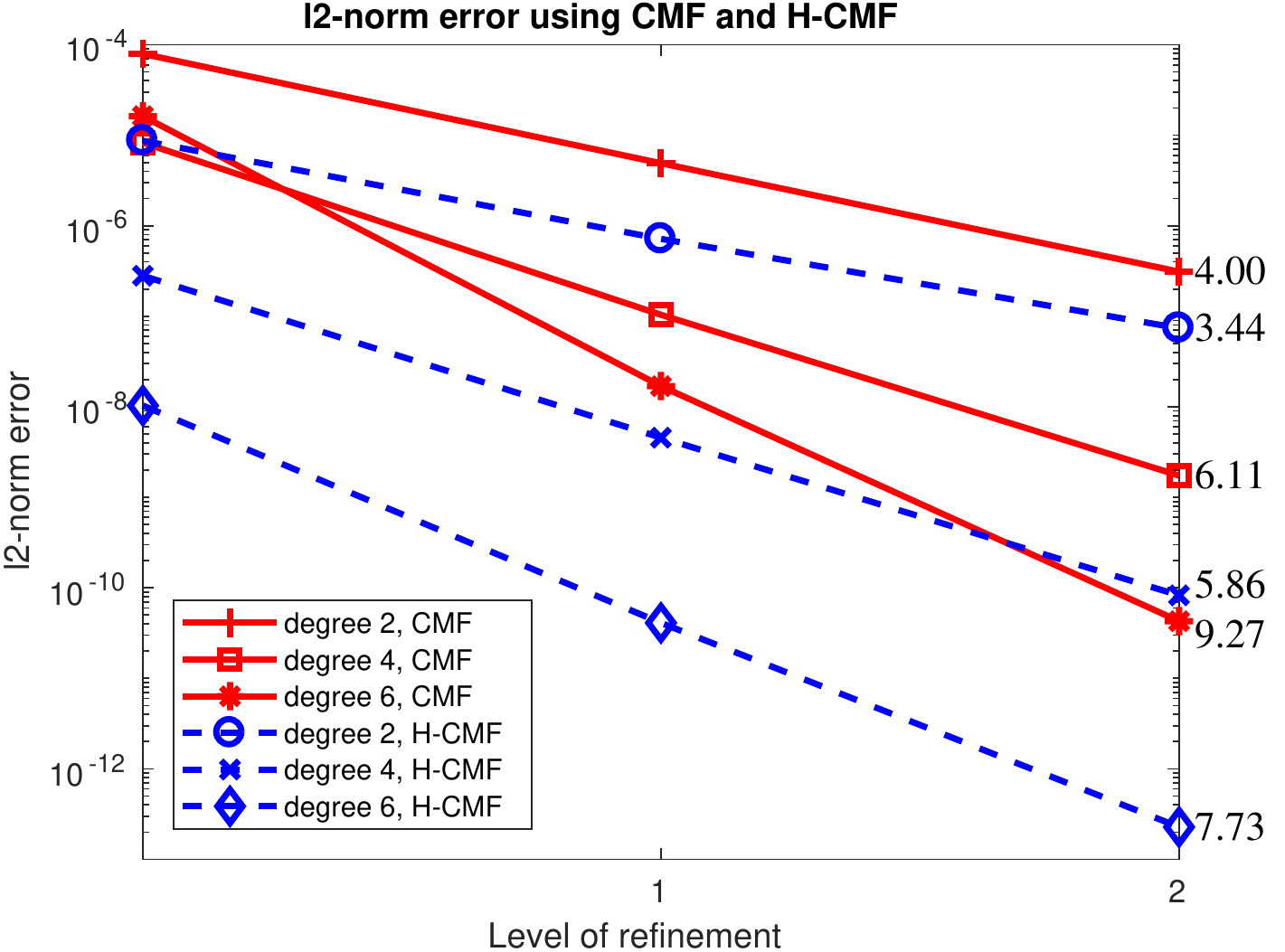}
\par\end{centering}
}\subfloat[]{\begin{centering}
\includegraphics[width=0.5\textwidth]{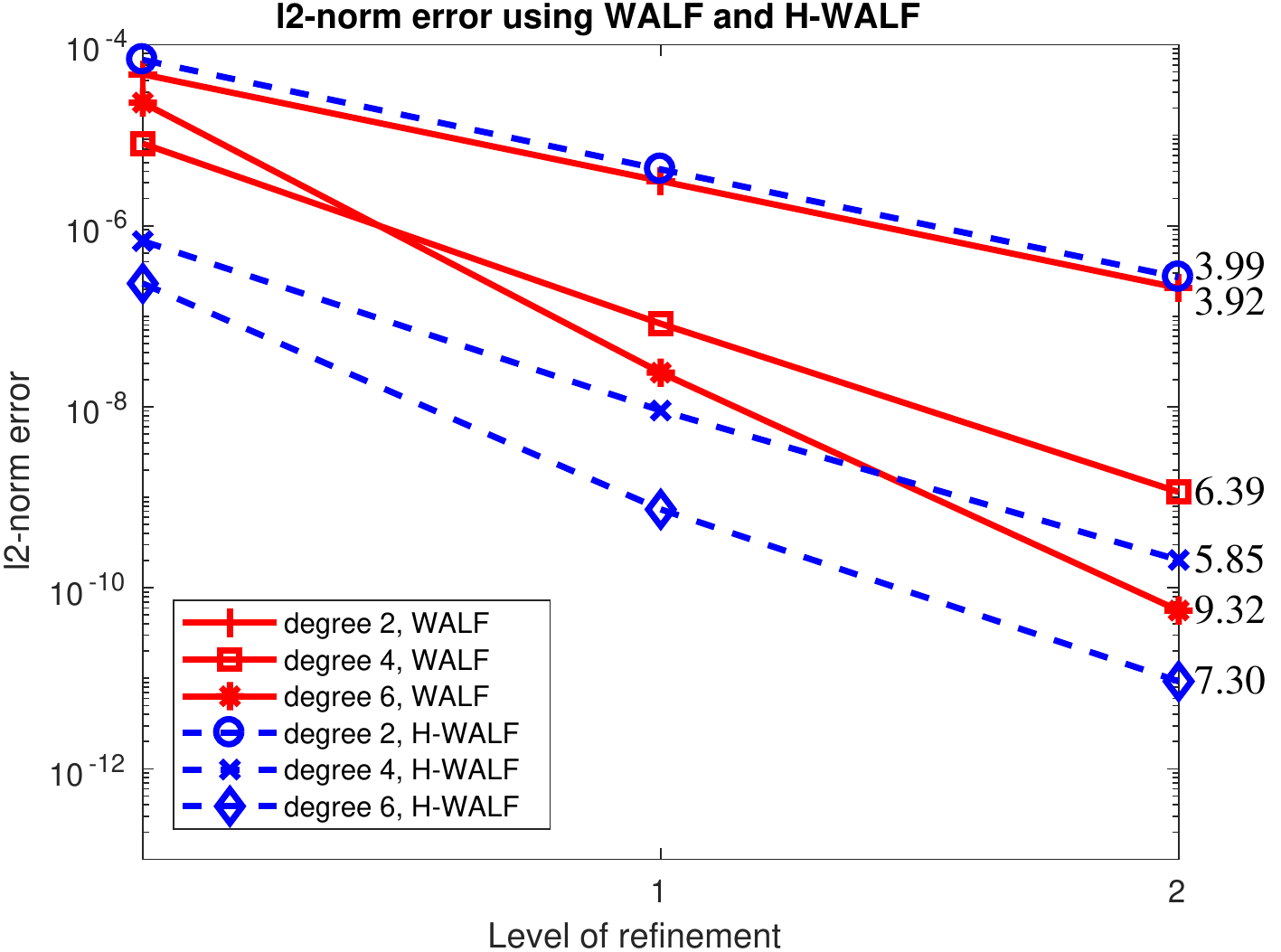}
\par\end{centering}
}
\par\end{centering}
\caption{\label{fig:convergence-helix}Comparison of point-based and Hermite-style
curve reconstructions for the conical helix.}
\end{figure}

\subsection{Benefits of IFA Parameterizations.}

In Section~\ref{subsec:Iterative-Feature-Aware-Paramete}, we introduced
the iterative feature-aware (IFA) parameterization. To demonstrate
its effectiveness, we compare the reconstruction of the double-sphere
and the half-sphere geometries using CMF and H-CMF, with three different
parameterizations:
\begin{itemize}
\item Non-FAP: Projecting the mid-edge and mid-face nodes from the linear
triangle, as illustrated in Figure~\ref{fig:distorted-param};
\item FAP: Using quadratic element to construct intermediate nodes, as illustrated
in Figure~\ref{fig:smoother-param};
\item IFAP: Using multiple levels of intermediate nodes with exponential
growth of the degree, as described in Section~~\ref{subsec:Iterative-Feature-Aware-Paramete}.
\end{itemize}
Since our focus is for feature awareness, we consider only the elements
incident on sharp features when computing the $l_{2}$-norm errors.
Figure~\ref{fig:compare-nearfea-treatment} shows the convergence
rates of degree-6 CMF and H-CMF. To demonstrate the benefit of IFAP
for higher-degree fittings, we also show the convergence results with
degree-8 CMF and H-CMF for the two finer meshes. It is clear that
FAP and IFAP outperformed non-FAP in all cases. IFAP outperformed
FAP significantly for degree-8 H-CMF, although they performed similarly
for the other cases.

\begin{figure}
\begin{centering}
\subfloat[]{\centering{}\includegraphics[width=0.5\textwidth]{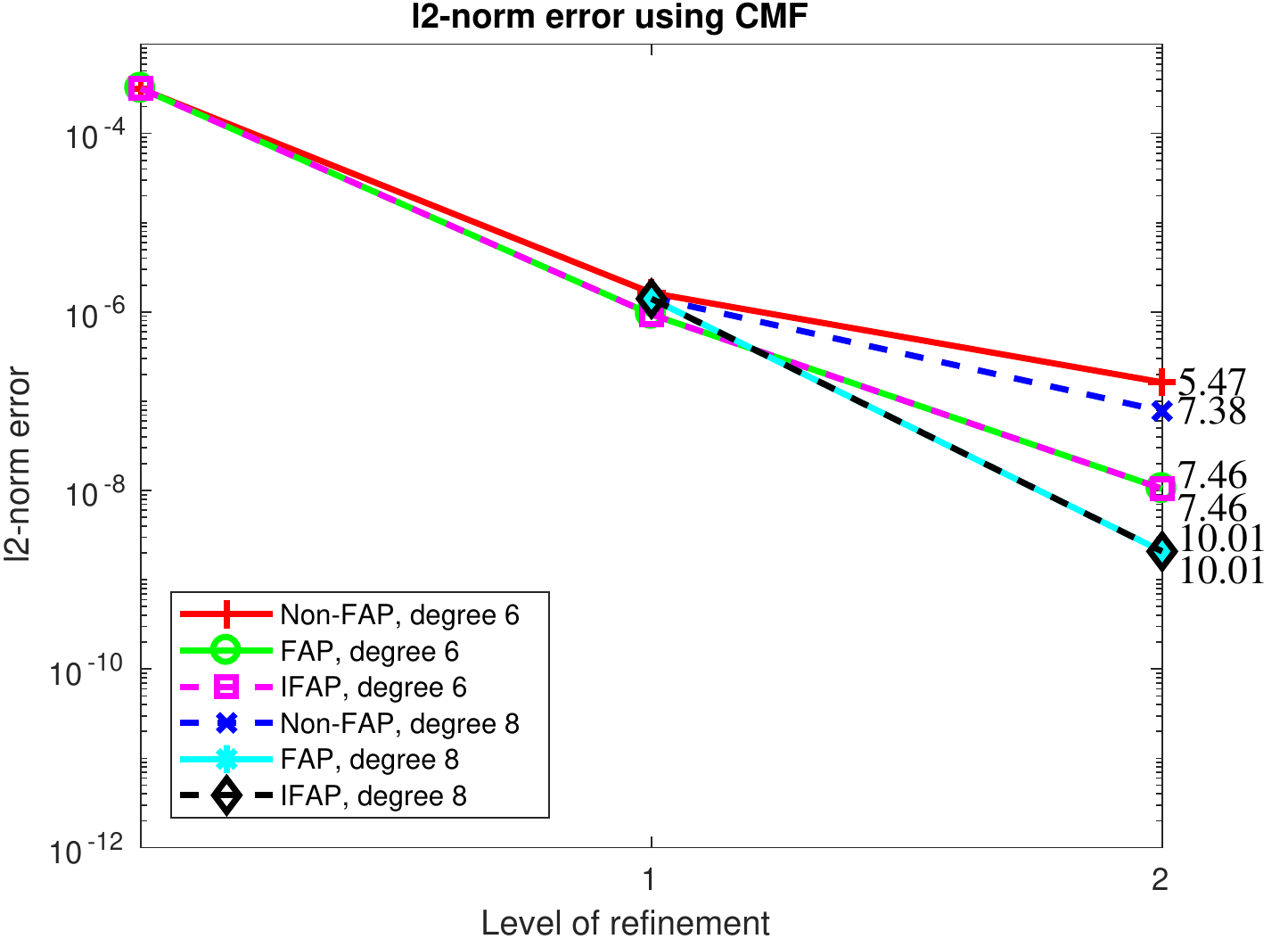}}\subfloat[]{\centering{}\includegraphics[width=0.5\textwidth]{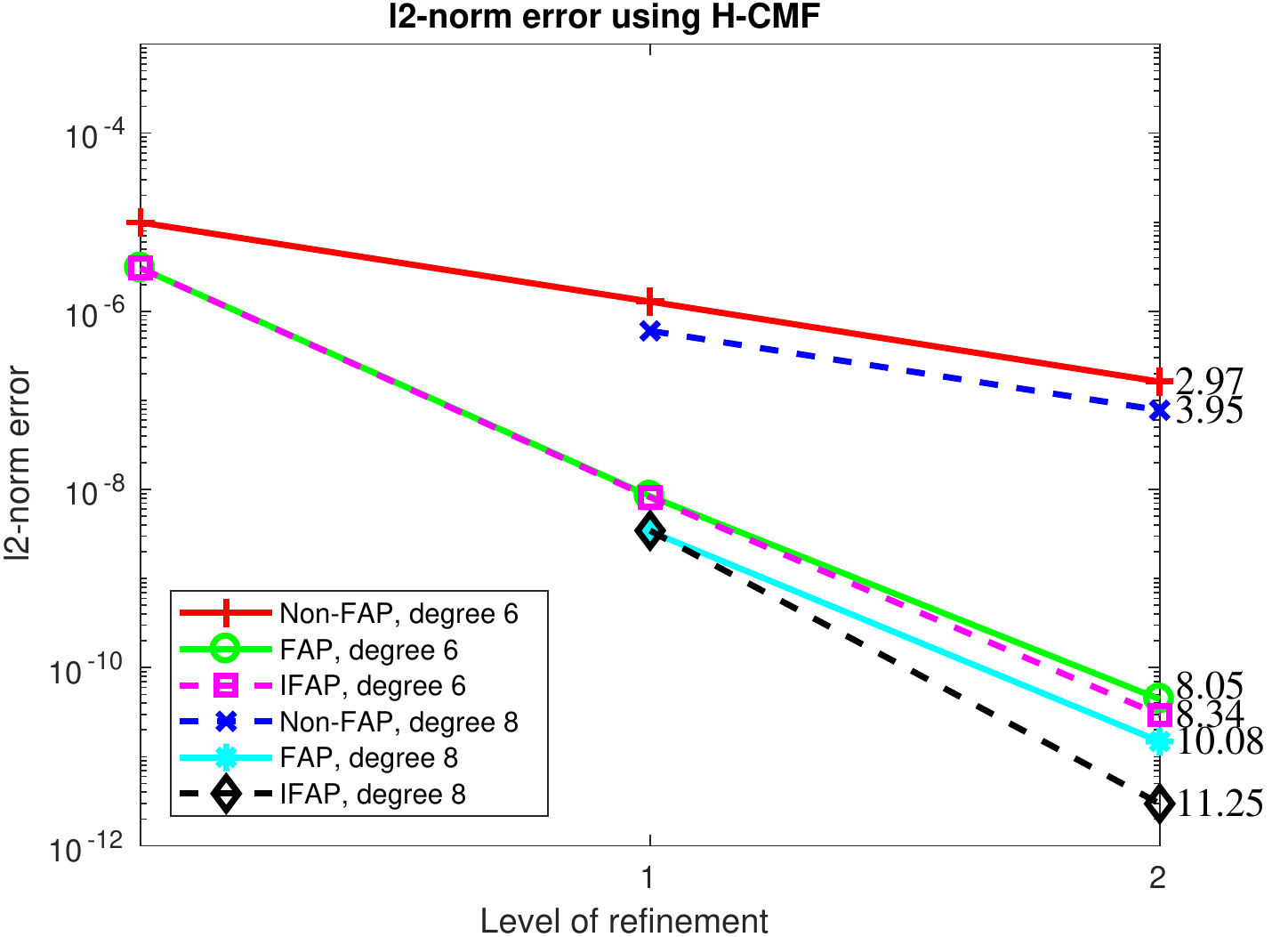}}
\par\end{centering}
\caption{\label{fig:compare-nearfea-treatment}Convergences of CMF with no
Non-FAP, FAP and IFAP for double sphere.}
\end{figure}

\subsection{\label{subsec:FEM}Application to High-Order FEM}

Finally, we demonstrate the application of high-order surface reconstruction
to high-order FEM with curved geometries. It is well known that linear
FEM can deliver second-order convergence rates. High-order FEM uses
higher-degree basis functions to approximate the solutions. However,
these methods may fail to deliver high-order convergence rates if
the curved boundaries are not approximated to at least the same order
of accuracy. Hence, high-order surface reconstruction can play an
important role for these problems.

To demonstrate the effectiveness of high-order surface reconstruction,
we solve the Poisson equation with Dirichlet boundary conditions on
the double-sphere geometry with the analytical solution
\[
u(x,y,z)=e^{(x-0.5)^{2}+y^{2}+z^{2}}.
\]
When the boundary representation is inexact, we set the boundary condition
to the numerical values at the closest points on the exact surface.
This is because the boundary condition are often available only on
the exact surface geometry. To generate the high-order meshes for
the problem, we start with a series of linear tetrahedral mesh with
1,455, 11,640, and 93,120 tetrahedra, respectively, and then add mid-edge,
mid-face, and mid-cell nodes to the linear tetrahedra. The mid-edge
and mid-face nodes on the surfaces are reconstructed using H-CMF with
IFA parameterization. For the tetrahedra incident on the boundary,
we first reconstruct their mid-face and mid-cell nodes also using
IFA parameterization, as mentioned in Section~\ref{subsec:Iterative-Feature-Aware-Paramete}.
We consider the quartic and sextic FEM, which use degree-4 and degree-6
polynomial basis functions, respectively. To isolate the potential
errors in numerical quadrature rule, we used degree-$(2p-2)$ quadrature
rules for degree-$p$ FEM. Figure~\ref{fig:fem-on-double-sphere}
compares the convergence rates of the pointwise errors of interior
nodes in $l_{2}$-norm using piecewise linear boundaries, degree-$p$
H-WALF and H-CMF reconstructed surfaces, and the exact surface. It
can be seen that with piecewise linear boundary, the convergence rates
were limited to second order. For quartic FEM, the results for both
H-WALF and H-CMF are virtually indistinguishable from those using
the exact geometry; we observed the same behavior with quadratic FEM,
whose plots are omitted. For sextic FEM, the results from H-CMF were
the same as using the exact geometry, whereas H-WALF lost some accuracy
due to $\mathcal{O}(h^{6})$ error bound.

\begin{figure}
\subfloat[Quartic FEM.]{\begin{centering}
\includegraphics[width=0.5\textwidth]{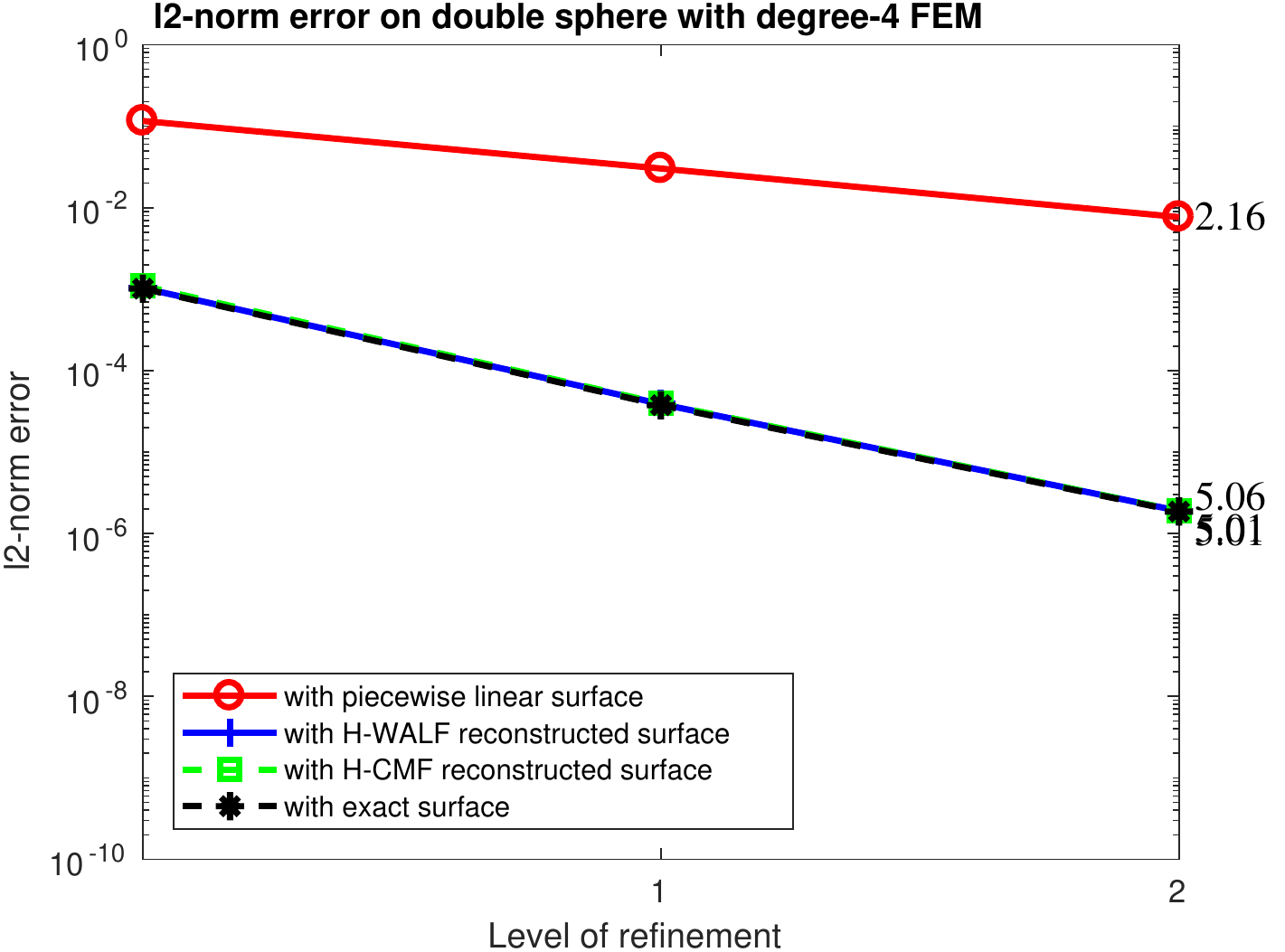}
\par\end{centering}
}\subfloat[Sextic FEM.]{\begin{centering}
\includegraphics[width=0.5\textwidth]{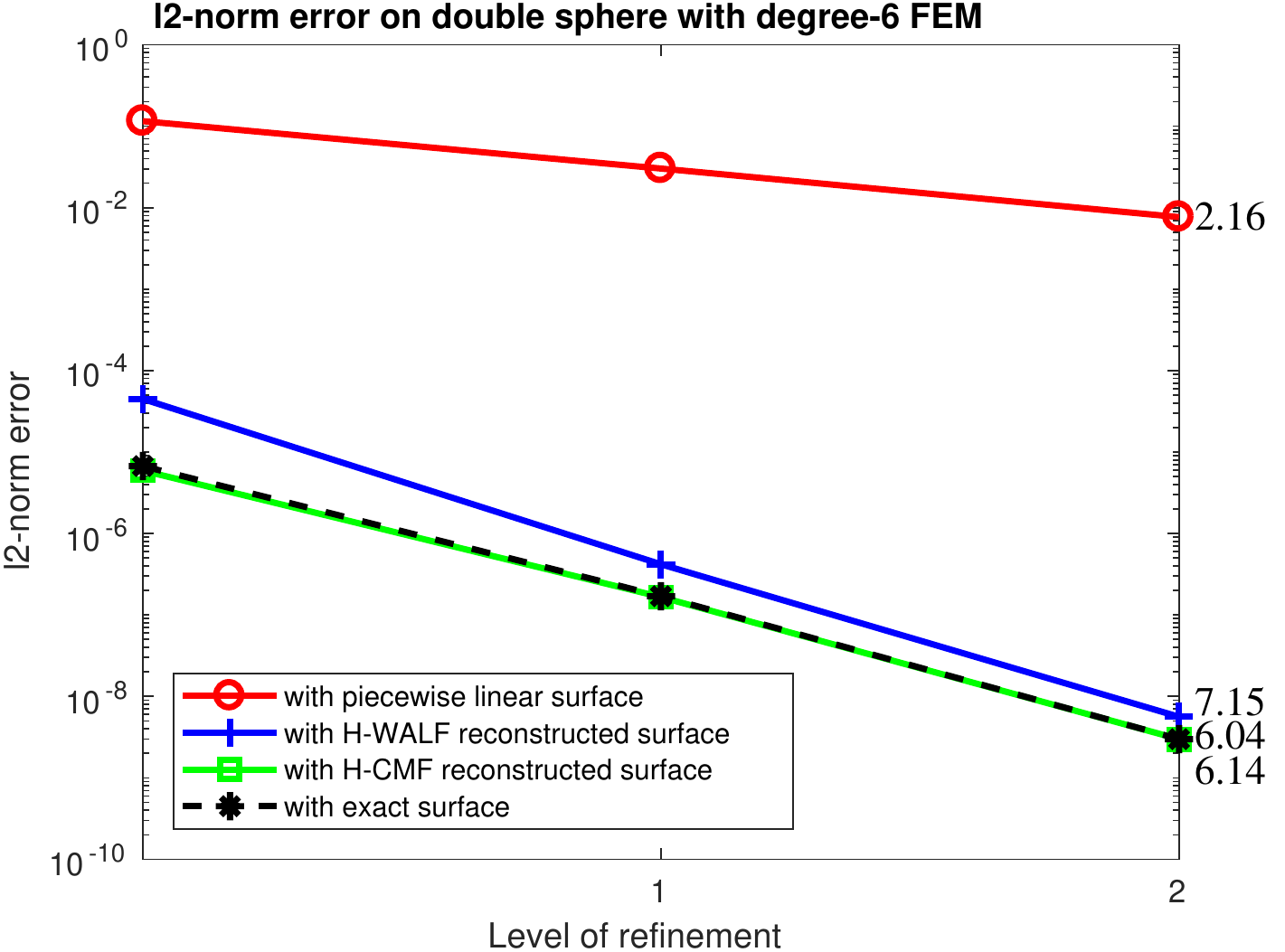}
\par\end{centering}
}

\caption{\label{fig:fem-on-double-sphere} Convergence rates of FEM for solving
Poisson equation with Dirichlet boundary conditions on double-sphere
geometry.}
\end{figure}

For high-order FEM, the feature-aware parameterizations are important
for both the surface and volume meshes. To demonstrate this, Figure~\ref{fig:ifap-FEM}
compares the solutions of quartic and sextic FEM for the same problem
as above with four parameterization strategies:
\begin{itemize}
\item Non-FAP: neither the surface nor the volume elements used FAP;
\item S-FAP: FAP is applied to surface elements but not to volume elements;
\item V-FAP: FAP is applied to volume elements next to the boundary, but
not to surface elements;
\item SV-FAP: FAP is applied to both surface elements and volume elements
next to the boundary.
\end{itemize}
It can be seen that V-FAP and SV-FAP improved accuracy significantly
compared to Nno-FAP and S-FAP. This indicates that FAP for the volume
elements has the most impact on FEM solutions. This impact is expected
to be even greater if the boundary is concave or the degree is even
higher. However, the FAP on the surface elements is also important,
especially for degree-six FEM, when used in conjunction of FAP for
the volume elements.

\begin{figure}
\subfloat[Quartic FEM.]{\begin{centering}
\includegraphics[width=0.5\textwidth]{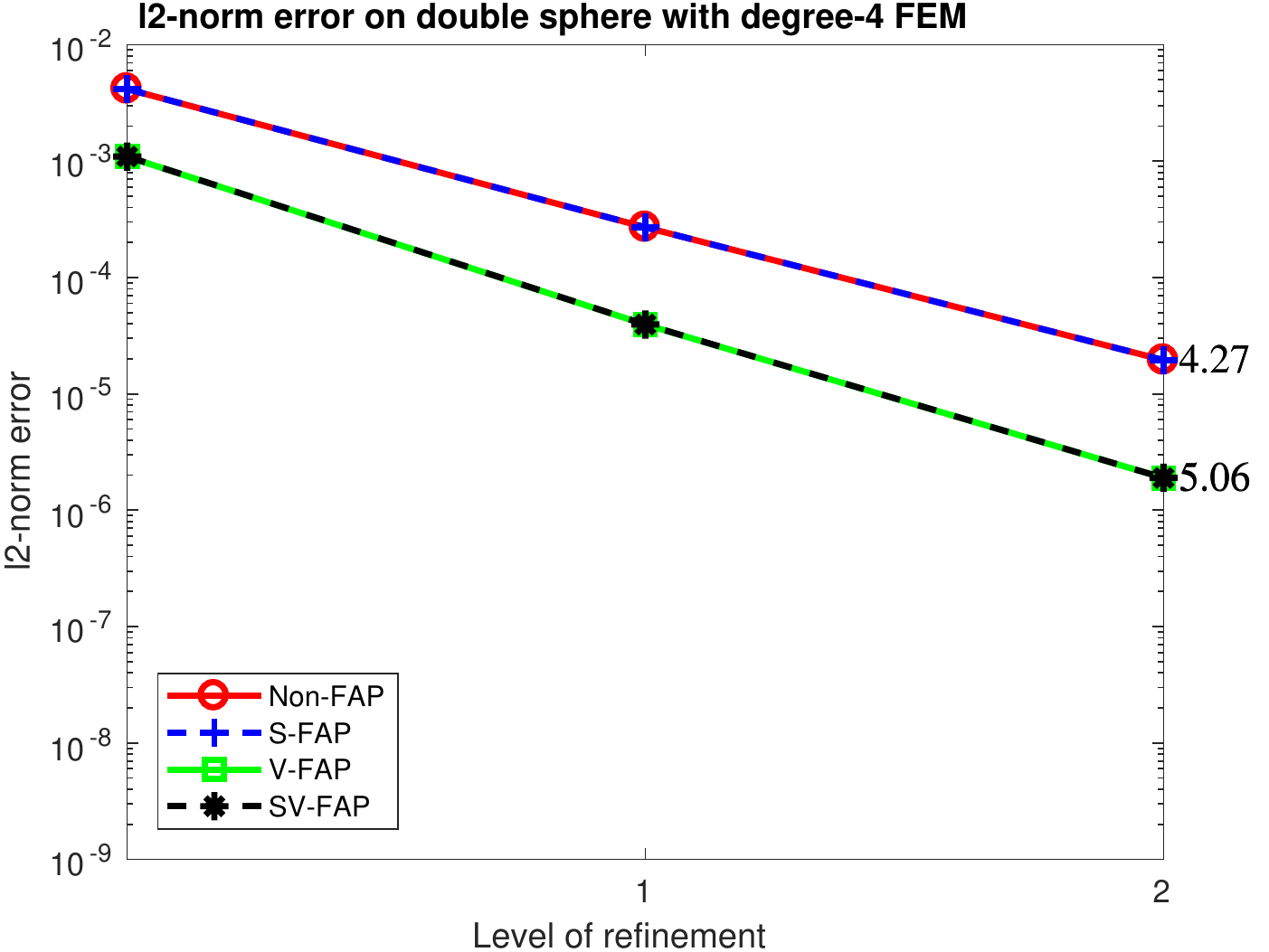}
\par\end{centering}
}\subfloat[Sextic FEM.]{\begin{centering}
\includegraphics[width=0.5\textwidth]{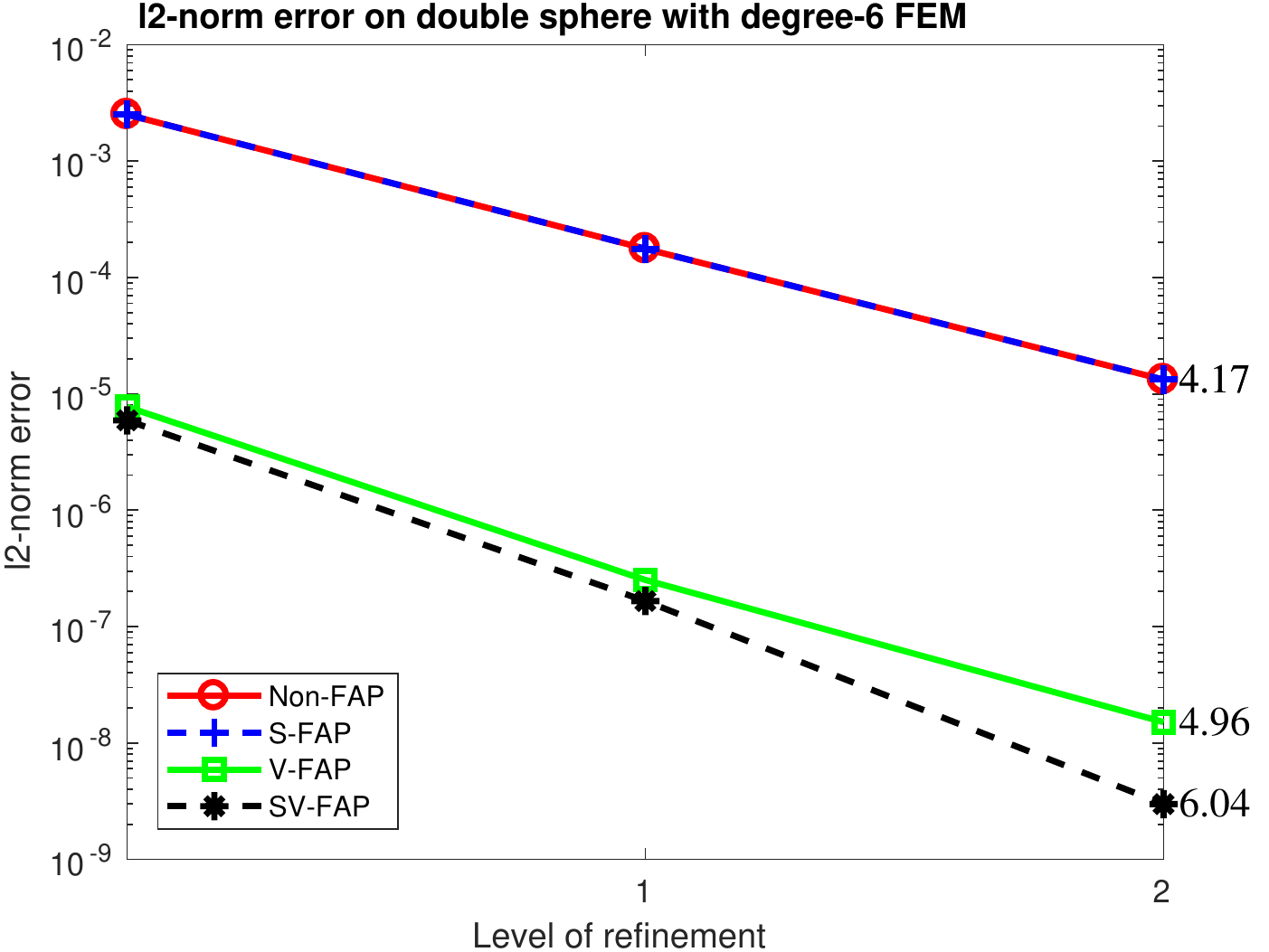}
\par\end{centering}
}

\caption{\label{fig:ifap-FEM} Demonstration of FAP for volume elements near
boundaries in FEM with curved boundaries.}
\end{figure}

\section{\label{sec:Conclusions}Conclusions}

In this paper, we considered the problem of high-order reconstruction
of a piecewise smooth surface from its surface triangulation. This
is important for meshing, geometric modeling, finite element methods,
etc. We introduced two Hermite-style surface reconstructions, called
H-CMF and H-WALF, which extended the point-based CMF and WALF in \cite{Jiao2011RHO}
by taking into account the normals in addition to the points of the
input mesh. In addition, we introduced an iterative feature-aware
(IFA) parameterization for elements near sharp features and boundaries,
which allows us to construct $G^{0}$ continuous parametric surfaces
with guaranteed $(p+1)$st order accuracy for piecewise smooth surfaces.
In addition, we also showed that with even-degree polynomials, the
reconstructions can superconverge at about $(p+2)$nd order. We assessed
the accuracy and stability of these techniques both through theoretical
analysis and numerical experimentations. In terms of applications,
we demonstrated that our high-order reconstructions enabled virtually
indistinguishable results as using exact geometry for high-order FEM.
This shows that our method provides a valuable tool for high-order
FEM, either as an alternative to using exact CAD models when they
are inconvenient to use (such as on supercomputers), or potentially
as the only viable option if the CAD models are unavailable.

Our techniques only enforce $G^{0}$ continuity of the reconstructed
surfaces, which is sufficient for most applications in terms of accuracy
and stability of the numerical approximations, given that the local
parameterizations are smooth. In comparison, some other techniques,
such as NURBS, T-splines, moving least squares, etc., aim for $G^{1}$,
$G^{2}$, or even $G^{\infty}$ continuities, which, however, have
no direct correlation with the accuracy of numerical approximations.
However, our proposed techniques are by no means replacements of traditional
CAD techniques for all applications, as they can complement each other
in different contexts. An important integration of the CAD models
and the proposed Hermite-style reconstruction is the extraction of
the surface normal and the feature curve tangents. Another extension
of this work is to adapt the proposed techniques for high-order reconstructions
of functions on surfaces, which are important for high-order data
transfer across meshes in multiphysics simulations, as well as the
high-order imposition of Neumann boundary conditions on curved geometries
for some variants of finite element methods.

\section*{Acknowledgements}

This work was supported in part under the SciDAC program in the US
Department of Energy Office of Science, Office of Advanced Scientific
Computing Research  through subcontract \#462974 with Los Alamos National
Laboratory and  under a subcontract with Argonne National Laboratory
under Contract DE-AC02-06CH11357.

Assigned: LA-UR-19-20389. LANL is operated by Triad National Security,
LLC, for the National Nuclear Security Administration of the U.S.
DOE.

\bibliographystyle{abbrv}
\bibliography{refs/Uref,refs/refs,refs/hisurf}

\end{document}